\documentclass[pdflatex,sn-mathphys]{sn-jnl}

\usepackage[utf8]{inputenc} 
\usepackage[T1]{fontenc}    
\usepackage{hyperref}       
\usepackage{url}            
\usepackage{booktabs}       
\usepackage{algorithm,algorithmicx}
\usepackage{amsfonts}       
\usepackage{nicefrac}       
\usepackage{microtype}      
\usepackage{xcolor}         

\usepackage[numbers]{natbib}

\usepackage{graphicx}
\usepackage{subfigure}
\usepackage{multirow}
\usepackage{amsmath}
\usepackage{amsfonts,amssymb}
\usepackage{amsthm}

\newcommand\verB[1]{#1}

\renewcommand{\R}{\mathbb R}
\renewcommand{|}{\vert}

\DeclareMathOperator*{\argmin}{arg\,min}

\jyear{2022}%

\theoremstyle{thmstyleone}%
\newtheorem{theorem}{Theorem}
\theoremstyle{thmstyletwo}%
\newtheorem{remark}{Remark}%
\newtheorem{scenario}{Scenario}

\theoremstyle{thmstylethree}%

\begin{document}

\title[MAD: A Meta-Learning Based ROM for Solving Parametric PDEs]{Meta-Auto-Decoder: A Meta-Learning Based Reduced Order Model for Solving Parametric Partial Differential Equations}


\author[1]{\fnm{Zhanhong} \sur{Ye}}\email{yezhanhong@pku.edu.cn}
\equalcont{These authors contributed equally to this work.}

\author[2]{\fnm{Xiang} \sur{Huang}}\email{sahx@mail.ustc.edu.cn}
\equalcont{These authors contributed equally to this work.}

\author[3]{\fnm{Hongsheng} \sur{Liu}}\email{liuhongsheng4@huawei.com}
\author*[1,4]{\fnm{Bin} \sur{Dong}}\email{dongbin@math.pku.edu.cn}

\affil*[1]{\orgdiv{Beijing International Center for Mathematical Research}, \orgname{Peking University}, \orgaddress{\city{Beijing}, \country{China}}}

\affil[2]{\orgdiv{School of Computer Science and Technology}, \orgname{University of Science and Technology of China}, \orgaddress{\city{Hefei}, \country{China}}}

\affil[3]{\orgdiv{Central Software Institute}, \orgname{Huawei Technologies Co. Ltd}, \orgaddress{\city{Hangzhou}, \country{China}}}

\affil[4]{\orgdiv{Center for Machine Learning Research}, \orgname{Peking University}, \orgaddress{\city{Beijing}, \country{China}}}


\abstract{
Many important problems in science and engineering require solving the so-called parametric partial differential equations (PDEs), i.e., PDEs with different physical parameters, boundary conditions, shapes of computational domains, etc. 
Typical reduced order modeling techniques accelarate solution of the parametric PDEs by projecting them onto a linear trial manifold constructed in the offline stage. 
These methods often need a predefined mesh as well as a series of precomputed solution snapshots, and
may struggle to balance between efficiency and accuracy due to the limitation of the linear ansatz. 
Utilizing the nonlinear representation of neural networks, 
we propose Meta-Auto-Decoder (MAD) to construct a nonlinear trial manifold, 
whose best possible performance is measured theoretically by the decoder width. 
Based on the meta-learning concept, the trial manifold can be learned in a mesh-free and unsupervised way during the pre-training stage. 
Fast adaptation to new (possibly heterogeneous) PDE parameters is enabled by searching on this trial manifold, and optionally fine-tuning the trial manifold at the same time. 
Extensive numerical experiments show that the MAD method exhibits faster convergence speed without losing accuracy than other deep learning-based methods.
}




\maketitle

\section{Introduction}\label{sec:intro}

\noindent
Many important problems in science and engineering, such as inverse problems, control and optimization, risk assessment, and uncertainty quantification~\cite{cohen2015approximation,khoo2021solving}, require solving the so-called parametric partial differential equations (PDEs), i.e., PDEs with different physical parameters, boundary conditions, or solution regions.
Mathematically, they require to solve the so-called \textit{parametric} PDEs that can be formulated as:
\begin{equation}\label{def:PDE}
	\mathcal{L}_{\widetilde{\pmb{x}}}^{\gamma_1} u = 0, \ {\widetilde{\pmb{x}}} \in \Omega \subset \R^d,\qquad
	\mathcal{B}_{\widetilde{\pmb{x}}}^{\gamma_2} u = 0, \ {\widetilde{\pmb{x}}} \in \partial \Omega
\end{equation}
where $\mathcal{L}^{\gamma_1}$ and $\mathcal{B}^{\gamma_2}$ are partial differential operators parametrized by $\gamma_1$ and $\gamma_2$, respectively, and ${\widetilde{\pmb{x}}}$ denotes the independent variable in spatiotemporal-dependent PDEs. 
Given $\mathcal{U} = \mathcal{U}(\Omega;\R^{d_u})$ and the space of parameters $\mathcal{A}$, $\eta=(\gamma_1,\gamma_2,\Omega) \in \mathcal{A}$ is the variable parameter of the PDEs and $u \in \mathcal{U}$ is the solution of the PDEs. 
Note that the form of $\eta$ considered here is very general with possible heterogeneity allowed, since the computational domain shape $\Omega$ and the functions defined on this domain or its boundary (which may be involved in $\gamma_1,\gamma_2$) are obviously of different types. 
Solving parametric PDEs requires to construct an infinite-dimensional operator $G:\mathcal{A} \to \mathcal{U}$ that maps any PDE parameter $\eta$ to its corresponding solution $u^\eta$ (i.e., the solution mapping). 

\subsection{Reduced Order Modeling}
Solvers such as Finite Element Methods (FEM)~\cite{zienkiewicz1977finite} and Finite Difference Methods (FDM)~\cite{liszka1980finite} would first discretize the parametric PDEs on a certain mesh, and then solve the corresponding (parametric) algebraic system. 
A fine-resolution mesh is often used to obtain a solution with acceptable accuracy, and the resulting high-dimensional algebraic system would then lead to prohibitive computational costs. 
One common solution to this issue is to replace the original algebraic system by a reduced order model (ROM) without losing much of the physical features~\cite{Quarteroni2015ReducedBM,Benner2017ModelRA}. 
Being relatively lower-dimensional, the ROM can be solved faster, providing a reasonable approximate solution to the parametric PDEs. 

Typical reduced order modeling techniques feature an offline-online solving strategy. 
In the offline stage, a series of the solutions to the original discretized algebraic system (with different parameters) are generated. 
These solutions are used to construct a suitable trial manifold embedded in the high-dimensional space, 
and the ROM is obtained by restricting the original algebraic system to this trial manifold. 
In the online stage, when a new PDE parameter comes, the lower-dimensional ROM for this specific parameter can be easily solved, and an approximate solution is given. 
The offline stage could be time-consuming, but needs to be executed only once. 
After that, the much faster online stage would enable its application to real-time or multi-query scenarios, where solutions corresponding to a large number of PDE parameters need to be solved in a limited amount of time. 
Intuitively speaking, the similarity between the solution snapshots is extracted during the offline stage, informing us of the construction of the trial manifold. 
Utilizing such similarity can reduce the difficulty of the online solving process. 

Validity of the reduced order modeling techniques relies on the assumption 
that the solution set of the parametric PDEs is contained (at least approximately) in a low-dimensional manifold. 
A widespread family of ROM methods further require this manifold to be a linear subspace, and represent the approximate solutions as a linear combination of basis functions. 
These basis functions can be obtained for example by applying proper orthogonal decomposition (POD) to the snapshot matrix. 
However, advection-dominated parametric PDEs such as the wave equation~\cite{Greif2019DecayKN} typically exhibit a slow decaying Kolmogorov $n$-width. 
The dimension of a suitable linear trial manifold for these parametric PDEs often becomes extremely high, even if the true solution manifold is intrinsically low-dimensional. 
Consequently, the efficiency of the online stage to get a solution with acceptable accuracy diminishes. 

Neural networks (NNs) have shown great potential in representing complex nonlinear mappings, and the impact has covered a wide range of fields. 
With this powerful tool at hand, efficient \emph{nonlinear} ROMs can be built to tackle more sophisticated parametric PDEs~\cite{Lee2020ModelRD,Fresca2021AComprehensiveDL,Fresca2022DeepLB}. 
These methods construct the reduced trial manifold by using the decoder mapping of a convolutional autoencoder, whose representation capability is now related to the manifold width~\cite{DeVore1989OptimalNA,Cohen2022OptimalSN} instead of Kolmogorov $n$-width. 
The dimension of such ROMs can often be set equal or nearly equal to the intrinsic dimension of the solution manifold, improving their efficiency. 
In~\cite{Lee2020ModelRD}, the dynamical systems considered are projected onto the nonlinear learned manifold, and the time evolution is computed by minimizing the residual of the equation. 
The reduced dynamics on the nonlinear manifold can be learned by neural networks as well~\cite{Fresca2021AComprehensiveDL}. 
However, these methods still requires discretizing the parametric PDEs on a predefined mesh, restricting their flexibility. 
Furthermore, a lot of high-fidelity solution snapshots have to be precomputed during the offline stage, leading to considerable cost at solving the high-dimensional algebraic equations. 

\subsection{Learning-Based Solvers}
More possibilities of solving parametric PDEs come along with the use of a neural network. 
In recent years, learning-based PDE solvers have become very popular, and it is generally believed that learning-based PDE solvers have the potential to improve efficiency~\cite{raissi2018deep,ChiyuMaxJiang2020MeshfreeFlowNetAP,DmitriiKochkov2021MachineLA}.
The learning-based PDE solvers can be categorized into two categories in terms of the objects that are approximated by neural networks, i.e., the approximation of the solution $u^\eta$ and the approximation of the solution mapping $G$.

\bmhead{NN as a new ansatz of solution}
This kind of approaches represent the approximate solution of the PDEs with a neural network, and a discretization mesh is no longer required. 
Without the discretized algebraic equations, they mainly rely on the governing equations and boundary conditions (or their variants) to train the neural networks. 
For example, Physics-Informed Neural Networks (PINNs)~\cite{raissi2018deep} and Deep Galerkin Method (DGM)~\cite{sirignano2018dgm} constrain the output of deep neural networks to satisfy the given governing equations and boundary conditions. 
Deep Ritz Method (DRM)~\cite{weinan2018deep} exploits the variational form of PDEs and can be used to solve PDEs that can be reformulated as equivalent energy minimization problems. 
Based on a weak formulation of PDEs, Weak Adversarial Network (WAN)~\cite{zang2020weak} parameterizes the weak solution and test functions as primal and adversarial neural networks, respectively. 
These neural approximation methods can work in an unsupervised manner, without the need to generate solution snapshots from conventional computational methods as the labeled data. 
However, all these methods contains no offline stage. They treat different PDE parameters as independent tasks, and need to retrain the neural network from scratch for each PDE parameter. 
When a large number of tasks with different PDE parameters need to be solved, these methods are computationally expensive and impractical.
In order to mitigate retraining cost, E and Yu~\cite{weinan2018deep} recommended a transfer learning method that uses a model trained for one task as the initial model to train another task. 
However, according to our experiments, transfer learning is not always effective in improving convergence speed (see Sec.\ref{sec:burgers}, \ref{sec:laplace}). 

\bmhead{NN as a new ansatz of solution mapping}
This kind of approaches use neural networks to learn the solution mapping between two infinite-dimensional function spaces~\cite{long2018pde,long2019pde,lu2019deeponet, bhattacharya2020model, li2020fourier}.
For example, PDE-Nets~\cite{long2018pde,long2019pde} are among the earliest neural operators that are specifically designed convolutional neural networks inspired by finite difference approximations of PDEs. 
They are able to uncover hidden PDE models from observed dynamical data and perform fast and accurate  predictions at the same time.
Deep Operator Network (DeepONet)~\cite{lu2019deeponet} uses two subnets to encode the parameters and location variables of the PDEs separately, and merge them together to compute the solution.
Fourier Neural Operator (FNO)~\cite{li2020fourier} utilizes fast Fourier transform to build the neural operator architecture and learn the solution mapping between two infinite-dimensional function spaces.
A significant advantage of these approaches is that once the neural network is trained, the online prediction time is almost negligible. 
Although they have demonstrated promising results across a wide range of applications, several issues occur. 
First, the data acquisition cost is prohibitive in complex physical, biological, or engineering systems, and the generalization ability of these models is poor when there is not enough labeled data during the offline training stage~\cite{cai2021physics}.
Second, most of these methods~\cite{long2018pde,long2019pde,bhattacharya2020model,li2020fourier} require a predefined mesh like typical ROMs, and utilize the labeled data on the mesh for training and inference. 
Third, simply applying one forward inference in the online stage may lead to unsatisfactory generalization, especially on out-of-distribution (OOD) settings (i.e., PDE parameters for training and inference are from different probability distributions).
Finally, these operators directly takes the PDE parameter $\eta$ as network input, which would bring inconvenience in network implementation if $\eta$ is heterogeneous.
The recently proposed Physics-Informed DeepONet (PI-DeepONet)~\cite{wang2021learning} can learn a mesh-free solution mapping without any labeled data and retraining.
However, it needs to collect a large number of training samples in the parameter space $\mathcal{A}$ to obtain an acceptable accuracy (see Sec.\ref{sec:burgers}), and is still inflexible dealing with heterogeneous PDE parameters. 

\bmhead{Meta-learning}
Different from conventional machine learning that learns to do a given task, meta-learning learns to improve the learning algorithm itself based on multiple learning episodes over a distribution of related tasks.
As a result, meta-learning can handle new tasks faster and better. 
In this field, the Model-Agnostic Meta-Learning (MAML)~\cite{finn2017model} algorithm and its variants~\cite{antoniou2019train, nichol2018reptile, yoon2018bayesian} have beed widely used. 
These algorithms try to find an initial model with good generalization ability such that it can be adapted to new tasks with a small number of gradient updates.  
For example, MAML~\cite{finn2017model} firstly trains a meta-model with good initialization weight on a variety of learning tasks, which is then fine-tuned on a new task through a few steps of gradient descent to get the target model.
The Reptile~\cite{nichol2018reptile} algorithm eliminates second-order derivatives in MAML by repeatedly sampling a task, training on it, and moving the initialization towards the trained weight on that task.

Borrowing the idea of meta-learning may inspire new ways to solve parametric PDEs. 
Different PDE parameters are viewed as different tasks, and the similarity between them is extracted to form an efficient learning algorithm, enabling fast solving when a new parameter comes. 
This strategy to accelerate parametric PDE solving has much in common with the offline-online formalism of the ROMs. 
There are still differences, since it finds the solution via learning rather than solving algebraic equations, and the precomputed solution snapshots are no longer required in the offline stage. 
To the best of our knowledge, Meta-MgNet~\cite{chen2022meta} is the first work that view solving parametric PDEs as a meta-learning problem, which is based on hypernet and the multigrid algorithm. 
Meta-MgNet utilizes the shared knowledge across tasks to generate good smoothing operators adaptively, and thereby accelerates the solution process, but is not directly applicable to PDEs on which the multigrid algorithm is not available. 
Recently, the Reptile algorithm is also used to accelerate the PDE solving problems~\cite{liu2021novel}.
However, MAML and Reptile are not always effective in improving the convergence speed (see Sec.\ref{sec:maxwell},~\ref{sec:laplace} and~\ref{sec:helmholtz}). 

\subsection{Our Contributions}
We propose Meta-Auto-Decoder (MAD), a mesh-free and unsupervised deep learning method that enables the pre-trained model to be quickly adapted to equation instances by implicitly encoding heterogeneous PDE parameters as latent vectors.
MAD makes use of the similarity between tasks from the perspective of manifold learning, and tries to learn an approximation of the solution manifold during the offline pre-training stage. 
Thanks to the expressive power of neural networks, this approximated solution manifold can be highly nonlinear, and it lies in an \emph{infinite-dimensional} function space, rather than a discretized finite-dimensional space as in conventional ROMs. 
We construct the ansatz of solution as a neural network in the form $u_\theta({\widetilde{\pmb{x}}},\pmb{z})$. 
By taking the spatial (or spatial-temporal) coordinate ${\widetilde{\pmb{x}}}$ directly as the network input, unsupervised training is allowed, and a mesh is no longer required. 
As the additional input $\pmb{z}$ varies, $u_\theta({\widetilde{\pmb{x}}},\pmb{z})$ moves on 
the trial manifold, 
which may be an approximation of the true solution manifold for certain $\theta$. 
The PDE parameter $\eta$ is implicitly encoded into $\pmb{z}$ by applying the auto-decoder architecture motivated by \cite{park2019deepsdf}, regardless of the possible heterogeneity. 
When a new task comes in the online stage, MAD achieves fast transfer by projecting the new task to the manifold and optionally fine-tuning the manifold at the same time.
The main contributions of this paper are summarized as follows: 
\begin{itemize}
\item A mesh-free and unsupervised deep neural network approach is proposed to solve parametric PDEs. Based on the concept of meta-learning, once the neural network is pre-trained, solving a new task involves only a small number of iterations. In addition, the auto-decoder architecture adopted by MAD can realize auto-encoding of heterogeneous PDE parameters.
\item The mathematical intuition behind the MAD method is analyzed from the perspective of manifold learning. In short, a neural network is pre-trained to approximate the solution manifold by a trial manifold, and the required solution is searched on the trial manifold or in a neighborhood of the trial manifold.
\item To quantify the best possible approximation performance of such a trial manifold, we introduce the decoder width, which provides a theoretical tool to analyze the effectiveness of the MAD method. 
\item Extensive numerical experiments are carried out to demonstrate the effectiveness of our method, which show that MAD can significantly improve the convergence speed and has good extrapolation ability for OOD settings.
\end{itemize}
A preliminary version of this work appeared as~\cite{huang2021metaautodecoder}, while we included more extended discussions on the relations between MAD and the reduced order modeling techniques, as well as some theoretical analysis and additional experiments. 

\section{Widths for Quantifying Approximation Accuracy}\label{sec:widths}

Both conventional ROMs and the MAD method aim to find a suitable trial manifold in the offline or pre-training stage. 
A trial manifold of low dimensionality is prefered, as it can simplify parametric PDE solving in the online or fine-tuning stage. 
In the mean time, the set of solutions should be contained in the trial manifold, at least in the approximate sense, so that a decent numerical solution can be obtained from the trial manifold. 
Taking the intrinsic dimension of the solution set into account, 
we have to decide whether it is possible to get a desired accuracy using a trial manifold with certain dimensionality. 
The concept of widths is then introduce to give a quantified criterion for this purpose. 
For different types of trial manifolds, different definitions of widths are considered accordingly. 

\subsection{The Kolmogorov $n$-Width and the Manifold Width}
We first consider the parametric PDEs on a fixed domain $\Omega$, and defer the case of a variable domain to Sec.\ref{sec:widthsVarDom}. 
Let $\mathcal{U}=\mathcal{U}(\Omega;\R^{d_u})$ be a Banach space, and the solution set (or solution manifold) to be approximated be a compact subset $\mathcal{K}=G(\mathcal{A})=\{G(\eta)\mid\eta\in\mathcal{A}\}\subset\mathcal{U}$. 
Many classical reduced order modeling techniques approximate the true solution $u^\eta\in\mathcal{K}$ using a linear combination of basis functions $\hat u=\sum_{i=1}^{n}a_iu_i$, and the underlying trial manifold is in fact linear. 
The best possible approximation of elements of $\mathcal{K}$ using such linear trial manifolds can be measured by the well-known Kolmogorov $n$-width. 
To be more specific, for a given linear subspace $U_n\subset\mathcal{U}$ of dimension $n$, its performance in approximating the elements of $\mathcal{K}$ is evaluated by the worst case error
\begin{equation}
\sup_{u\in\mathcal{K}}d_{\mathcal{U}}(u,U_n) = \sup_{\eta\in\mathcal{A}}\inf_{v\in U_n}\|u^\eta-v\|_{\mathcal{U}} .
\end{equation}
The Kolmogorov $n$-width is then defined by taking the infimum over all possible linear subspaces
\begin{equation}
d_n(\mathcal{K})=\inf_{U_n}\sup_{\eta\in\mathcal{A}}\inf_{v\in U_n}\|u^\eta-v\|_{\mathcal{U}} .
\end{equation}

For certain types of parametric PDEs, such as the diffusion-dominated equations or the elliptic ones~\cite{Cohen2010AnalyticRP,Tran2017AnalysisQO}, the set of solutions $\mathcal{K}$ is known to have a fast decaying Kolmogorov $n$-width as $n$ increases. 
In this case, a linear trial manifold with a reasonable number of dimensions would provide a good approximation of $\mathcal{K}$. 
However, the decay rate is found to be quite slow for advection-dominated parametric PDEs, including the wave equation~\cite{Greif2019DecayKN}. 
This encourages the study of nonlinear trial manifolds and the corresponding nonlinear widths. 

Making use of the nonlinear representation of neural networks, many recent methods~\cite{Lee2020ModelRD,Fresca2021AComprehensiveDL,Fresca2022DeepLB} find the trial manifold by training a convolutional autoencoder. 
The manifold width introduced in~\cite{DeVore1989OptimalNA} is a natural alternative to measure the best possible approximation of this type of trial manifolds. 
Let $Z=\R^n$ be a fixed $n$-dimensional latent space, and $E:\mathcal{U}\to Z$, $D:Z\to\mathcal{U}$ be two continuous mappings, the worst case reconstruction error is
\begin{equation}
\sup_{u\in\mathcal{K}}\|u-D(E(u))\|_{\mathcal{U}}=\sup_{\eta\in\mathcal{A}}\|u^\eta-D(E(u^\eta))\|_{\mathcal{U}} .
\end{equation}
Taking infimum over all possible mapping pairs $(E,D)$ would then give the manifold width
\begin{equation}
d_n^{\mathrm{Mani}}(\mathcal{K})=\inf_{E,D}\sup_{\eta\in\mathcal{A}}\|u^\eta-D(E(u^\eta))\|_{\mathcal{U}} .
\end{equation}
For solution sets of certain elliptic equations, this width may achieve zero even for finite $n$~\cite{Franco2021DeepLA}. 
If we further require these two mappings to be $l$-Lipschitz continuous
, a variant called the stable manifold width~\cite{Cohen2022OptimalSN} is then obtained as%
\footnote{The original definition given in~\cite{Cohen2022OptimalSN} actually takes infimum over all possible norms $\|\cdot\|_Z$ on $Z$ along with $E,D$. Here we choose a fixed norm $\|\cdot\|_Z=\|\cdot\|_2$ for simplicity, which won't make a difference for our purpose. }
\begin{equation}
d_{n,l}^{\mathrm{SMani}}(\mathcal{K})=\inf_{E,D\ l\text{-Lip}}\sup_{\eta\in\mathcal{A}}\|u^\eta-D(E(u^\eta))\|_{\mathcal{U}} .
\end{equation}
Essentially, in these two definitions of nonlinear widths, the solution set $\mathcal{K}$ is approximated by the trial manifold $D(Z)=\{D(\pmb{z})\mid \pmb{z}\in\R^n\}$, along with certain stability requirements imposed by the continuity of the mapping $E$. 

\subsection{The Decoder Width}
Similar to the autoencoders, the MAD method to be introduced in Sec.\ref{sec:methodology} constructs the trial manifold by using a decoder mapping $D:Z\to\mathcal{U}$. 
However, the encoder mapping $E:\mathcal{U}\to Z$ is no longer involved, and we solely learn the mapping $D$ in a mesh-free and unsupervised manner. 
The corresponding notion of width would change accordingly, and we introduce the \emph{decoder width} given as
\begin{equation}\label{eq:decWidth}\begin{split}
	d_{n,l}^\mathrm{Deco}(\mathcal{K})&=\inf_{D\ l\text{-Lip}}\sup_{u\in\mathcal{K}}d_{\mathcal{U}}(u,\{D(\pmb{z})\mid\|\pmb{z}\|\le 1\})
	\\&=\inf_{D\ l\text{-Lip}}\sup_{\eta\in\mathcal{A}}\inf_{\pmb{z}\in Z_B}\|u^\eta-D(\pmb{z})\|_\mathcal{U}
,\end{split}\end{equation}
where the first infimum is taken over all mappings $D:Z\to\mathcal{U}$ that are $l$-Lipschitz continuous, and $Z_B=\{\pmb{z}\in Z\mid\|\pmb{z}\|\le 1\}$ is the closed unit ball of $Z$. 
The constraint $\|\pmb{z}\|\le 1$ and the Lipschitz continuity condition are introduced to avoid highly irregular mappings resembling the space-filling curves. 
The following theorem implies that the decoder width decays at least as fast as the stable manifold width as $n$ increases. 
\begin{theorem}
	Assume the solution set $\mathcal{K}$ has radius $R=\sup_{u\in\mathcal{K}}\|u\|_{\mathcal{U}}<+\infty$. Then we have
	\begin{equation}
d_{n,l^2R}^\mathrm{Deco}(\mathcal{K})\le d_{n,l}^\mathrm{SMani}(\mathcal{K}) .
\end{equation}
\end{theorem}
\begin{proof}
	Let $E:\mathcal{U}\to Z$ and $D:Z\to\mathcal{U}$ be two $l$-Lipschitz continuous mappings. 
	We define a new mapping $\bar D:Z\to\mathcal{U}$ by $\bar D(\pmb{z})=D(LR\pmb{z}+E(\mathbf{0}_{\mathcal{U}}))$, where $\mathbf{0}_{\mathcal{U}}$ is the zero vector of $\mathcal{U}$. 
	Then $\bar D$ is $l^2R$-Lipschitz continuous. 
	For given $u\in\mathcal{K}$, $\bar {\pmb{z}}=\frac1{lR}(E(u)-E(\mathbf{0}_{\mathcal{U}}))$ satisfies
	\begin{equation}
		\|\bar {\pmb{z}}\|_2=\frac1{lR}\|E(u)-E(\mathbf{0}_{\mathcal{U}})\|_2
		\le\frac1{R}\|u-\mathbf{0}_{\mathcal{U}}\|_{\mathcal{U}}
		\le 1
	,\end{equation}
	and
	\begin{equation}
		\bar D(\bar {\pmb{z}})=D(lR\bar {\pmb{z}}+E(\mathbf{0}_{\mathcal{U}}))
		=D(E(u)-E(\mathbf{0}_{\mathcal{U}})+E(\mathbf{0}_{\mathcal{U}}))=D(E(u))
	.\end{equation}
	This gives
	\begin{equation}
		\inf_{\pmb{z}\in Z_B}\|u-D(\pmb{z})\|_{\mathcal{U}}
		\le\|u-D(\bar {\pmb{z}})\|_{\mathcal{U}}
		=\|u-E(D(u))\|_{\mathcal{U}}
	.\end{equation}
	Taking supremum over all $u\in\mathcal{K}$, and then infimum over all possible mappings $E,D,\bar D$ would then give
	\begin{equation}
	\begin{split}
		d_{n,l^2R}^\mathrm{Deco}(\mathcal{K})&=\inf_{\bar D\ l^2R\text{-Lip}}\sup_{u\in\mathcal{K}}\inf_{\pmb{z}\in Z_B}\|u-\bar D(\pmb{z})\|_\mathcal{U}
		\\&\le\inf_{E,D\ l\text{-Lip}}\sup_{u\in\mathcal{K}}\|u-D(E(u))\|_{\mathcal{U}}
		\\&=d_{n,l}^{\mathrm{SMani}}(\mathcal{K})
	.\end{split}
	\end{equation}
\end{proof}

\subsection{Dealing with Variable Domains}\label{sec:widthsVarDom}
The shape of the computational domain $\Omega$ may be part of the variable parameter in some parametric PDEs, 
and the corresponding definitions of widths can be introduced by using either a reference domain or a master domain. 

In the first case, we assume there is a \emph{reference domain} $\Omega^\text{ref}$, and every possible computational domain $\Omega$ is associated with a diffeomorphism $T_\Omega:\Omega^\text{ref}\to\Omega$. 
Then for any PDE parameter $\eta=(\gamma_1,\gamma_2,\Omega)\in\mathcal{A}$, the true solution $u^\eta(\widetilde{\pmb{x}})\in\mathcal{U}(\Omega;\R^{d_u})$ can be associated with a function on the reference domain
\begin{equation}
\bar{u}^\eta(\widetilde{\pmb{x}})=u^\eta(T_\Omega(\widetilde{\pmb{x}}))\in\mathcal{U}=\mathcal{U}(\Omega^\text{ref};\R^{d_u}) .
\end{equation}
Taking the modified solution set to be
\begin{equation}
\mathcal{K}=\{\bar{u}^\eta(\widetilde{\pmb{x}})\mid\eta\in\mathcal{A}\}\subset\mathcal{U} ,
\end{equation}
all previous definitions of widths can be applied. 
This reformulation is convenient for classical mesh-based ROMs~\cite{Quarteroni2015ReducedBM}, but it has the prerequisite that no topology changes in $\Omega$ should occur. 

In the second case, we assume instead there is a \emph{master domain} $\Omega^\text{mast}$ containing all possible $\Omega$'s as its subdomain, i.e., $\Omega\subseteq\Omega^\text{mast}$ for any $\eta=(\gamma_1,\gamma_2,\Omega)\in\mathcal{A}$. 
The decoder width is now defined as
\begin{equation}
	d_{n,l}^\mathrm{Deco}(\mathcal{K})=\inf_{D\ l\text{-Lip}}\sup_{\eta=(\gamma_1,\gamma_2,\Omega)\in\mathcal{A}}\inf_{\pmb{z}\in Z_B}
	\Bigl\|u^\eta-D(\pmb{z})\vert_{\Omega}\Bigr\|_{\mathcal{U}(\Omega;\R^{d_u})}
,\end{equation}
where the first infimum is taken over all mappings $D:Z\to\mathcal{U}(\Omega^\text{mast};\R^{d_u})$ that are $l$-Lipschitz continuous. 
Conceptually, it aims to approximate the solution $u^\eta$ by restricting the decoded function $D(\pmb{z})$ on the master domain $\Omega^\text{mast}$ to the subdomain $\Omega\subseteq\Omega^\text{mast}$ on which $u^\eta$ is defined. 
This alternate reformulation is prefered in the context of the MAD method, since we take neural networks as the mesh-free ansatz of PDE solutions, which can be defined for all $\widetilde{\pmb{x}}\in\R^d$. 

\section{Methodology}\label{sec:methodology}

\subsection{Meta-Auto-Decoder}

\begin{figure}
	\centering
	\includegraphics[width=\columnwidth]{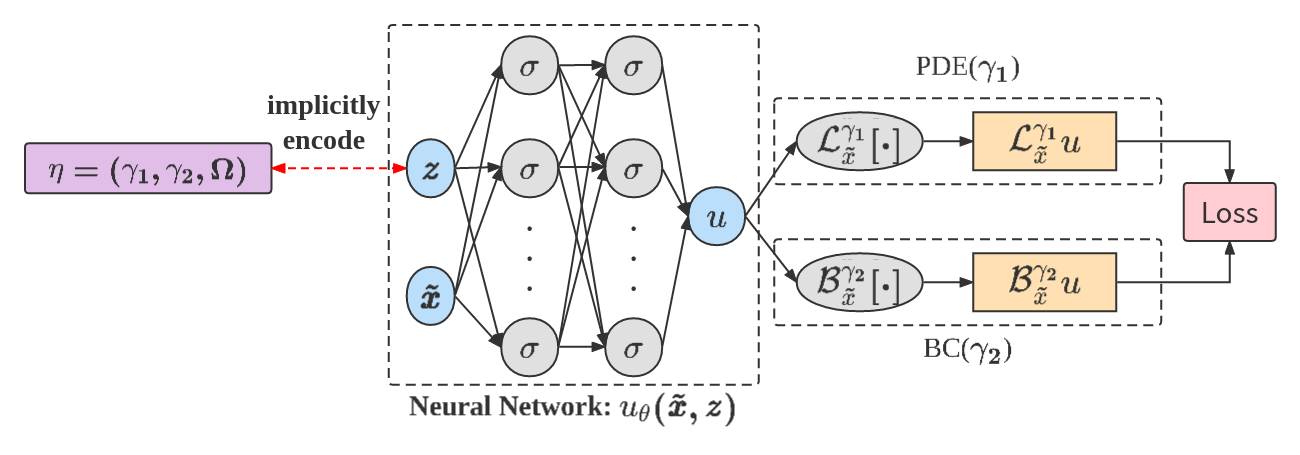}
	\caption{Architecture of Meta-Auto-Decoder.}
	\label{fig:AD_DNN}
\end{figure}

We adopt meta-learning concept to realize fast solution of parametric PDEs. 
Our basic idea is to first learn some universal meta-knowledge from a set of sampled tasks in the pre-training stage, and then solve a new task quickly by combining the task-specific knowledge with the shared meta-knowledge in the fine-tuning stage. 
We also adapt the auto-decoder architecture in~\cite{park2019deepsdf}, 
and introduce $u_\theta({\widetilde{\pmb{x}}},\pmb{z})$ to approximate the solutions of parametric PDEs. 
The architecture of $u_\theta({\widetilde{\pmb{x}}},\pmb{z})$ is shown in Fig.\ref{fig:AD_DNN}.
A physics-informed loss is used for training, making the proposed method unsupervised.
Putting all these together, we propose a new method Meta-Auto-Decoder (MAD) to solve parametric PDEs. 
For the rest of the subsection, the loss function and the two stages of training will be explained in details. 

To enable unsupervised learning, 
given any PDE parameter $\eta\in\mathcal{A}$, the physics-informed loss $L^\eta:\mathcal{U}\to[0,\infty)$ about Eq.\eqref{def:PDE} is taken to be
\begin{equation}\label{eq:PIloss}
	L^\eta[u] = \|\mathcal{L}_{\widetilde{\pmb{x}}}^{\gamma_1} u\|_{L_2(\Omega)}^2 + \lambda_\text{bc}\|\mathcal{B}_{\widetilde{\pmb{x}}}^{\gamma_2} u\|_{L_2(\partial\Omega)}^2
,\end{equation}
where $\lambda_\text{bc}>0$ is a weighting coefficient.
The Monte Carlo estimate of $L^\eta[u]$ is
\begin{equation}\label{eq:MCPIloss}
	\hat L^\eta[u]
	=\frac1{M_\text{r}}\sum_{j=1}^{M_\text{r}}\Bigl\|\mathcal{L}_{\widetilde{\pmb{x}}}^{\gamma_1} u({\widetilde{\pmb{x}}}_j^\text{r})\Bigr\|_2^2 +
	\frac{\lambda_\text{bc}}{M_\text{bc}}\sum_{j=1}^{M_\text{bc}}\Bigl\|\mathcal{B}_{\widetilde{\pmb{x}}}^{\gamma_2} u({\widetilde{\pmb{x}}}_j^\text{bc})\Bigr\|_2^2
,\end{equation}
where $\{{\widetilde{\pmb{x}}}_j^\text{r}\}_{j\in\{1,\dots,M_\text{r}\}}$ and $\{{\widetilde{\pmb{x}}}_j^\text{bc}\}_{j\in\{1,\dots,M_\text{bc}\}}$ are two sets of random points sampled from $\Omega$ and $\partial\Omega$, respectively. 
This task-specific loss $\hat L^\eta[u]$ can be computed by automatic differentiation~\cite{baydin2018automatic},
and will be used in the pre-training stage and the fine-tuing stage. 

In the pre-training stage, through minimizing the loss function, a pre-trained model parametrized by $\theta^*$ is learned for all tasks and each task is paired with its own decoded latent vector $\pmb{z}_i^*$.
Such a pre-trained model is considered as the meta knowledge as it is learned from the distribution of all tasks and the learned latent vector $\pmb{z}_i^*$ is the task-specific knowledge. 
When solving a new task in the fine-tuning stage, we keep the model weight $\theta^*$ fixed and minimize the loss by fine-tuning the latent vector $\pmb{z}$.
Alternatively, we may unfreeze $\theta$ and allow it to be fine-tuned along with $\pmb{z}$. 
These two fine-tuning strategies give rise to different versions of MAD, which are called \textit{MAD-L} and \textit{MAD-LM}, respectively. 
The corresponding problems of pre-training and fine-tuning are formulated as follows:

\bmhead{Pre-training stage}
Given $N$ randomly generated PDE parameters $\eta_1,\dots,\eta_N\in\mathcal{A}$, both \textit{MAD-L} and \textit{MAD-LM} solve the following optimization problem
\begin{equation}\label{eq:MADtr}
	(\{\pmb{z}^*_i\}_{i\in\{1,\dots,N\}},\; \theta^*) = \operatorname*{\arg\min}_{\theta,\{\pmb{z}_i\}_{i\in\{1,\dots,N\}}}\sum_{i=1}^{N}\left(\hat L^{\eta_i}[u_\theta(\cdot,\pmb{z}_i)]+\frac1{\sigma^2}\|\pmb{z}_i\|^2\right),
\end{equation}
where $\theta^*$ is the optimal model weight, $\{\pmb{z}_i^*\}_{i\in\{1,\dots,N\}}$ are the optimal latent vectors for different PDE parameters, and $\hat L^{\eta_i}$ is defined in Eq.\eqref{eq:MCPIloss}.
\verB{See Alg.\ref{alg:pretrain}. }

\bmhead{Fine-tuning stage (\textit{MAD-L})}
Given a new PDE parameter $\eta_{\text{new}}$, 
\textit{MAD-L} keeps $\theta^*$ fixed, and minimizes the following loss function to get
\begin{equation}\label{eq:MADinf}
	\pmb{z}_{\text{new}}^*=\operatorname*{\arg\min}_{\pmb{z}}\hat L^{\eta_{\text{new}}}[u_{\theta^*}(\cdot,\pmb{z})]+\frac1{\sigma^2}\|\pmb{z}\|^2
.\end{equation}
Then $u_{\theta^*}(\cdot,\pmb{z}_{\text{new}}^*)$ is the approximate solution of PDEs with parameter $\eta_{\text{new}}$.
To speed up convergence, we can set the initial value of $\pmb{z}$ to $\pmb{z}_i^*$ obtained during pre-training where $\eta_i$ is the nearest%
\footnote{
For example, if $\mathcal{A}$ is a space of functions, we can discretize a function into a vector and then find the Euclidean distance between the two vectors as the distance between two PDE parameters.}
to $\eta_{\text{new}}$.
\verB{See Alg.\ref{alg:finetune}. }

\bmhead{Fine-tuning stage (\textit{MAD-LM})}
\textit{MAD-LM} fine-tunes the model weight $\theta$ with the latent vector $\pmb{z}$ simultaneously, and solves the following optimization problem
\begin{equation}\label{eq:MADinf_LM}
	(\pmb{z}_{\text{new}}^*,\theta_{\text{new}}^*)=\operatorname*{\arg\min}_{\pmb{z},\theta}\hat L^{\eta_{\text{new}}}[u_{\theta}(\cdot,\pmb{z})]+\frac1{\sigma^2}\|\pmb{z}\|^2
\end{equation}
with initial model weight $\theta^*$. 
This would produce an alternative approximate solution $u_{\theta_{\text{new}}^*}(\cdot,\pmb{z}_{\text{new}}^*)$. 
The latent vector is initialized in the same way as \textit{MAD-L}. 
\verB{See Alg.\ref{alg:finetune}. }
\begin{remark}
	The MAD method has several key advantages compared with existing methods. 
	Besides being mesh-free and unsupervised, it can deal with heterogeneous PDE parameters painlessly, since $\eta$ is not taken as the network input, and is encoded into $\pmb{z}$ in an implicit way. 
	Introduction of the meta-knowledge $\theta^*$ would accelerate the fine-tuning process, which can be better understood in the light of the manifold learning perspective. 
	For \textit{MAD-LM}, the accuracy on OOD tasks is likely to be at least comparable with training from scratch based on PINNs. 
	Although the fine-tuning process of MAD is still slower than one forward inference of a neural network solution mapping, the advantages presented above can make it more suitable for some real applications. 
\end{remark}

\begin{remark}
	If we replace the physics-informed loss by certain supervised loss, the \textit{MAD-L} method would then coincide with the DeepSDF algorithm~\cite{park2019deepsdf}. 
	Despite of this, the field of solving parametric PDEs is quite different from 3D shape representation in computer graphics. 
	Moreover, the introduction of model weight fine-tuning in \textit{MAD-LM} can significantly improve solution accuracy, as is explained intuitively in Sec.\ref{sec:MAD_L},\ref{sec:MAD_LM} and validated by numerical experiments in Sec.\ref{sec:numerical_experiments}. 
\end{remark}

\begin{algorithm}
\caption{Pre-training stage of MAD}\label{alg:pretrain}
\begin{algorithmic}[1]
	\Require PDE parameter samples $\{\eta_i\}_{1\le i\le N}$, learning rates $\alpha_\theta,\alpha_z$
	\State randomly initialize $\theta,\{\pmb{z}_i\}$
	\While{stopping-criterion not met}
		\State $L(\theta,\{\pmb{z}_i\}) \leftarrow \sum_{i=1}^{N}\bigl(\hat L^{\eta_i}[u_\theta(\cdot,\pmb{z}_i)]+\frac1{\sigma^2}\|\pmb{z}_i\|^2\bigr)$
		\State $\theta\leftarrow\theta-\alpha_\theta\nabla_\theta L(\theta,\{z_i\})$
		\State $\pmb{z}_i\leftarrow \pmb{z}_i-\alpha_z\nabla_{\pmb{z}_i} L(\theta,\{\pmb{z}_i\})$
	\EndWhile
	\Ensure $\theta^*,\{\pmb{z}_i^*\}$
\end{algorithmic}
\end{algorithm}
\begin{algorithm}
\caption{Fine-tuning stage of MAD}\label{alg:finetune}
\begin{algorithmic}[1]
	\Require new PDE parameter $\eta_\text{new}$, pre-trained $\theta^*,\{\pmb{z}_i^*\}$, learning rates $\alpha_\theta,\alpha_z$
	\State $i\leftarrow\arg\min_i\|\eta_i-\eta_\text{new}\|$
	, $\pmb{z}\leftarrow \pmb{z}_i^*$, $\theta\leftarrow \theta^*$
	\While{stopping-criterion not met}
		\State $L(\theta,\pmb{z}) \leftarrow \hat L^{\eta_{\text{new}}}[u_{\theta}(\cdot,\pmb{z})]+\frac1{\sigma^2}\|\pmb{z}\|^2$
		\State $\pmb{z}\leftarrow \pmb{z}-\alpha_z\nabla_{\pmb{z}} L(\theta,\pmb{z})$
		\If{using \textit{MAD-LM}}
			\State $\theta\leftarrow\theta-\alpha_\theta\nabla_\theta L(\theta,\pmb{z})$
		\EndIf
	\EndWhile
	\Ensure $\pmb{z}_\text{new}^*,\theta_\text{new}^*$, approximate solution $u\approx u_{\theta_{\text{new}}^*}(\cdot,\pmb{z}_{\text{new}}^*)$
\end{algorithmic}
\end{algorithm}

\subsection{Manifold Learning Interpretation of \textit{MAD-L}}\label{sec:MAD_L}

\begin{figure}
	\centering
	\includegraphics[width=0.6\columnwidth]{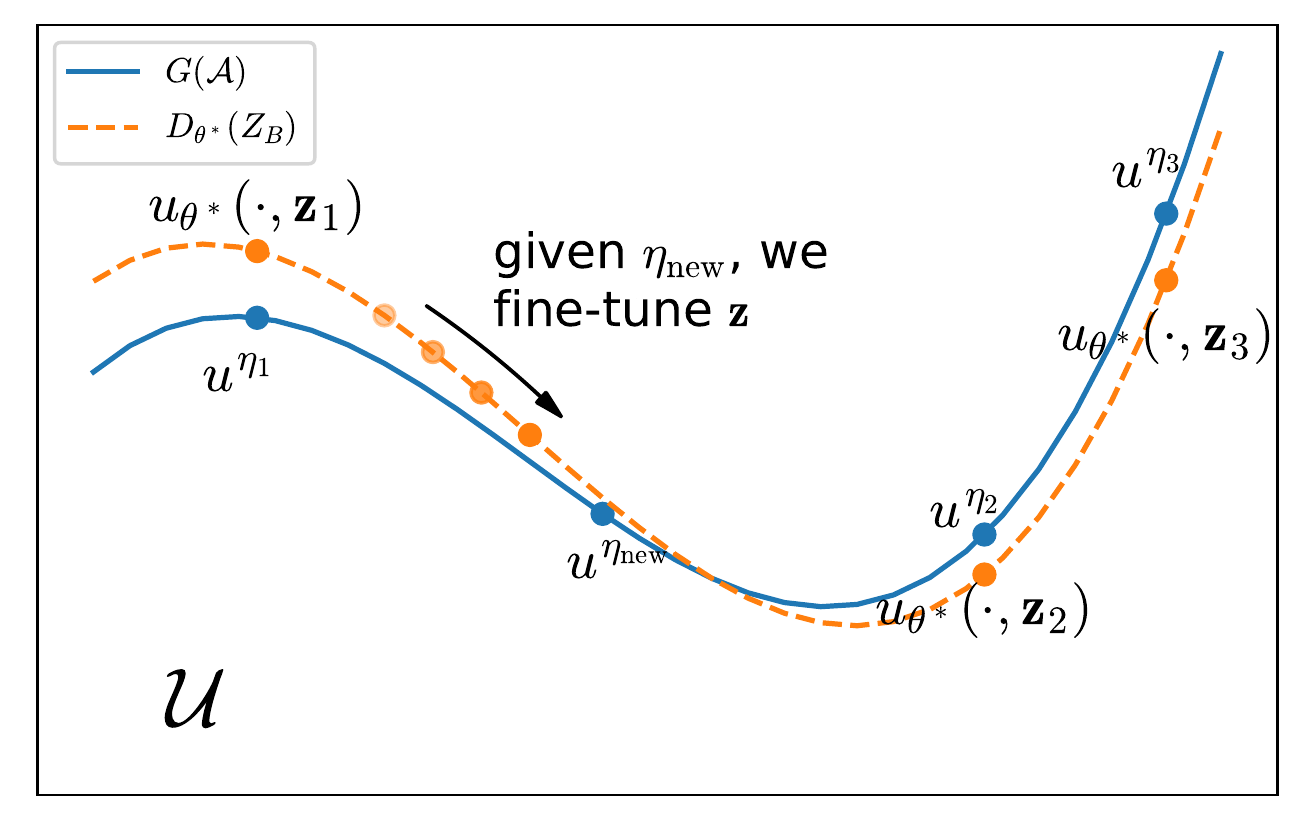}
	\caption{Illustration of how \textit{MAD-L} works from the manifold learning perspective.
	The function space $\mathcal{U}$ is mapped to a 2-dimensional plane.
	The solid curve represents the solution set $G(\mathcal{A})$ formed by exact solutions corresponding to all possible PDE parameters, and each point on the curve represents an exact solution corresponding to one PDE parameter.
	The dotted curve represents the trial manifold $D_{\theta^*}(Z_B)$ obtained by the pre-trained model, and each point on the curve corresponds to a latent vector $\pmb{z}$.
	Given $\eta_{\text{new}} \in \mathcal{A}$, rather than searching in the entire function space $\mathcal{U}$, \textit{MAD-L} only searches on the dotted curve to find an optimal $\pmb{z}$ such that its corresponding solution $u_{\theta^*}(\cdot,\pmb{z})$ is nearest to the point $u^{\eta_{\text{new}}}$.
	}
	\label{fig:MAD}
\end{figure}

We interpret how the \textit{MAD-L} method works from the manifold learning perspective, which also provides a new interpretation of the DeepSDF algorithm~\cite{park2019deepsdf}.
For the rest of this section, the domain $\Omega$ is fixed and excluded from $\eta$ for simplicity, and the general variable domain case can be handled similarly by choosing a master domain as explained in Sec.\ref{sec:widthsVarDom}. 
Now, we consider the following scenario. 
\begin{scenario}\label{as:lowdim}
	The decoder width of the solution set $d_{n,l}^\mathrm{Deco}(\mathcal{K})$, as defined in~\eqref{eq:decWidth}, is zero for certain $n,l<\infty$, and both infimums can be achieved. 
	In other words, there is a $l$-Lipschitz continuous mapping $D:Z\to\mathcal{U}$, and
	\begin{equation}
		\min_{\pmb{z}\in Z_B}\|u^\eta-D(\pmb{z})\|_\mathcal{U} =0
	\end{equation}
	holds for all $\eta\in\mathcal{A}$. 
	This in fact guarantees the relation $\mathcal{K}=G(\mathcal{A})\subseteq D(Z_B)$. 
\end{scenario}
Once the mapping $D$ is found as above, then for a given parameter $\eta$, 
searching for the solution $u^\eta$ in the whole space $\mathcal{U}$ is no longer needed. 
Instead, we may focus on the smaller trial manifold $D(Z_B)$, i.e. the class of functions in $\mathcal{U}$ that is parametrized by $\pmb{z}\in Z_B$, since $u^\eta=G(\eta)\in D(Z_B)$ holds for any $\eta\in\mathcal{A}$. 
We then solve the optimization problem
\begin{equation}\label{eq:trueDecoderSol}
\pmb{z}^\eta=\operatorname*{\argmin}_{\pmb{z}\in Z_B}L^\eta[D(\pmb{z})],
\end{equation} 
and $D(\pmb{z}^\eta)$ is the approximate solution. 

The remaining problem is how to find the mapping $D$. 
We first choose a proper dimension $n$ of the latent space $Z$, 
which is easy if the decoder width of the solution set $d_{n,l}^\mathrm{Deco}(\mathcal{K})$ is theoretically known, 
and can be done through trial and error for more general parametric PDEs. 
Since such a mapping $D$ is usually complex and hard to design by hand, we consider the $\theta$-parametrized
\footnote{Two types of parametrization are considered here. The latent vector $\pmb{z}$ parametrizes a point on the trial manifold $D(Z_B)$ or $D_\theta(Z_B)$, and $\theta$ parametrizes the shape of the entire trial manifold $D_\theta(Z_B)$.}
version $D_\theta:Z\to\mathcal{U}$, and find the best $\theta$ automatically by solving an optimization problem. 
The mapping $D_\theta$ can be constructed in the simple form
\begin{equation}
	D_\theta(\pmb{z})({\widetilde{\pmb{x}}})=u_\theta({\widetilde{\pmb{x}}},\pmb{z})
,\end{equation} 
where $u_\theta$ is a neural network 
that takes the concatenation of ${\widetilde{\pmb{x}}}\in\R^d$ and $\pmb{z}\in\R^n$ as input%
\footnote{Alternatives of the construction of $u_\theta$ exist, and a special one is considered in Sec.\ref{sec:helmholtz}. }. 
The next step is to find the optimal model weight $\theta$ via training, 
which essentially targets at learning a good trial manifold $D_\theta(Z_B)\approx D(Z_B)$. 
Replacing the mapping $D$ in~\eqref{eq:decWidth} by the parametrized version $D_\theta$, the optimization problem to be solved is
\begin{equation}
	\min_\theta\sup_{\eta\in\mathcal{A}}\inf_{\pmb{z}\in Z_B}\|u^\eta-D_\theta(\pmb{z})\|_\mathcal{U}
	=\min_\theta\sup_{\eta\in\mathcal{A}}\inf_{\pmb{z}\in Z_B}\|u^\eta-u_\theta(\cdot,\pmb{z})\|_\mathcal{U}
.\end{equation}
Several adaptations are introduced to derive a simplified problem. 
First, as keeping track of the supremum over $\eta\in\mathcal{A}$ is computationally difficult, it is relaxed to an expectation over $\eta\sim p_{_{\mathcal{A}}}$, where $p_{_{\mathcal{A}}}$ is the probability distribution of the PDE parameters. 
Second, in case we do not have direct access to the exact solutions $u^\eta$, the
physics-informed loss $L^\eta[u_\theta(\cdot,\pmb{z})]$ is used to replace the distance $\|u^\eta-u_\theta(\cdot,\pmb{z})\|_{\mathcal{U}}$. 
Third, the hard constraint $\pmb{z}\in Z_B=\{\pmb{z}\in Z\mid\|\pmb{z}\|\le 1\}$ is dropped, and a relaxed soft constraint is applied by introducing a penalty term $\frac1{\sigma^2}\|\pmb{z}\|^2$. 
We may assume the infimum over $\pmb{z}$ can be attained. 
The optimization problem now becomes
\begin{equation}
	\min_\theta \operatorname*{\mathbb{E}}_{\eta\sim p_{_{\mathcal{A}}}}\left[\min_{\pmb{z}\in Z}\left(L^\eta[u_\theta(\cdot,\pmb{z})]+\frac1{\sigma^2}\|\pmb{z}\|^2\right)\right]
.\end{equation}
In the specific implementation, the expectation on $\eta\sim p_{_{\mathcal{A}}}$ is estimated by Monte Carlo samples $\eta_1,\dots,\eta_N$, and the optimal network weight $\theta$ is taken to be
\begin{equation}\label{eq:rawMADtr}
	\theta^*\approx\operatorname*{\arg\min}_{\theta}\frac1N\sum_{i=1}^{N}\min_{\pmb{z}_i}\left(L^{\eta_i}[u_\theta(\cdot,\pmb{z}_i)]+\frac1{\sigma^2}\|\pmb{z}_i\|^2\right)
.\end{equation}
We further estimate the physics-informed loss $L^\eta$ using Monte Carlo method to obtain Eq.\eqref{eq:MADtr}.
After that, when a new PDE parameter $\eta_{\text{new}} \in \mathcal{A}$ comes, a direct adaptation of Eq.\eqref{eq:trueDecoderSol} would then give rise to the fine-tuning process Eq.\eqref{eq:MADinf}, since $u_{\theta^*}(\cdot,\pmb{z})=D_{\theta^*}(\pmb{z})\approx D(\pmb{z})$ holds. 
An intuitive illustration of how \textit{MAD-L} works from the manifold learning perspective is given in Fig.\ref{fig:MAD}. 

\subsubsection{A Visualization Example}
\begin{figure}
	\centering
	\includegraphics[width=0.7\columnwidth]{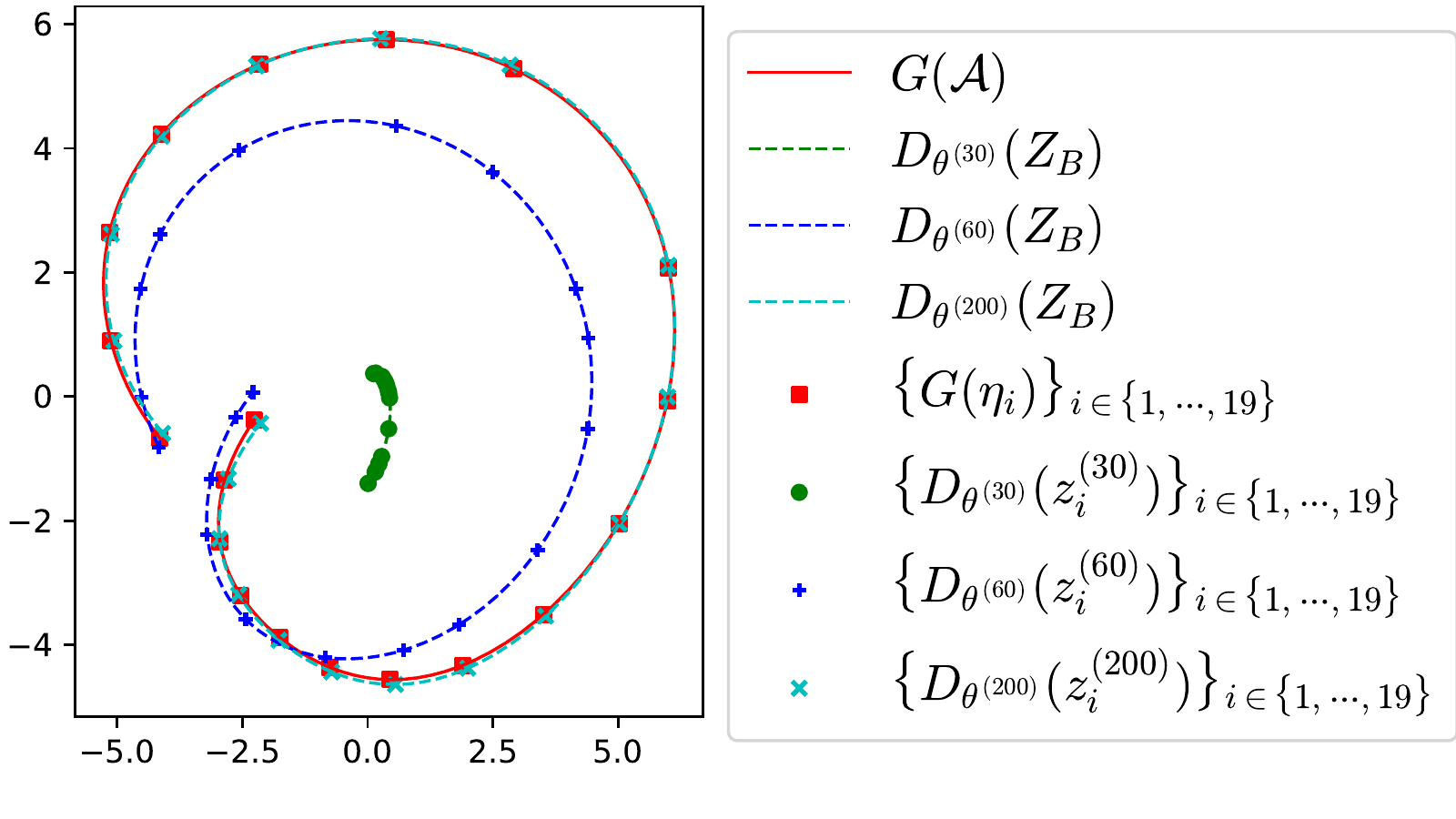}
	\caption{Visualization of the MAD pre-training process for the ODE problem.}%
	\label{fig:MADtr}
\end{figure}
\begin{figure}
	\centering
	\includegraphics[width=0.5\columnwidth]{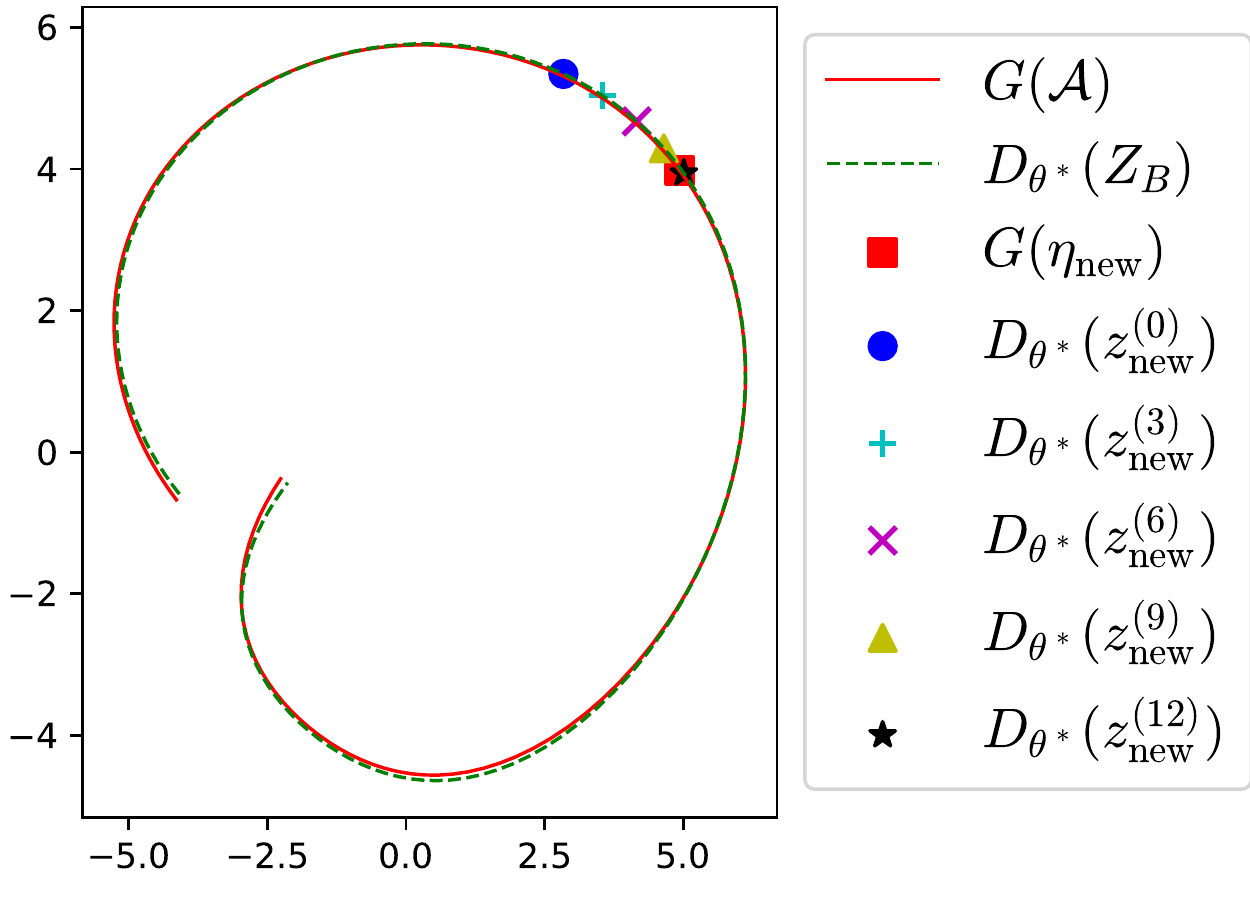}
	\caption{Visualization of the MAD fine-tuning process for the ODE problem.}%
	\label{fig:MADinfL}
\end{figure}

An ordinary differential equation (ODE) is used to visualize the pre-training and fine-tuning processes of \textit{MAD-L}. Consider the following problem with domain $\Omega=(-\pi,\pi)\subset\R$:
\begin{align}
	\frac{\mathrm{d}u}{\mathrm{d}x}&=2(x-\eta)\cos\bigl((x-\eta)^2\bigr),\qquad u(\pm\pi)=\sin\bigl((\pm\pi-\eta)^2\bigr).
\end{align}
We sample 20 points equidistantly on the interval $[0, 2]$ as the variable ODE parameters, and randomly select one $\eta_{\text{new}}$ for the fine-tuning stage and the rest $\{\eta_i\}_{i \in \{1,\cdots,19\}}$ for the pre-training stage.
Taking $Z=\R^1$, 
\textit{MAD-L} generates a sequence of $(\theta^{(m)}, \{z_i^{(m)}\}_{i\in\{1,\cdots,19\}})$ in the pre-training stage, and terminates at $m=200$ with the optimal $(\theta^*, \{z_i^*\}_{i\in\{1,\cdots,19\}})$.
The infinite-dimensional function space $\mathcal{U}=C([-\pi,\pi])$ is projected onto a 2-dimensional plane using Principal Component Analysis (PCA).
Fig.\ref{fig:MADtr} visualizes how $D_\theta(Z_B)$ gradually fits $G(\mathcal{A})$ in the pre-training stage.
The set of exact solutions $G(\mathcal{A})$ forms a 1-dimensional manifold (i.e. the solid curve), and the marked points $\{G(\eta_i)\}_{i \in \{1,\cdots,19\}}$ represent the corresponding ODE parameters used for pre-training.
Each dotted curve represents a trial manifold $D_\theta^{(m)}(Z_B)$ obtained by the neural network at the $m$-th iteration with the points $D_{\theta^{(m)}}(z_i^{(m)})=u_{\theta^{(m)}}(\cdot,z_i^{(m)})$ also marked on the curve.
As the number of iterations $m$ increases, the network weight $\theta=\theta^{(m)}$ updates, making the dotted curves evolve and finally fit the solid curve, i.e., the target manifold $G(\mathcal{A})$.
Fig.\ref{fig:MADinfL} illustrates the fine-tuning process for a given new ODE parameter $\eta_{\text{new}} \in \mathcal{A}$.
As in Fig.\ref{fig:MADtr}, the solid curve represents the set of exact solutions $G(\mathcal{A})$, while the dotted curve represents the trial manifold $D_{\theta^*}(Z_B) = D_{\theta^{(200)}}(Z_B)$ obtained by the pre-trained network.
As $z=z_{\text{new}}^{(m)}$ updates (i.e., through fine-tuning $z$), the marked point $D_{\theta^*}(z_{\text{new}}^{(m)})$ moves on the dotted curve, and finally converges to the approximate solution $D_{\theta^*}(z_{\text{new}}^*) = D_{\theta^*}(z_{\text{new}}^{(12)}) \approx G(\eta_{\text{new}})$.

\subsection{Manifold Learning Interpretation of \textit{MAD-LM}}\label{sec:MAD_LM}
\begin{figure}
	\centering
	\includegraphics[width=0.6\columnwidth]{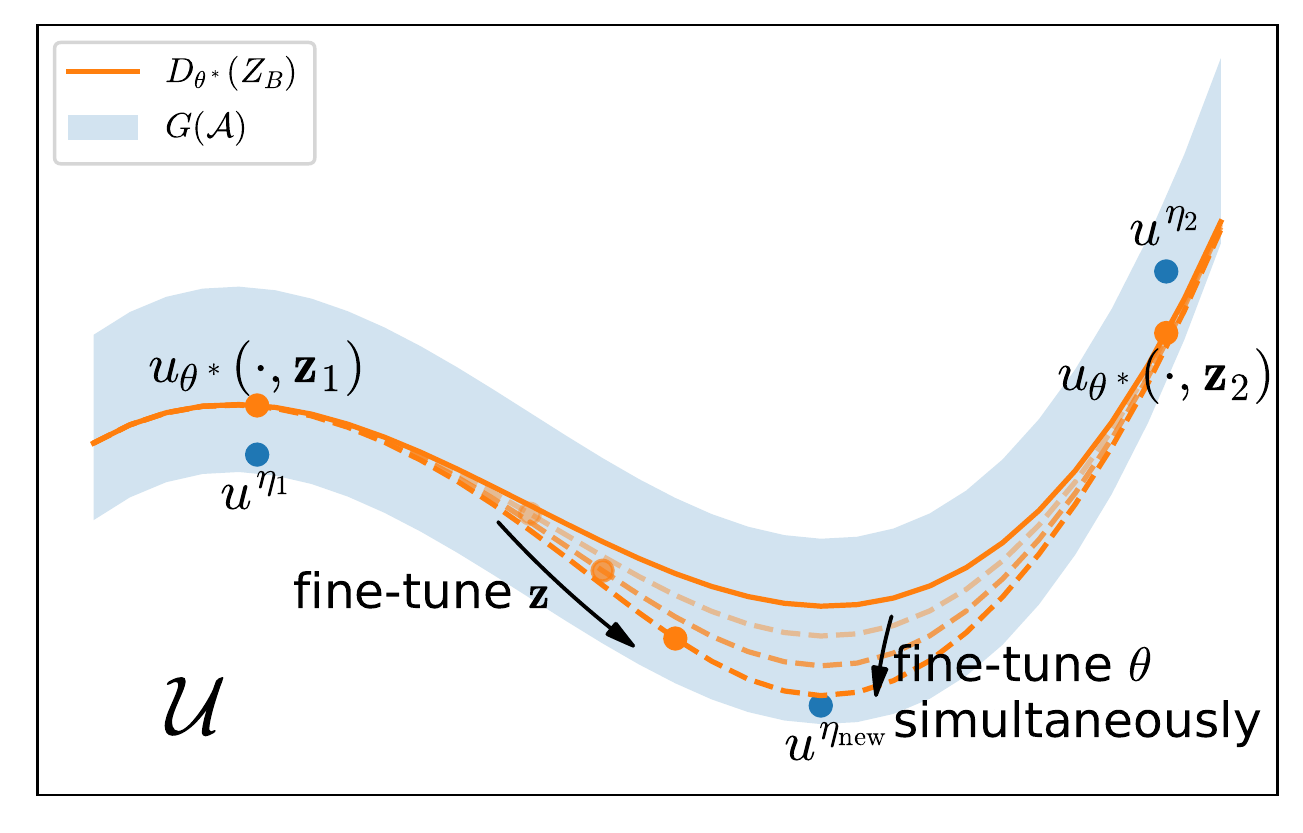}
	\caption{Illustration of how \textit{MAD-LM} works from the manifold learning perspective.
	The trial manifold $D_{\theta^*}(Z_B)$ obtained in the pre-training stage is represented by a solid curve, and the solution set $G(\mathcal{A})$ lies within a neighborhood of $D_{\theta^*}(Z_B)$ that is represented by a gray shadow band. 
	To find the solution $u^{\eta_{\text{new}}}$, we have to fine-tune $\theta$ (i.e., the dotted lines) and the latent vector $\pmb{z}$ (i.e., the points on the dotted lines) simultaneously to approach the exact solution. 
	As the search scope is limited to a strip with a small width, the fine-tuning process 
	can be expected to converge quickly. 
	}
	\label{fig:MADft}
\end{figure}

\textit{MAD-L} is designed for Scenario~\ref{as:lowdim}. 
However, many parametric PDEs encountered in real applications do not fall into this scenario, especially when the parameter set $\mathcal{A}$ of PDEs is an infinite-dimensional function space.
Simply applying \textit{MAD-L} to these PDE solving problems would likely lead to unsatisfactory results. 
\textit{MAD-LM} works in a more general scenario, and thus has the potential of getting improved performance for a wider range of parametric PDE problems. 
This alternative scenario is given as follows. 
\begin{scenario}\label{as:lowdimApprox}
	The decoder width of the solution set $d_{n,l}^\mathrm{Deco}(\mathcal{K})$, as defined in~\eqref{eq:decWidth}, is small for certain $n,l<\infty$. 
	In other words, there is a $l$-Lipschitz continuous mapping $D:Z\to\mathcal{U}$, and
	\begin{equation}
		\inf_{\pmb{z}\in Z_B}\|u^\eta-D(\pmb{z})\|_\mathcal{U} < c
	\end{equation}
	holds for all $\eta\in\mathcal{A}$, where $c$ is a relatively small constant. 
	This makes $\mathcal{K}=G(\mathcal{A})$ to be contained in the $c$-neighborhood of $D(Z_B)\subset\mathcal{U}$. 
\end{scenario}
In this new scenario, similar derivation leads to the same pre-training stage, which is used to find the initial decoder mapping $D_{\theta^*}\approx D$. 
However, in the fine-tuning stage, simply fine-tuning the latent vector $\pmb{z}$ won't give a satisfactory solution in general due to the existence of the $c$-gap. 
Therefore, we have to fine-tune the model weight $\theta$ with the latent vector $\pmb{z}$ simultaneously, and solve the optimization problem Eq.\eqref{eq:MADinf_LM}.
It produces a new decoder $D_{\theta_{\text{new}}^*}$ specific to the parameter $\eta_{\text{new}}$.
An intuitive illustration is given in Fig.\ref{fig:MADft}.

\section{Numerical Experiments}\label{sec:numerical_experiments}
To evaluate the effectiveness of the MAD method, we apply it to solve four parametric PDEs:
(1) Burgers' equation with variable initial conditions;
(2) Maxwell's equations with variable equation coefficients;
(3) Laplace's equation with variable computational domains and boundary conditions (heterogeneous PDE parameters);
and (4) Helmholtz's equation with variable sound speed distributions. 
Accuracy of the model is measured by $average\ relative\ L_2\ error$
(abbreviated as $L_2\ error$) between the predicted solutions and the reference solutions, and we provide the mean value and the 95\% confidence interval of $L_2\ error$.
In each experiment, the PDE parameters are divided into two sets: $S_1$ and $S_2$.
Parameters in $S_1$ correspond to sample tasks for pre-training, and parameters in $S_2$ correspond to new tasks for fine-tuning. 
The methods involved in the comparison are as follows:
\begin{itemize}
	\item \textit{From-Scratch}: Train the model from scratch based on the PINNs method~\cite{raissi2018deep} for all PDE parameters in $S_2$, case-by-case.
	\item \textit{Transfer-Learning}~\cite{weinan2018deep}: Randomly select a PDE parameter in $S_1$ for pre-training based on the PINNs method, and then load the weight obtained in the pre-training stage for PDE parameters in $S_2$ during the fine-tuning stage. 
	\item \textit{MAML}~\cite{finn2017model, antoniou2019train}: Meta-train the model for all PDE parameters in $S_1$ based on the MAML algorithm.
		In the meta-testing stage, we load the pre-trained weight $\theta^*$ and fine-tune the model for each PDE parameter in $S_2$.
	\item \textit{Reptile}~\cite{nichol2018reptile}: Similar to \textit{MAML}, except that the model weight is updated using the Reptile algorithm in the meta-training stage.
	\item \textit{PI-DeepONet}~\cite{wang2021learning}: The model is trained based on the method proposed in~\cite{wang2021learning} for all PDE parameters in $S_1$, and the inference is performed directly for the parameters in $S_2$.
	\item \textit{MAD-L}: Pre-train the model for all PDE parameters in $S_1$ based on our proposed method and then load and freeze the pre-trained weight $\theta^*$ for the second stage.
		In the fine-tuning stage, we choose a $\pmb{z}_i^*$ obtained in the pre-training stage to initialize a latent vector for each PDE parameter in $S_2$, and then fine-tune the latent vector.
		The selection of $\pmb{z}_i^*$ is based on the distance between $\eta_{\text{new}}$ and $\eta_i$.
	\item \textit{MAD-LM}: Different from \textit{MAD-L} that freezes the pre-trained weight, we fine-tune the model weight $\theta$ and the latent vector $\pmb{z}$ simultaneously in the fine-tuning stage.
\end{itemize}

Unless otherwise specified, we shall use the following configurations for the experiments:
\begin{itemize}
	\item In each iteration, we randomly sample $M_\text{r}=8192$ points in $\Omega$, and $M_\text{bc}=1024$ points on $\partial\Omega$. 
	\item To make a fair comparison, the network architectures of all methods (excluding PI-DeepONet) involved in comparison are the same except for the input layer due to the existence of a latent vector. 
		For Burgers' equation, Laplace's equation and Helmholtz's equation, the standard fully-connected neural networks with 7 fully-connected layers and 128 neurons per hidden layer are taken as the default network architecture.
		For Maxwell's equations, the MS-SIREN network architecture~\cite{huang2021solving} with 4 subnets is used, and each subnet has 7 fully-connected layers and 64 neurons per hidden layer.
		It is worth noting that our proposed method has gains in different network architectures, and we choose the network architecture that can achieve high accuracy for the \textit{From-Scratch} method to conduct our comparative experiments.
	\item The network architecture of \textit{PI-DeepONet} used for Burgers' equation is such that both the branch net and the trunk net have 7 fully-connected layers and 128 neurons per hidden layer.
	\item The sine function is used as the activation function, since it exhibits better performance than other alternatives such as the rectified linear unit (ReLU) and the hyperbolic tangent (Tanh)~\cite{sitzmann2020implicit}.
	\item The dimension of the latent vector $\pmb{z}$ is determined by trial and error and set to $128$ for Burgers' equation and Laplace's equation, $16$ for Maxwell's equations, and $64$ for Helmholtz's equation.
	\item The Adam optimizer~\cite{kingma2014adam} is used with the initial learning rate set to be $10^{-3}$ or $10^{-4}$, whichever achieves the better performance. 
		When the training process reaches 40\%, 60\% and 80\%, the learning rate is multiplied by $0.2$ in Helmholtz's equation, and $0.5$ in the other three equations. 
\end{itemize}

\begin{figure}
	\centering
	\includegraphics[width=0.7\columnwidth]{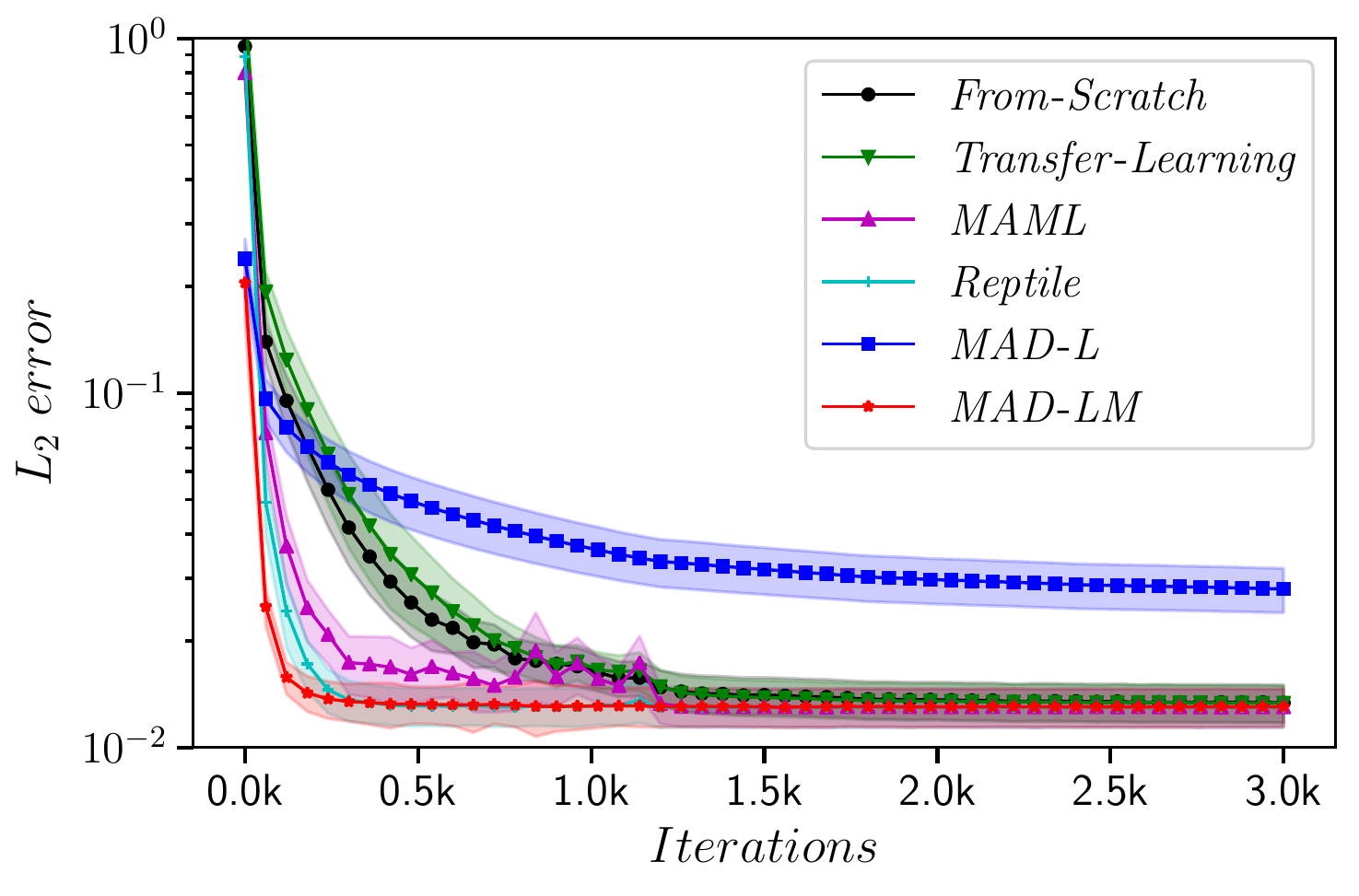}
	\caption{\textbf{Burgers' equation:} The convergence of mean $L_2\ error$ with respect to the number of training iterations.}%
	\label{fig:burgers_compare_other_method}
\end{figure}

\subsection{Burgers' Equation}\label{sec:burgers}

We consider the 1-D Burgers' equation with periodic boundary conditions:
\begin{equation}\label{def:burgers}
\begin{aligned}
	\frac{\partial u}{\partial t} + u \frac{\partial u}{\partial x} &= \nu \frac{\partial ^2 u}{\partial x^2}, \quad x \in (0,1), t \in (0,1],
	\\u(x,0) &= u_0(x), \quad x \in (0,1),
\end{aligned}
\end{equation}
which can model the one-dimensional flow of a viscous fluid. 
Here, $u(x,t)$ is the velocity field to be solved, $\nu=0.01$ is the viscosity coefficient, and the initial condition $u_0(x)$ is taken to be the variable parameter of the PDE, i.e. $\eta=u_0(x)$. 
The initial condition $u_0(x)$ is generated using a Gaussian random field (GRF)~\cite{liu2019advances} according to $u_0(x) \sim \mathcal{N}(0; 100(- \Delta + 9I)^{-3})$ with periodic boundary conditions.
We sample $150$ such initial conditions, and then randomly select $100$ cases to form $S_1$, leaving the rest $50$ for $S_2$. 
To generate the reference solutions, we construct a spatiotemporal mesh of size $1024\times 101$, and solved Eq.\eqref{def:burgers} using a split step method with the code released by~\cite{li2020fourier}. 
The hard constraint on periodic boundary condition is imposed in the neural network architectures as mentioned in~\cite{lu2021physics}.
For \textit{MAD-L} and \textit{MAD-LM}, the pre-training stages run for $50$k iterations while the \textit{Transfer-Learning} pre-trains 3k steps since it only handles one single case. 

Fig.\ref{fig:burgers_compare_other_method} shows the mean $L_2\ error$ of all methods as the number of training iterations increases in the fine-tuning stage. 
All methods converge to nearly the same accuracy (the mean $L_2\ error$ being close to $0.013$) except for \textit{MAD-L}, which could be probably due to the $c$-gap explained in Sec.\ref{sec:MAD_LM}. 
In terms of convergence speed, \textit{From-Scratch} and \textit{Transfer-Learning} need about 1200 iterations to converge, whereas \textit{MAML}, \textit{Reptile} and \textit{MAD-LM} need about 200 iterations to converge. 
\textit{MAD-LM} exhibits the highest convergence speed, requiring only 17\% of the training iterations of \textit{From-Scratch}. 
In this experiment, \textit{Transfer-Learning} does not show any advantage over \textit{From-Scratch}, which means that it fails to obtain any useful knowledge in the pre-training stage.
The model predictions of \textit{MAD-L} and \textit{MAD-LM} compared with the reference solutions for a randomly selected $u_0(x)$ in $S_2$ are shown in Fig.\ref{fig:burgers_result}. 
\textit{MAD-L} provides predictions that are in overall approximate agreement with the reference solutions, but the fit is poor at the spikes and troughs.
In the mean time, the solutions predicted by \textit{MAD-LM} is almost the same as the reference solutions.

\begin{figure}
\begin{center}
\centerline{\includegraphics[width=\columnwidth]{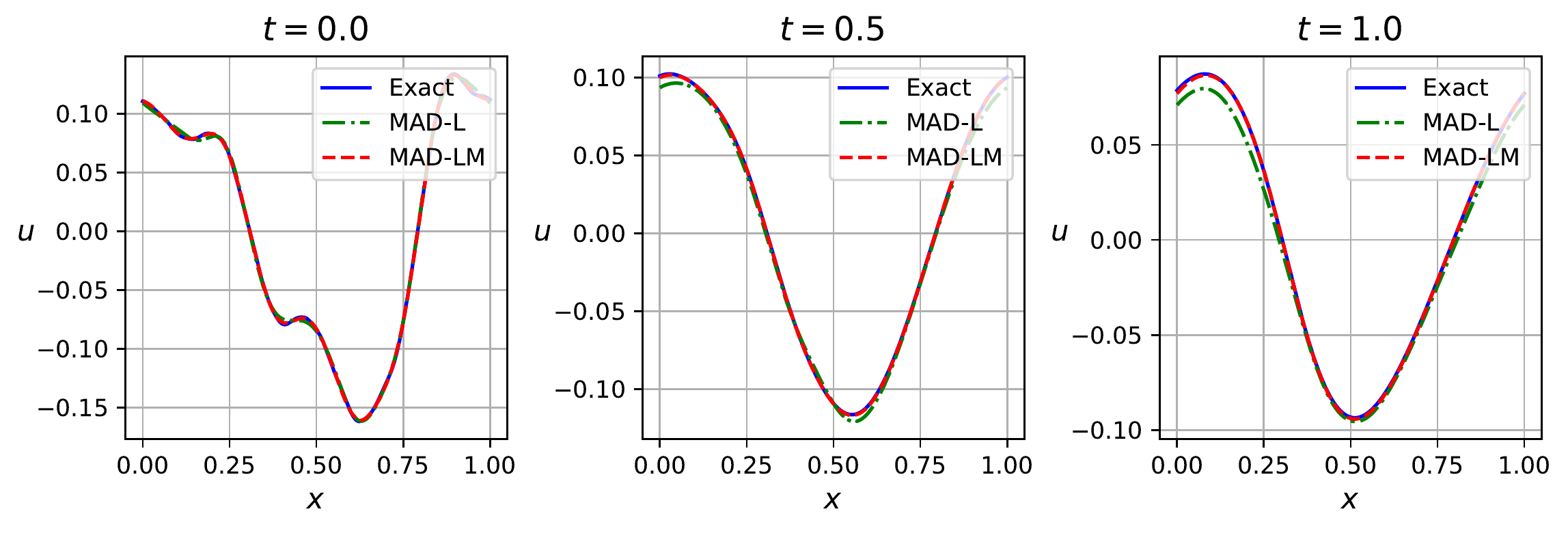}}
\caption{\textbf{Burgers' equation:} Reference solutions vs. model predictions at $t = 0.0$, $t = 0.5$ and $t = 1.0$, respectively.}
\label{fig:burgers_result}
\end{center}
\end{figure}

\begin{table}
\caption{The mean $L_2\ error$ of \textit{PI-DeepONet} and MAD under different numbers of samples in $S_1$.}
\label{tb:ad_vs_pi_deeponet}
\centering
\begin{tabular}{cccc}
	\toprule
	$|S_1|$ & \textit{PI-DeepONet} & \textit{MAD-L} & \textit{MAD-LM} \\
	\midrule
	10  &  0.715  & 0.365 & 0.015 \\
	50  &  0.247  & 0.046 & 0.013 \\
	100 &  0.217  & 0.028 & 0.013 \\
	200 &  0.169  & 0.020 & 0.013 \\
	300 &  0.181  & 0.018 & 0.013 \\
	400 &  0.183  & 0.019 & 0.013 \\
	\bottomrule
\end{tabular}
\vskip -0.2in
\end{table}

\textit{PI-DeepONet} can directly make the inference for unseen PDE parameters in $S_2$, so it has no fine-tuning process.
Table~\ref{tb:ad_vs_pi_deeponet} compares the mean $L_2\ error$ of \textit{PI-DeepONet} and MAD under different numbers of training samples in $S_1$.
The results show that \textit{PI-DeepONet} has a strong dependence on the number of training samples, and its mean $L_2\ error$ is remarkably high when $S_1$ is small. 
Moreover, its mean $L_2\ error$ is significantly higher than that of \textit{MAD-L} or \textit{MAD-LM} in all cases.

\begin{figure}
	\centering
	\includegraphics[width=0.6\columnwidth]{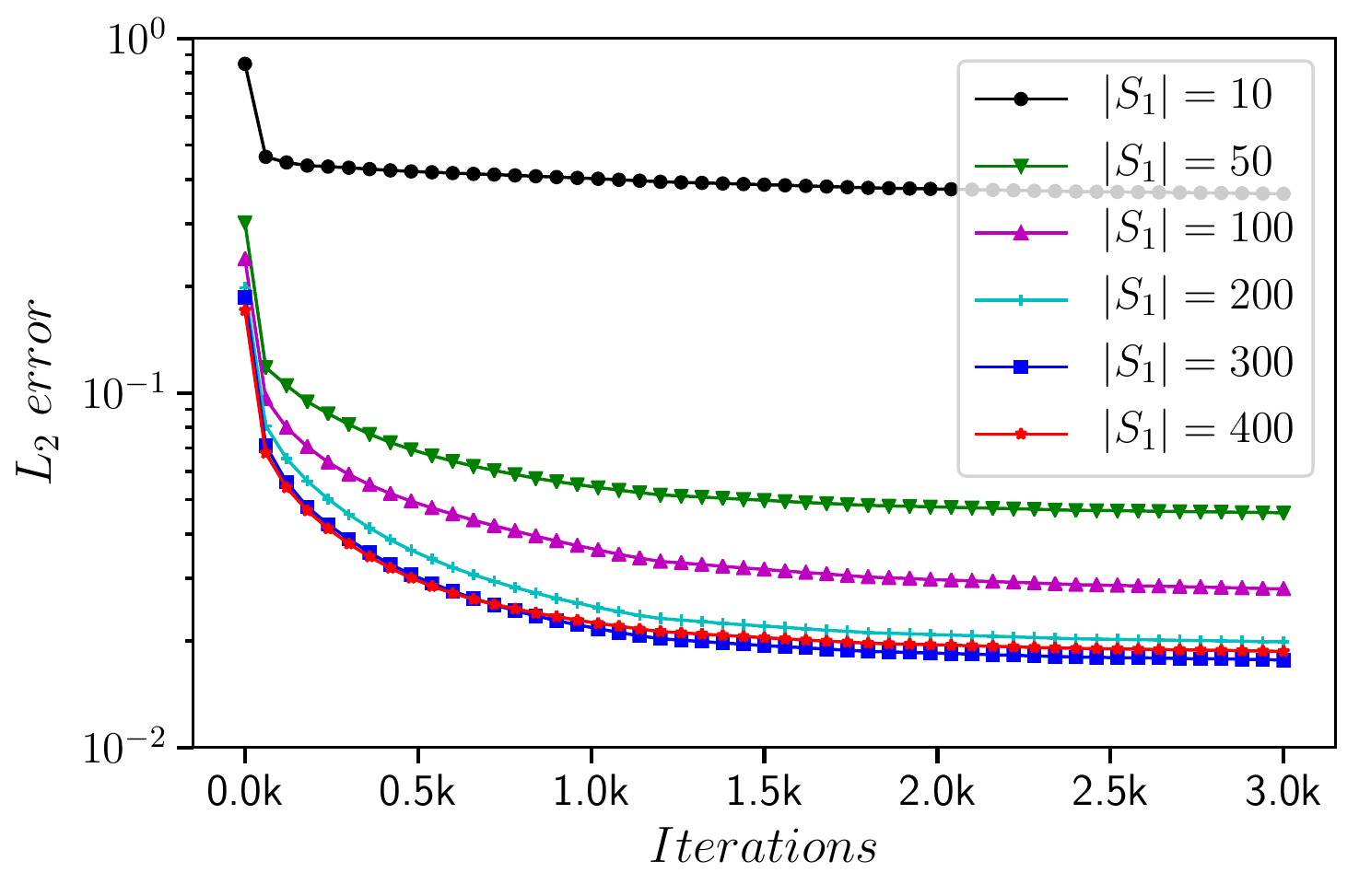}
	\caption{\textbf{Burgers' equation:} The convergence of mean $L_2\ error$ for \textit{MAD-L} with respect to the number of training iterations under different numbers of samples in $S_1$.}%
	\label{fig:burgers_ablation_samples_MAD_l}
\end{figure}
\begin{figure}
	\centering
	\includegraphics[width=0.6\columnwidth]{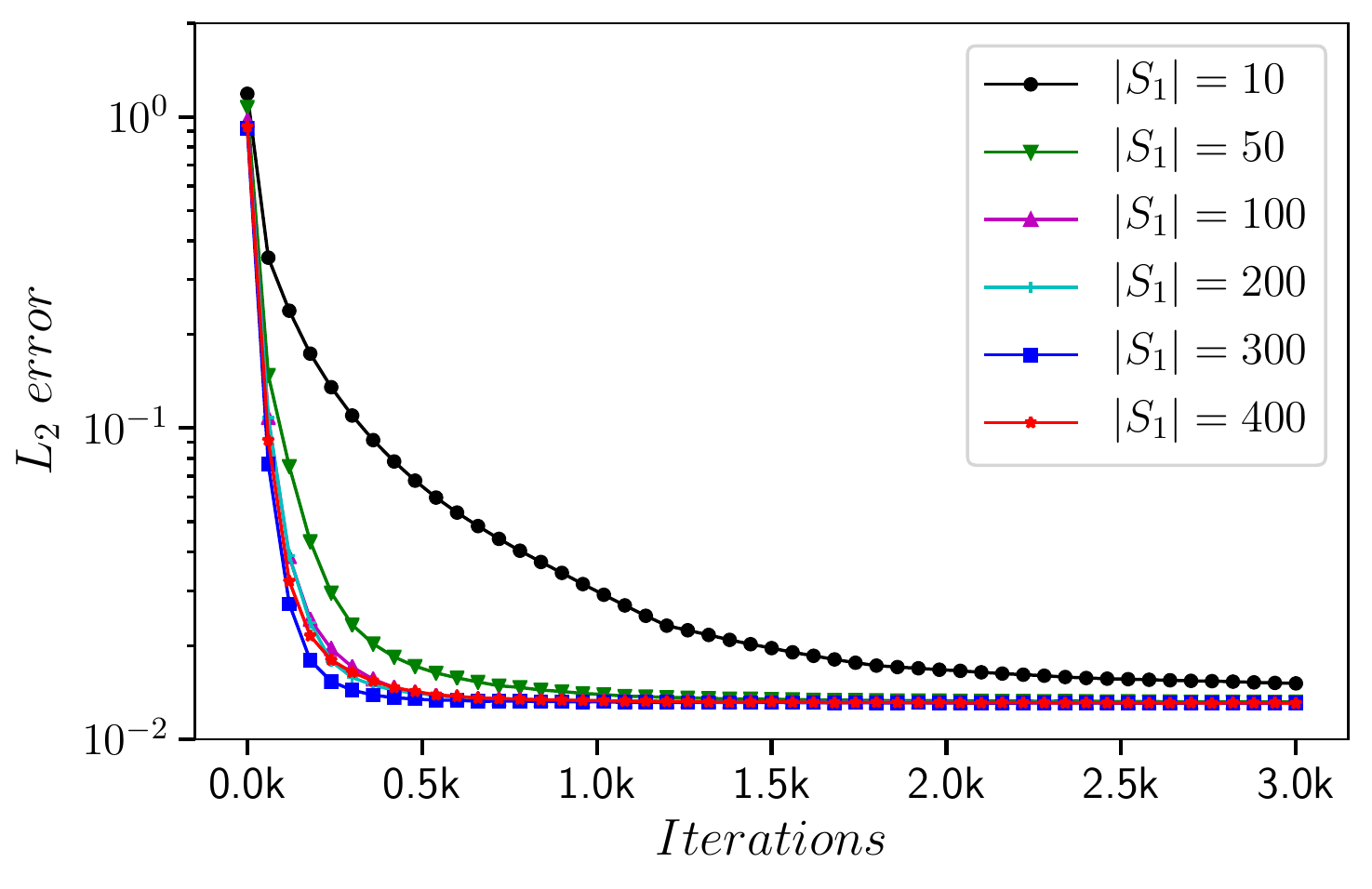}
	\caption{\textbf{Burgers' equation:} The convergence of mean $L_2\ error$ for \textit{MAD-LM} with respect to the number of training iterations under different numbers of samples in $S_1$.}%
	\label{fig:burgers_ablation_samples_MAD_lm}
\end{figure}

\bmhead{Effect of the size of $S_1$}
We investigated the effect of the number of samples $|S_1|$ in the pre-training stage on \textit{MAD-L} and \textit{MAD-LM}. 
Fig.\ref{fig:burgers_ablation_samples_MAD_l} shows that the accuracy of \textit{MAD-L} after convergence increases with $|S_1|$.
However, when $|S_1|$ reaches about 200, further increasing $|S_1|$ does not improve the accuracy of \textit{MAD-L}.
This result is consistent with the intuition shown in Fig.\ref{fig:MADft}.
Increasing the number of samples in the pre-training stage allows the trial manifold $D_{\theta^*}(Z_B)$ to gradually fall within the region formed by the solution set $G(\mathcal{A})$.
After $|S_1|$ reaches a certain level, the trial manifold only swings in the region of the solution set. 
Solely optimizing $\pmb{z}$ can make the predicted solution move inside the trial manifold formed by $D_{\theta^*}(Z_B)$, but  $u^{\eta_{\text{new}}}$ may not be close enough to this manifold.
Therefore, in order to obtain a more accurate solution, we need to fine-tune $\pmb{z}$ and $\theta$ simultaneously.
Fig.\ref{fig:burgers_ablation_samples_MAD_lm} shows that the accuracy and convergence speed of \textit{MAD-LM} do not change significantly with the increase of samples in the pre-training stage. 
It is only when the number of samples is very small (i.e., $|S_1| = 10$) that the convergence speed in the early stage is significantly affected.
This shows that \textit{MAD-LM} can perform well in the fine-tuning stage without requiring a large number of samples during the pre-training stage.

\bmhead{Effect of the latent size}
\begin{figure}
	\centering
	\includegraphics[width=0.6\columnwidth]{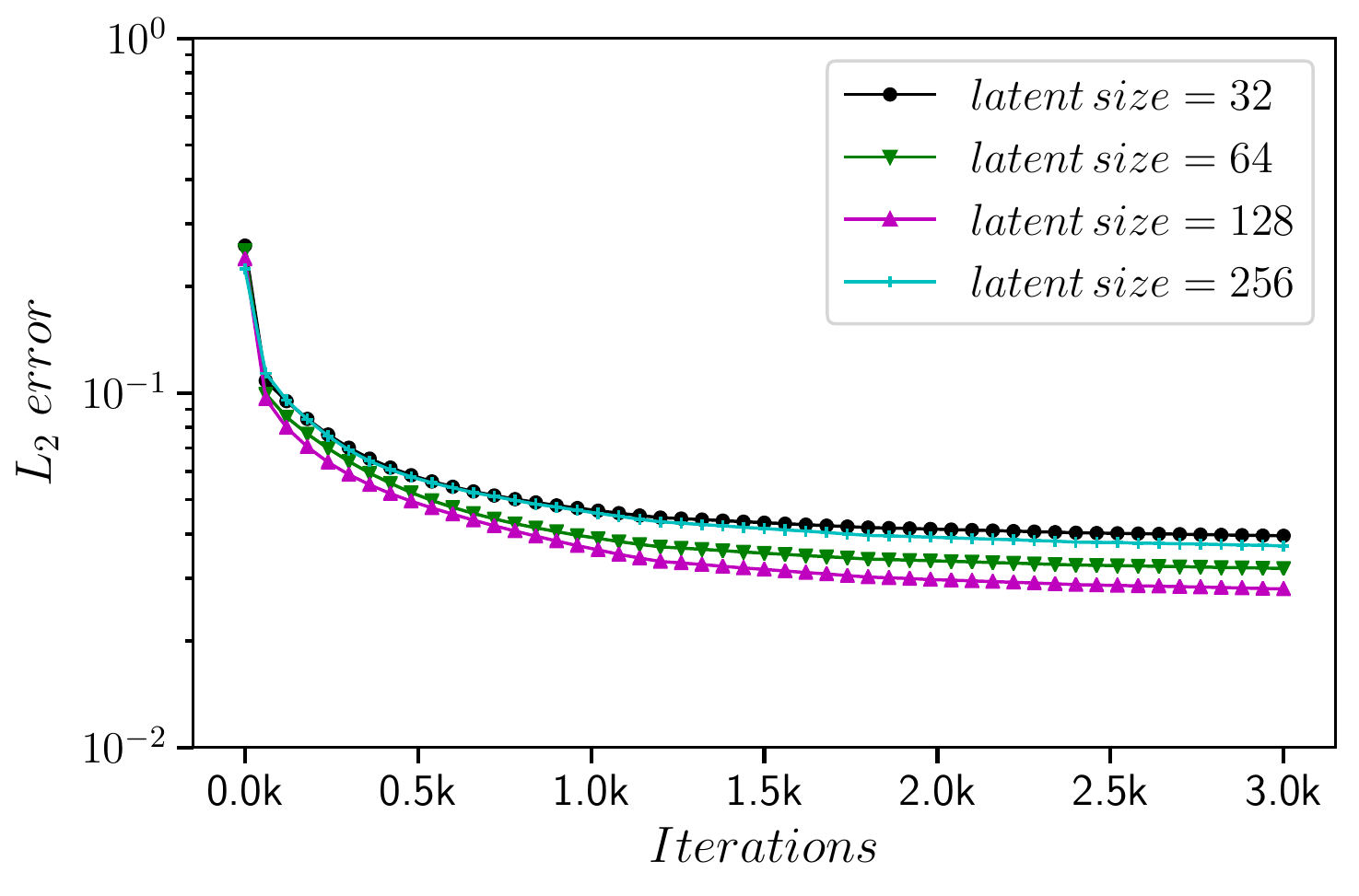}
	\caption{\textbf{Burgers' equation:} The convergence of mean $L_2\ error$ for \textit{MAD-L} with respect to the number of training iterations using different latent sizes.}%
	\label{fig:burgers_ablation_latent_size_MAD_l}
\end{figure}
\begin{figure}
	\centering
	\includegraphics[width=0.6\columnwidth]{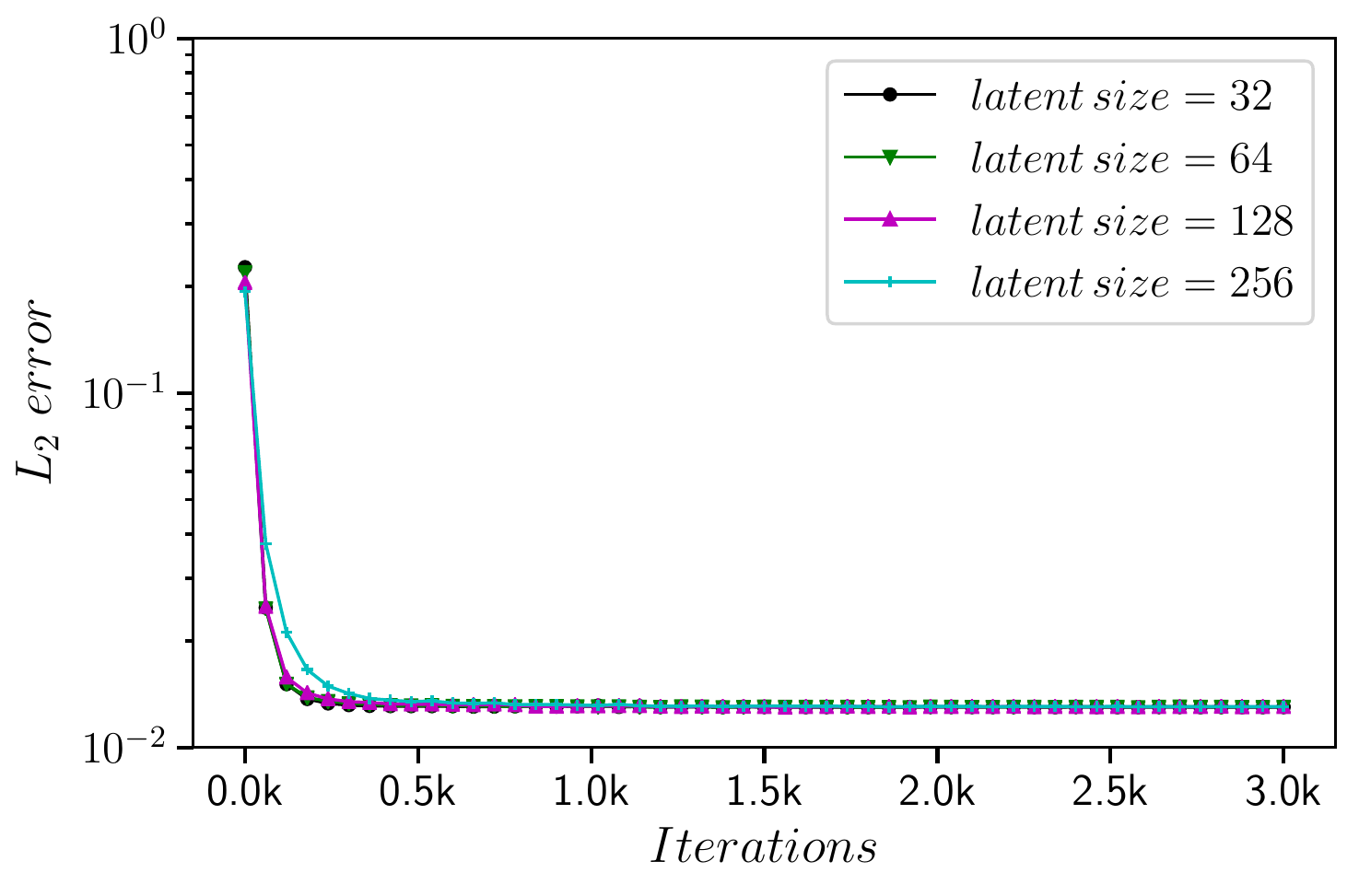}
	\caption{\textbf{Burgers' equation:} The convergence of mean $L_2\ error$ for \textit{MAD-LM} with respect to the number of training iterations using different latent sizes.}%
	\label{fig:burgers_ablation_latent_size_MAD_lm}
\end{figure}

We investigated the effect of the dimension of the latent vector (latent size) on performance. 
As can be seen from Fig.\ref{fig:burgers_ablation_latent_size_MAD_l}, for \textit{MAD-L}, different choices of the latent size would lead to different solution accuracy, among which the best is achieved at $n=128$. 
The difference is almost negligible for \textit{MAD-LM} according to Fig.\ref{fig:burgers_ablation_latent_size_MAD_lm}.

\verB{
\bmhead{Effect of a different loss function}
\begin{figure}
	\centering
	\includegraphics[width=0.7\columnwidth]{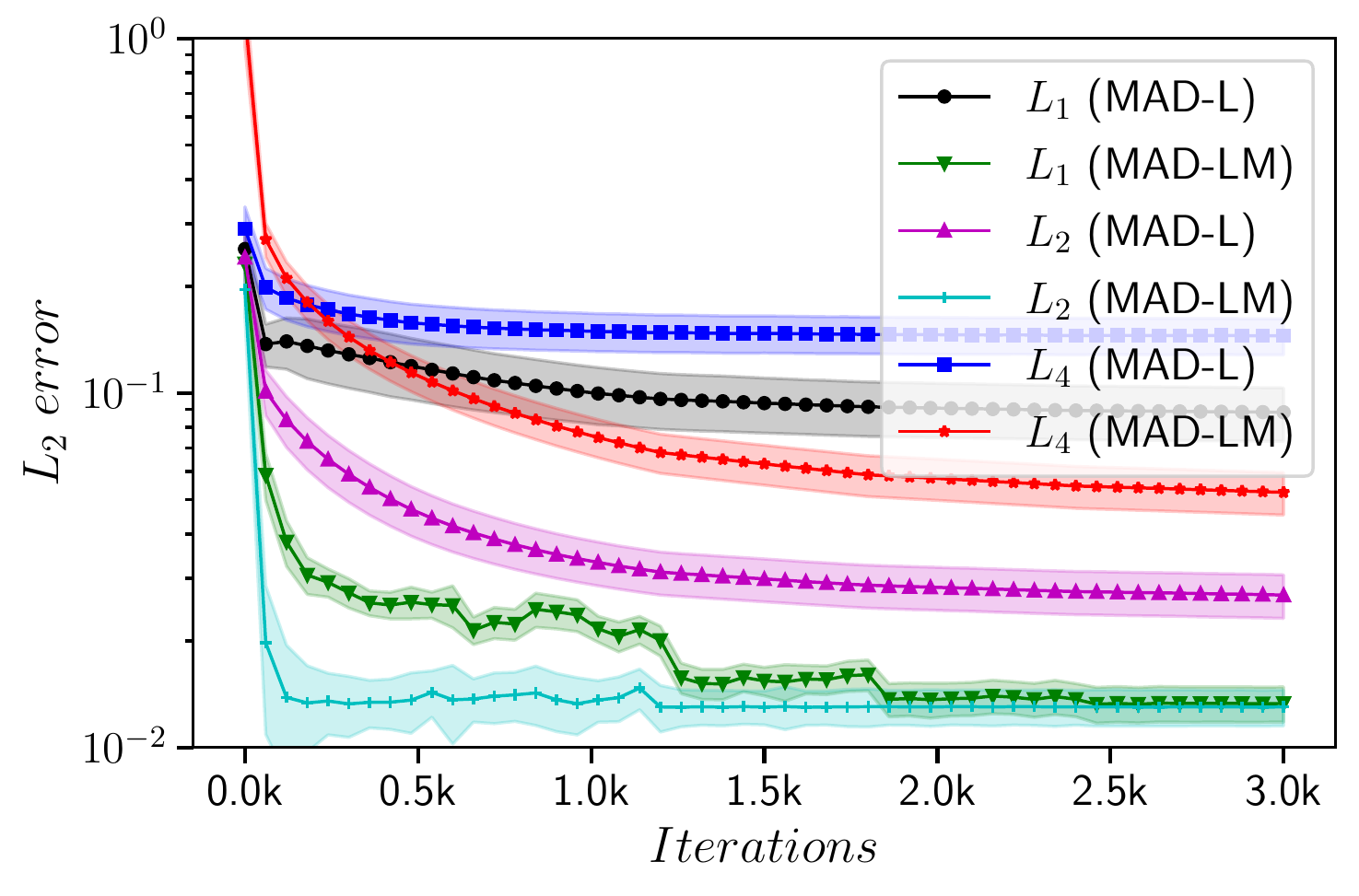}
	\caption{\textbf{Burgers' equation:} The convergence of mean $L_2\ error$ with respect to the number of training iterations using different physics-informed loss functions.}%
	\label{fig:burgers_ablation_norm}
\end{figure}
The physics-informed loss Eq.\eqref{eq:PIloss} can be generalized to
\begin{equation}
	L^\eta[u] = \|\mathcal{L}_{\widetilde{\pmb{x}}}^{\gamma_1} u\|_{L_p(\Omega)}^p + \lambda_\text{bc}\|\mathcal{B}_{\widetilde{\pmb{x}}}^{\gamma_2} u\|_{L_p(\partial\Omega)}^p
,\end{equation}
and is used consistently throughout the pre-training and the fine-tuning stages. 
We investigated the effect of $p$ on performance. 
As can be seen from Fig.\ref{fig:burgers_ablation_norm}, a different choice of $p$ affects the solution accuracy much for both \textit{MAD-L} and \textit{MAD-LM}, and the default choice $p=2$ is already good enough. 
However, the optimal choice of the loss function could be problem-dependent, especially for high-dimensional parametric PDEs. 
We refer the readers to~\cite{lu2022machineLE,wang2022isLP,Psaros2022MetaLP} for further discussions on this topic. 
}

\bmhead{Heterogeneous PDE parameters}
\begin{figure}
\begin{center}
\centerline{\includegraphics[width=0.7\columnwidth]{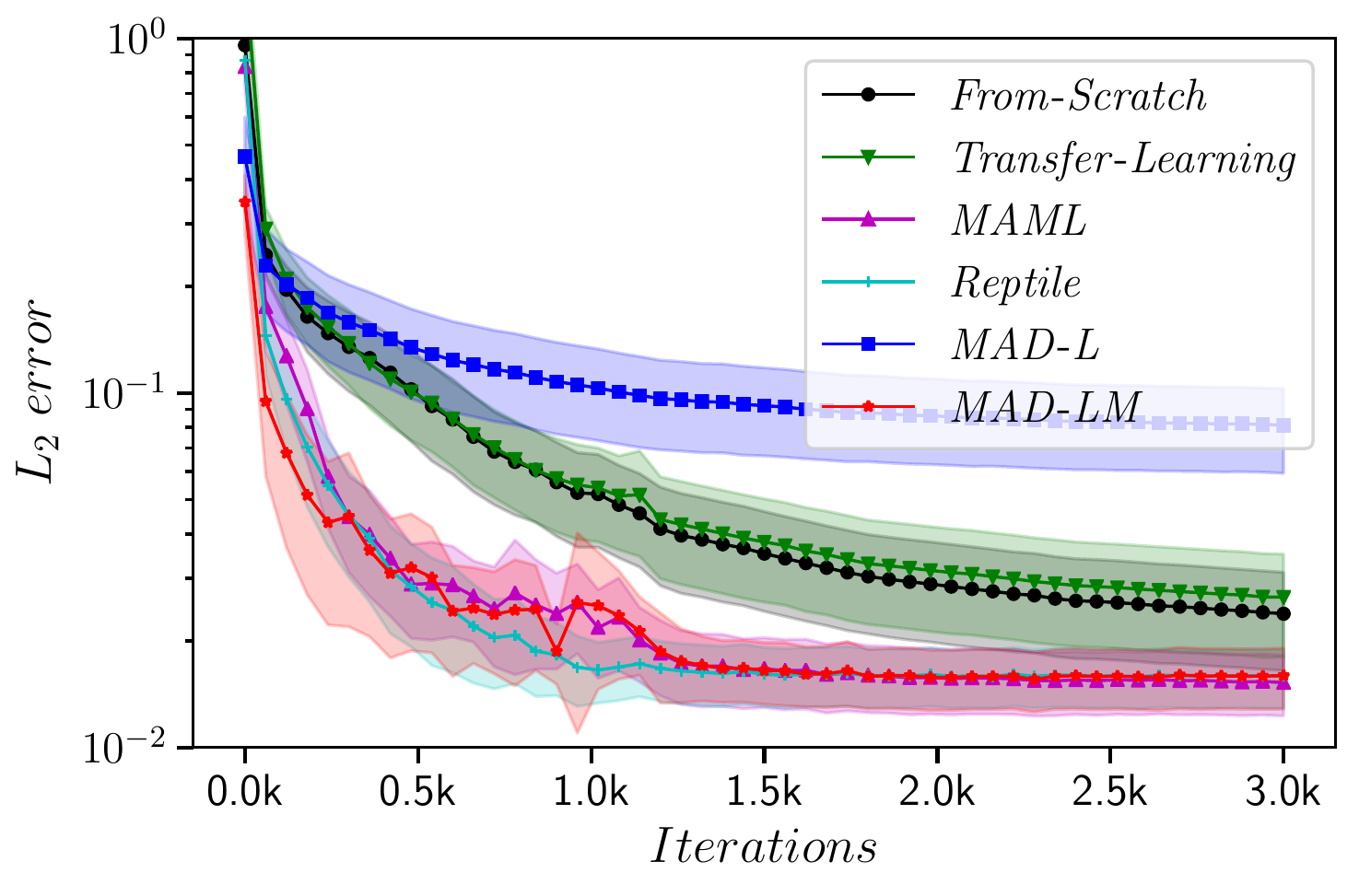}}
\caption{\textbf{Burgers' equation:} The mean $L_2\ error$ convergence with respect to the number of training iterations under heterogeneous PDE parameters.}
\label{fig:burgers_heterogeneous}
\end{center}
\end{figure}
We also consider the scenario when the viscosity coefficients $\nu$ in Eq.\eqref{def:burgers} vary within a certain range, which makes the variable PDE parameter $\eta=(\nu, u_0(x))$ heterogeneous. 
Specifically, we take $\nu \sim \{ 10^\beta \mid \beta \sim U(-3,-1) \}$ where $U$ is the uniform distribution, and 
$u_0(x) \sim \mathcal{N} (0; 100(- \Delta + 9I)^{-3})$ as before.
In this experiment, $|S_1|=100$ and $|S_2|=50$ while $S_1$ and $S_2$ come from the same task distribution.
Fig.\ref{fig:burgers_heterogeneous} compares the convergence curves of mean $L_2\ error$ corresponding to different methods.
\textit{MAD-LM} has an obvious speed and accuracy improvement over \textit{From-Scratch} and \textit{Transfer-Learning}.
It's worth noting that \textit{MAML} and \textit{Reptile} also perform well in this scenario.

\bmhead{Extrapolation}
In the above experiments, $\eta$'s in $S_1$ and $S_2$ come from the same GRF, so we can assume that the tasks in the pre-training stage come from the same task distribution as the tasks in the fine-tuing stage. 
We investigate the extrapolation capability of MAD, that is, tasks in the fine-tuing stage come from a different task distribution than those in the pre-training stage.
Specifically, $S_1$ is still the same as above, but $S_2$ is generated from $\mathcal{N}(0; 100(- \Delta + 25I)^{-2.5})$ instead.
Fig.\ref{fig:burgers_extrapolation_grf} shows the results of extrapolation experiments.
Since the distribution of tasks has changed, the trial manifold learned in the pre-training stage fits $G(\mathcal{A})$ worse, and the accuracy of \textit{MAD-L} diminishes as a consequence. 
However, as in Fig.\ref{fig:burgers_compare_other_method}, 
\textit{MAD-LM} exhibits a faster convergence speed than other methods. 
This validates the extrapolation capability of the MAD method. 

\begin{figure}
	\centering
	\includegraphics[width=0.7\columnwidth]{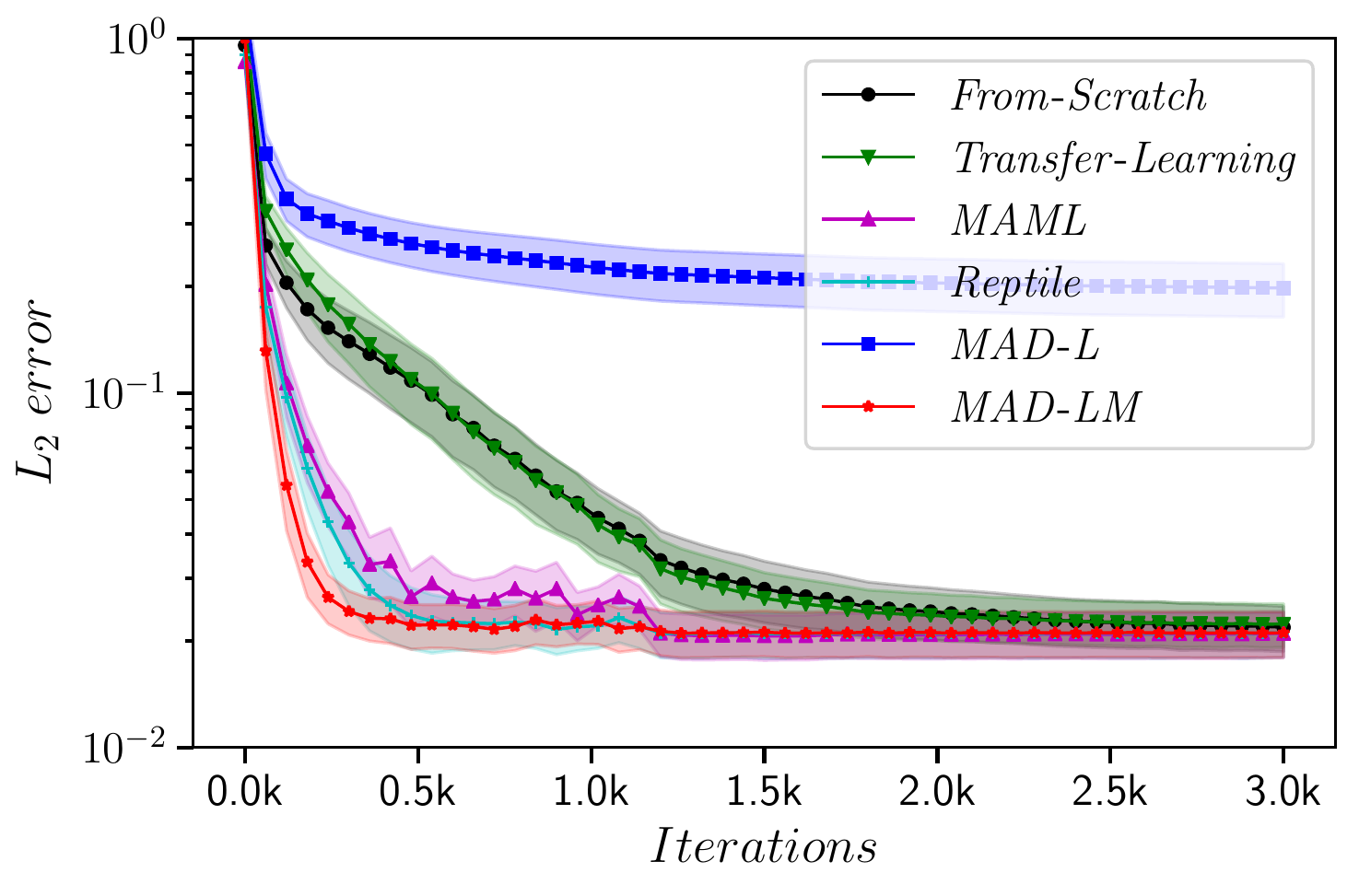}
	\caption{\textbf{Burgers' equation:} The convergence of mean $L_2\ error$ with respect to the number of training iterations for extrapolation experiments.}%
	\label{fig:burgers_extrapolation_grf}
\end{figure}

\subsection{Time-Domain Maxwell's Equations}\label{sec:maxwell}
\begin{figure}
	\centering
	\includegraphics[width=0.7\columnwidth]{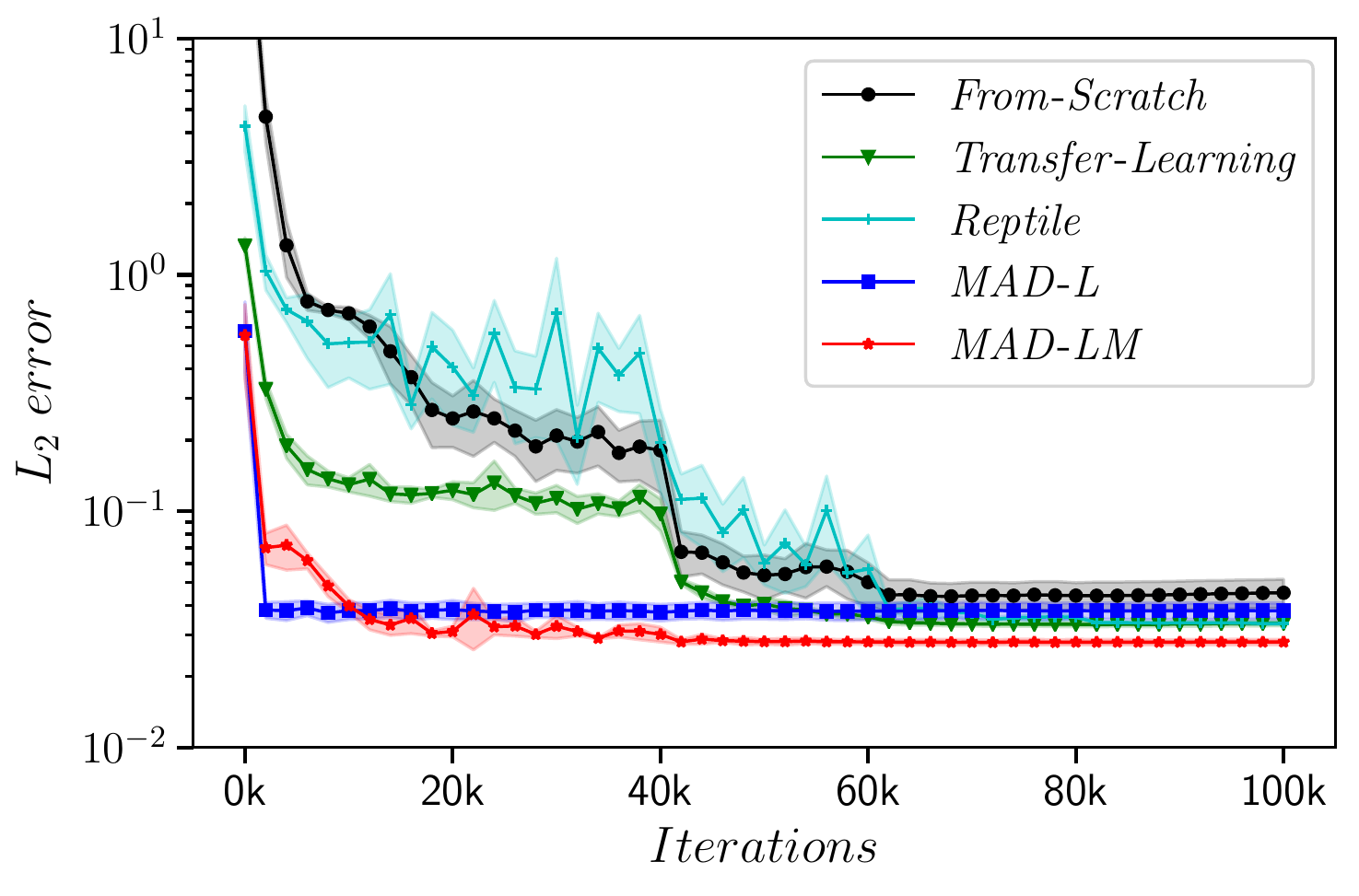}
	\caption{\textbf{Maxwell's equations:} The convergence of mean $L_2\ error$ with respect to the number of training iterations.}%
	\label{fig:maxwell_compare_other_method}
\end{figure}
We consider the time-domain 2-D Maxwell's equations with a point source in the transverse Electric (TE) mode~\cite{gedney2011introduction}:
\begin{equation}\label{eq:point_src}
\begin{aligned}
	\frac{\partial E_x}{\partial t } &= \frac{1}{\epsilon_0\epsilon_r} \frac{\partial H_z}{\partial y}, \quad
	\\\frac{\partial E_y}{\partial t } &= -\frac{1}{\epsilon_0\epsilon_r} \frac{\partial H_z}{\partial x}, \quad
	\\\frac{\partial H_z}{\partial t } &= -\frac{1}{\mu_0\mu_r} \left(\frac{\partial E_y}{\partial x} - \frac{\partial E_x}{\partial y} + J\right),
\end{aligned}
\end{equation}
where $E_x$, $E_y$ and $H_z$ are the electromagnetic fields, and $J$ is the point source term.
The equation coefficients $\epsilon_0$ and $\mu_0$ are the permittivity and permeability in vacuum, 
and $\epsilon_r$ and $\mu_r$ are the relative permittivity and relative permeability of the media, respectively.
The computational domain $\Omega$ is $[0, 1]^2 \times [0, 4\times 10^{-9}]$. 
The electromagnetic field is initialized to be zero everywhere and the boundary condition is the standard Mur's second-order absorbing boundary condition~\cite{schneider2010understanding}.
The source function $J$ in Eq.\ref{eq:point_src} is known and we set it as a Gaussian pulse. 
In the temporal domain, this function can be expressed as:
\begin{equation}\label{eq:gauss_pulse}
\begin{aligned}
	J(x, y, t) = \exp\left(-\left(\frac{t-d}{\tau }\right)^2\right) \delta(x-x_0)\delta(y-y_0)
\end{aligned}
,\end{equation}
where $d$ is the temporal delay, $\tau = 3.65 \times \sqrt{2.3}/(\pi f)$ is a pulse-width parameter, $\delta(\cdot)$ is the Dirac functional used to represent the point source, and $(x_0, y_0)=(0.5,0.5)$ is the location of the point source.
The characteristic frequency $f$ is set to be $1\,$GHz.

We take $(\epsilon_r$, $\mu_r)$ to be the variable parameters of the PDEs, i.e., $\eta=(\epsilon_r$, $\mu_r)$, which corresponds to the media properties in the simulation region. 
A total of 25 pairs of $(\epsilon_r$, $\mu_r)$ are collected from the equidistant grid of $[1,5]^2$, among which 20 are randomly selected to form $S_1$, and the rest 5 left for $S_2$.
The reference solutions are obtained through the finite-difference time-domain (FDTD)~\cite{schneider2010understanding} method. 
To deal with the singularity caused by the point source, we approximate $\delta(\cdot)$ with a smoothed function, and use the lower bound uncertainty weighting method
along with the MS-SIREN network structure as proposed in~\cite{huang2021solving}. 
We set the total number of iterations to $100$k 
for the training of \textit{From-Scratch}, the pre-training and fine-tuning of \textit{Transfer-Learning}, 
and the fine-tuning of both \textit{MAD-L} and \textit{MAD-LM}. 
A total of $200$k iterations are involved to pretrain the MAD method. 

Fig.\ref{fig:maxwell_compare_other_method} shows that all methods tested converge to a similar accuracy (mean $L_2\ error$ being close to $0.04$), and \textit{MAD-LM} achieves the lowest mean $L_2\ error$ ($0.028$). 
In terms of convergence speed, \textit{MAD-L} and \textit{MAD-LM} appear to be superior to the other methods. 
It is worth noting that \textit{MAML} fails to converge in the pre-training (or meta-training) stage, and the corresponding results are not included in Fig.\ref{fig:maxwell_compare_other_method}.
We suspect that this originates from the singularity brought by point source as well as the introduction of higher-order derivatives, which may pose great difficulties in solving the optimization problem. 
\textit{Reptile} does not show good generalization capability as well probably due to the same singularity issue.

\begin{figure}
	\centering
	\includegraphics[width=0.8\columnwidth]{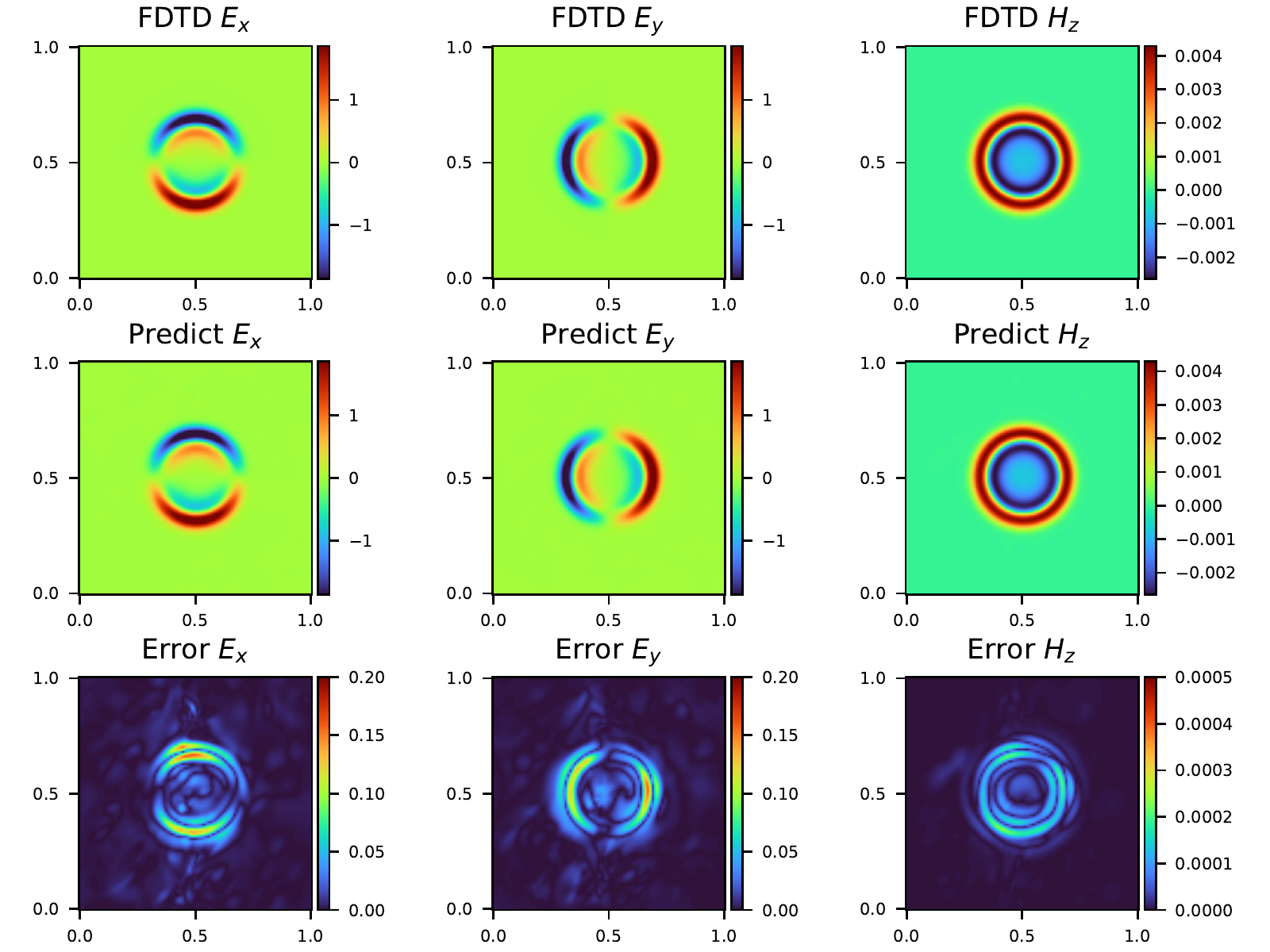}
	\caption{
		\textbf{Maxwell's equations:} Model predictions of \textit{MAD-L} compared with the FDTD solutions at $t = 4\,$ns.
		\textbf{Top:} The reference solutions of $(E_x, E_y, H_z)$ computed using FDTD.
		\textbf{Middle:} The predicted solutions of $(E_x, E_y, H_z)$ given by the learned model.
		\textbf{Bottom:} The absolute error between model predictions and the reference solutions.}
	\label{fig:maxwell_MAD_l}
\end{figure}
\begin{figure}
	\centering
	\includegraphics[width=0.8\columnwidth]{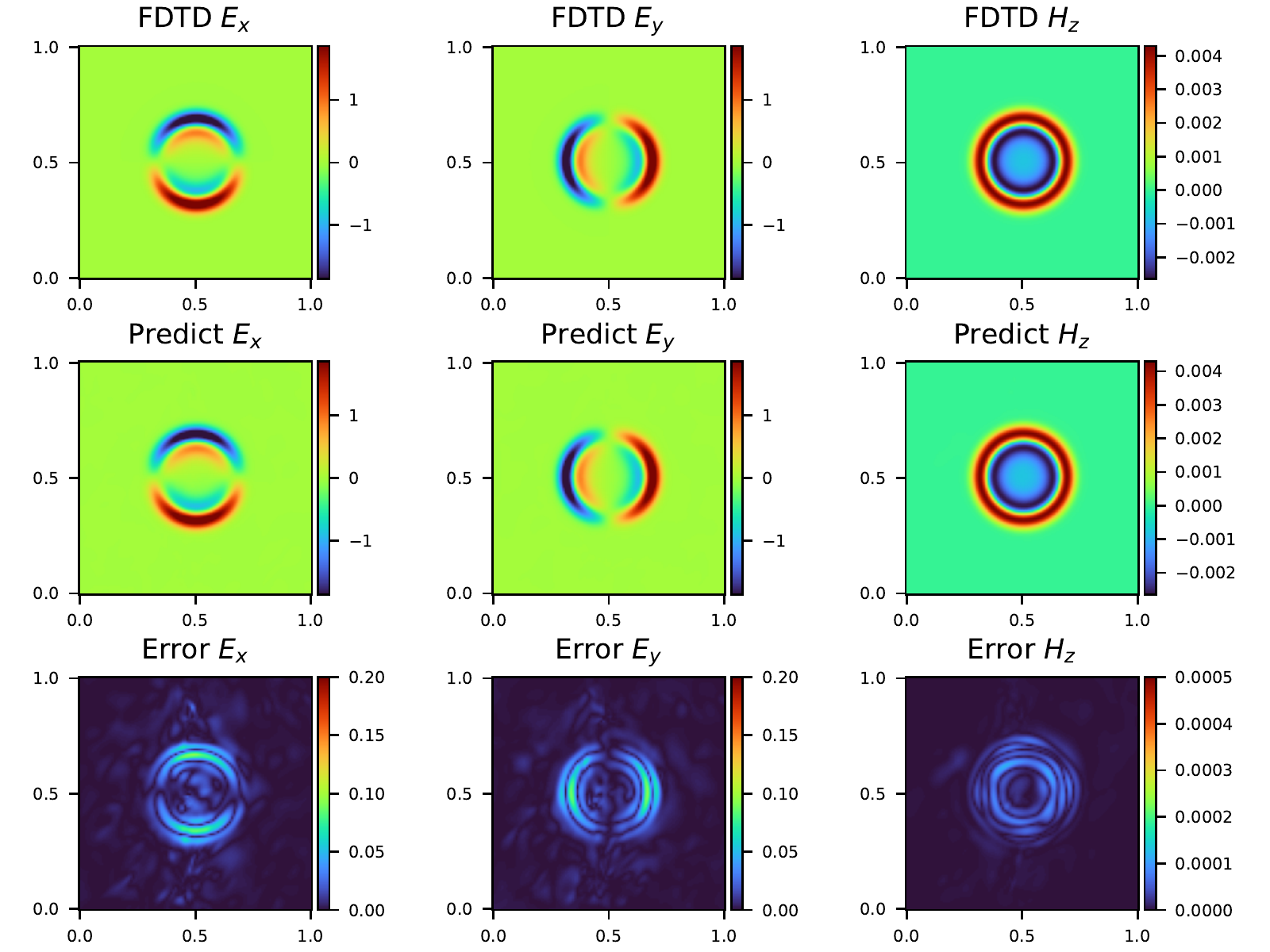}
	\caption{
		\textbf{Maxwell's equations:} Model predictions of \textit{MAD-LM} compared with the FDTD solutions at $t = 4\,$ns.
		\textbf{Top:} The reference solutions of $(E_x, E_y, H_z)$ computed using FDTD.
		\textbf{Middle:} The predicted solutions of $(E_x, E_y, H_z)$ given by the learned model.
		\textbf{Bottom:} The absolute error between model predictions and the reference solutions.}
	\label{fig:maxwell_MAD_lm}
\end{figure}

The instantaneous electromagnetic fields at time $t=4\,$ns of \textit{MAD-L} and \textit{MAD-LM} compared with the reference FDTD results for $(\epsilon_r$, $\mu_r) = (3, 5)$ are depicted in Fig.\ref{fig:maxwell_MAD_l} and Fig.\ref{fig:maxwell_MAD_lm}, respectively. 
\textit{MAD-LM} achieves a lower absolute error as can be observed in the figures. 
Specifically, the $L_2\ error$ is $0.037$ for \textit{MAD-L} and $0.030$ for \textit{MAD-LM}. 

\textit{PI-DeepONet} is also applied to the solution of time-domain Maxwell's equations with a point source. 
In this experiment, the branch net of \textit{PI-DeepONet} is a 4-layer fully connected network with 64 neurons in each hidden layer. 
The trunk net is an MS-SIREN~\cite{huang2021solving} network, which consists of 4 subnets, each with 7 fully connected layers and 64 neurons in each hidden layer.
The PDE parameter $(\epsilon_r,\mu_r)$ is taken directly as the input of the branch net. 
As there are 3 fields to be predicted $\left( E_x(x,y,t), E_y(x,y,t), H_z(x,y,t) \right)$, we adopt the method proposed in~\cite{LuLu2021ACA} to solve the multi-output problem, i.e., split the outputs of both the branch net and the trunk net into 3 groups, and then the $k$-th group is expected to provide the $k$-th field of the solution.
However, due to the optimization difficulties caused by the singularity of the point source, \textit{PI-DeepONet} struggles to achieve a satisfactory accuracy (mean $L_2\ error$ being $0.672$), while the mean $L_2\ error$ of \textit{MAD-LM} is $0.028$.

\begin{figure}
	\centering
	\includegraphics[width=0.7\columnwidth]{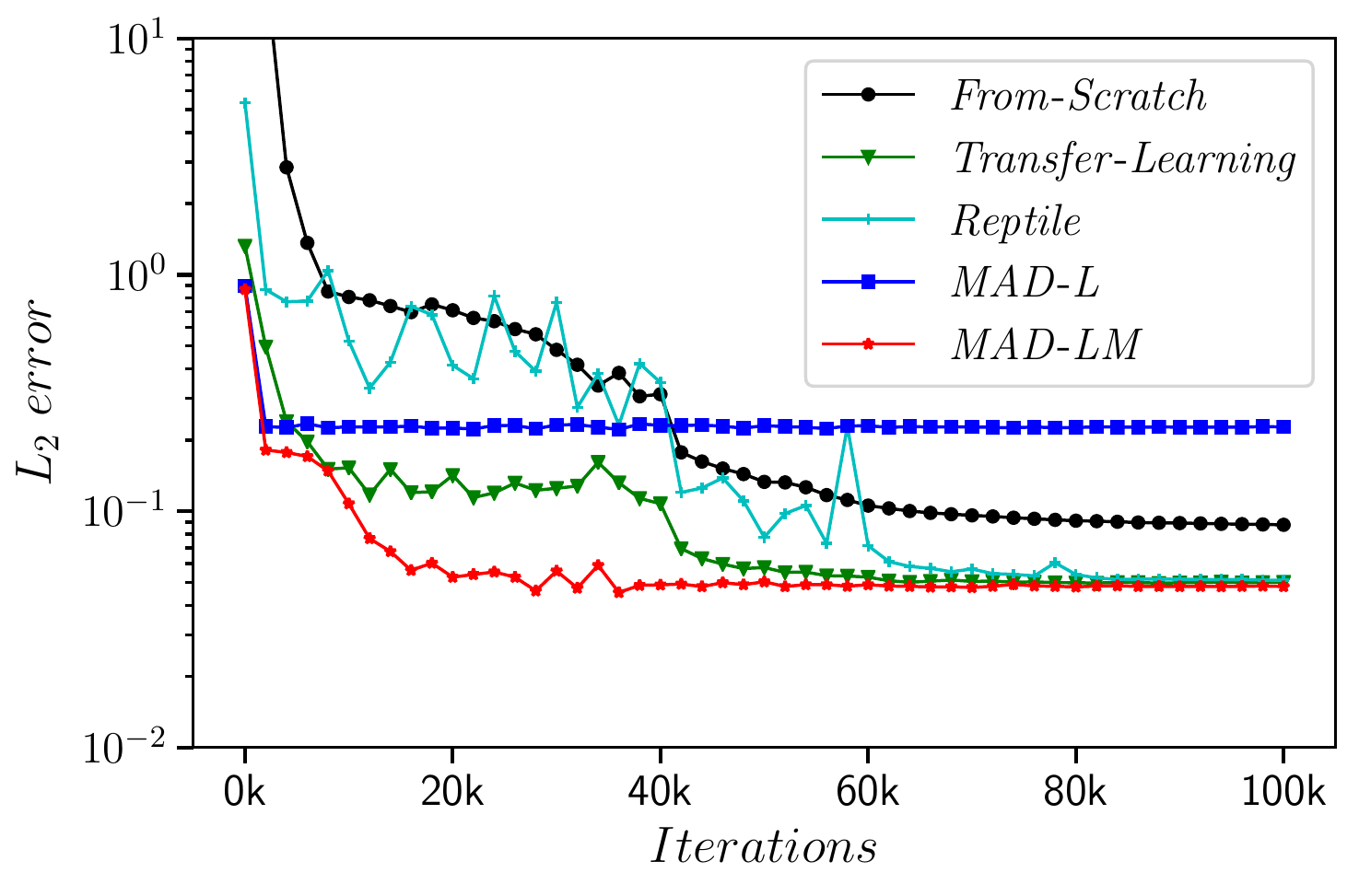}
	\caption{\textbf{Maxwell's equations:} The convergence of mean $L_2\ error$ with respect to the number of training iterations for extrapolation experiments.}%
	\label{fig:maxwell_extrapolation}
\end{figure}
\bmhead{Extrapolation}
In this case, the PDE parameters $(\epsilon_r, \mu_r)$ in $S_1$ come from $[1,5]^2$ as before, but in the fine-tuning stage, only the case $(\epsilon_r, \mu_r)=(7, 7)$ is considered.
Since the extrapolated task does not lie in the task distribution used in the pre-training stage, the corresponding point $G(\eta_\text{new})$ in the function space $\mathcal{U}$ is not on the learned trial manifold $D_{\theta^*}(Z_B)$, which makes \textit{MAD-L} to converge to a poor accuracy.
Fig.\ref{fig:maxwell_extrapolation} indicates that \textit{MAD-LM} is significantly faster than \textit{From-Scratch} and \textit{Reptile} in convergence speed while maintaining high accuracy.
Notably, \textit{Transfer-Learning} also exhibits faster convergence than \textit{From-Scratch} and \textit{Reptile}.
As a matter of fact, the PDE parameter randomly selected in its pre-training stage appears to be $(\epsilon_r, \mu_r)=(4, 5)$, which is very close to $(\epsilon_r, \mu_r)=(7, 7)$ in Euclidean distance.

\subsection{Laplace's Equation}\label{sec:laplace}
Consider the 2-D Laplace's equation as follows:
\begin{equation}\label{def:laplace}
\begin{aligned}
	\frac{\partial ^2 u}{\partial x^2}  + \frac{\partial ^2 u}{\partial y^2} &= 0, & (x,y) \in \Omega, 
	\\u(x,y) &= g(x,y), & (x,y) \in \partial \Omega,
\end{aligned}
\end{equation}
where the shape of $\Omega$ and the boundary condition $g(x,y)$ are the variable parameters of the PDE, i.e. $\eta=(\Omega, g(x,y))$.
The computational domain $\Omega$ is a convex polygon arbitrarily taken from the interior of the unit disk. 
To be more specific, we randomly choose an interger $k$ from $\{3,4,\dots,10\}$, and then sample $k$ points on the unit circle to form a convex polygon with $k$ vertices. 
Given that $h(x,y)$ is the boundary condition on the unit circle, we use a GRF to generate $h \sim \mathcal{N}(0, 10^{3/2}(-\Delta + 100I)^{-3})$ with periodic boundary conditions.
The analytical solution of the Laplace's equation on the unit disk can be obtained by computing the Fourier coefficients. 
Restricting this solution to the $k$ sides of the polygon would give rise to the boundary condition $g(x,y)$,
whereas restricting it to the interior of the polygon produces the reference solution. 
A total of $100$ such PDE parameters are generated to form $S_1$ and $50$ to form $S_2$, each corresponding to a different computational domain and a different $h(x,y)$, thereby a different boundary condition $g(x,y)$. 
For each set of the PDE parameter, we randomly select $16\times 1024$ points from the interior of the polygon, and obtain the analytic solution corresponding to these points to evaluate the accuracy of the models.

Note that the variable PDE parameters include the shape of the polygonal computational domain and the boundary conditions on the sides of the polygon, and are therefore heterogeneous. 
MAD can implicitly encode such heterogeneous PDE parameters as latent vectors conveniently, whereas \textit{PI-DeepONet} is unable to handle this case without further adaptations.
When we apply the MAD method to solve this problem, it is not convenient to measure the distance between $\eta_{\text{new}}$ and $\eta_i$ 
due to such heterogeneity. 
Therefore, the average of $|S_1|=100$ latent vectors obtained in the pre-training stage is used as the initialization of $\pmb{z}$ in the fine-tuning stage.
For the training of \textit{From-Scratch}, the pre-training and fine-tuning of \textit{Transfer-Learning}, we set the total number of iterations to $10$k. 
As for \textit{MAD-L} and \textit{MAD-LM}, the pre-training and fine-tuning involve $50$k and $10$k iterations, respectively. 

\begin{figure}
	\centering
	\includegraphics[width=0.7\columnwidth]{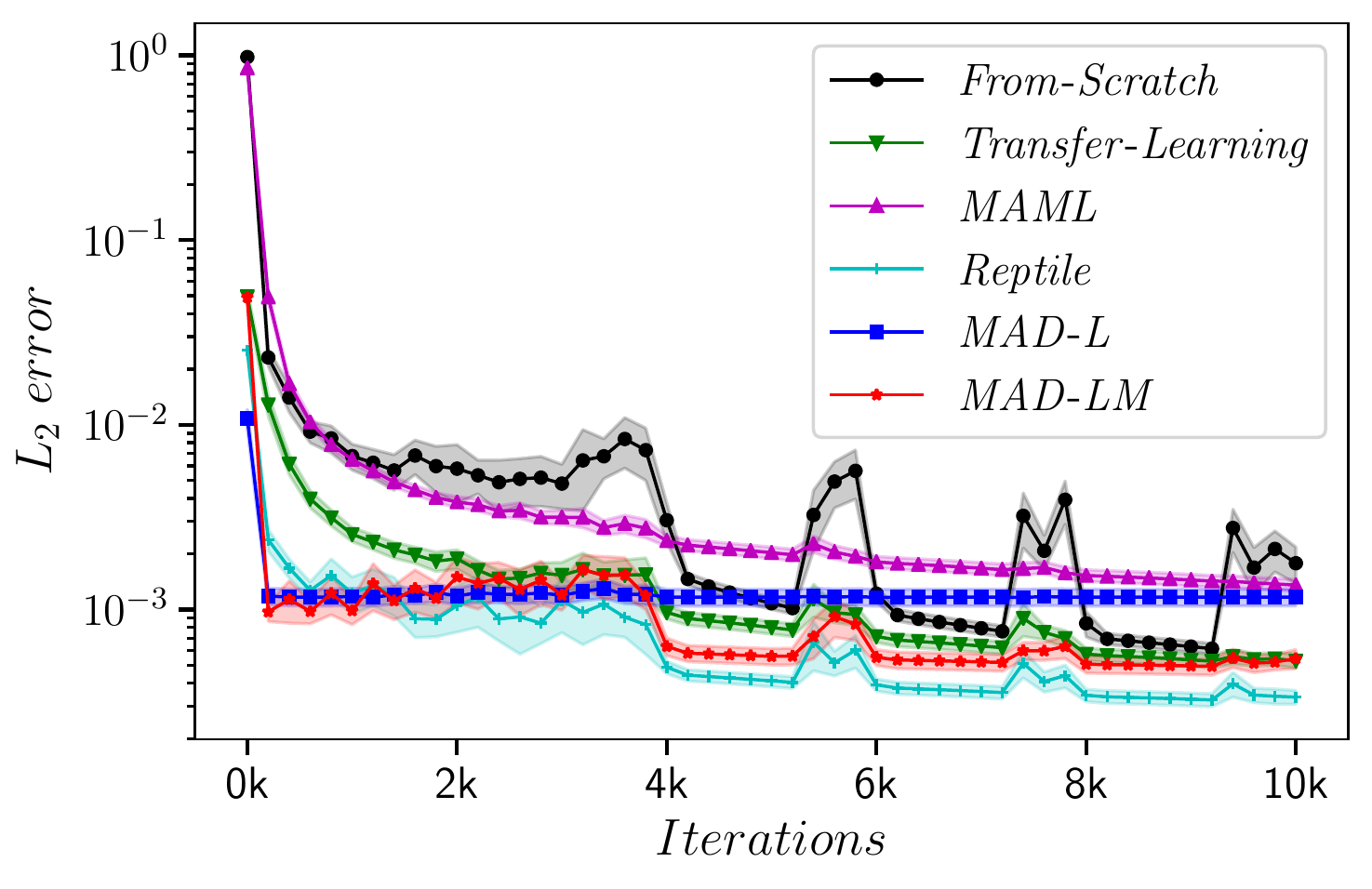}
	\caption{\textbf{Laplace's equation:} The convergence of mean $L_2\ error$ with respect to the number of training iterations.}
	\label{fig:laplace_polygon_a}
\end{figure}
\begin{figure}
	\centering
	\includegraphics[width=0.7\columnwidth]{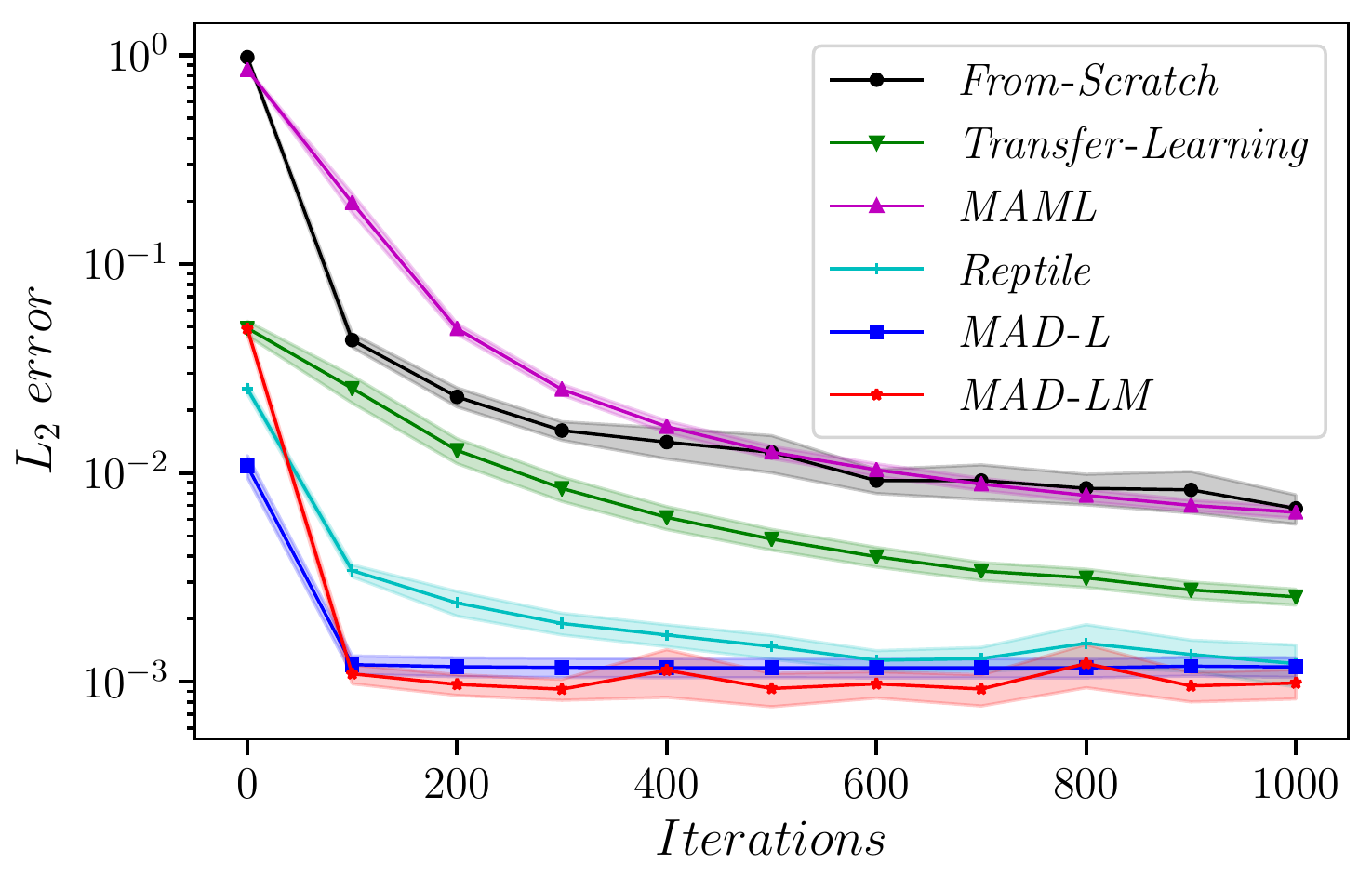}
	\caption{\textbf{Laplace's equation:} The convergence of mean $L_2\ error$ with respect to the number of training iterations, 
	zoomed-in to the first 1000 iterations.}
	\label{fig:laplace_polygon_b}
\end{figure}

\begin{figure}
\begin{center}
\centerline{\includegraphics[width=\columnwidth]{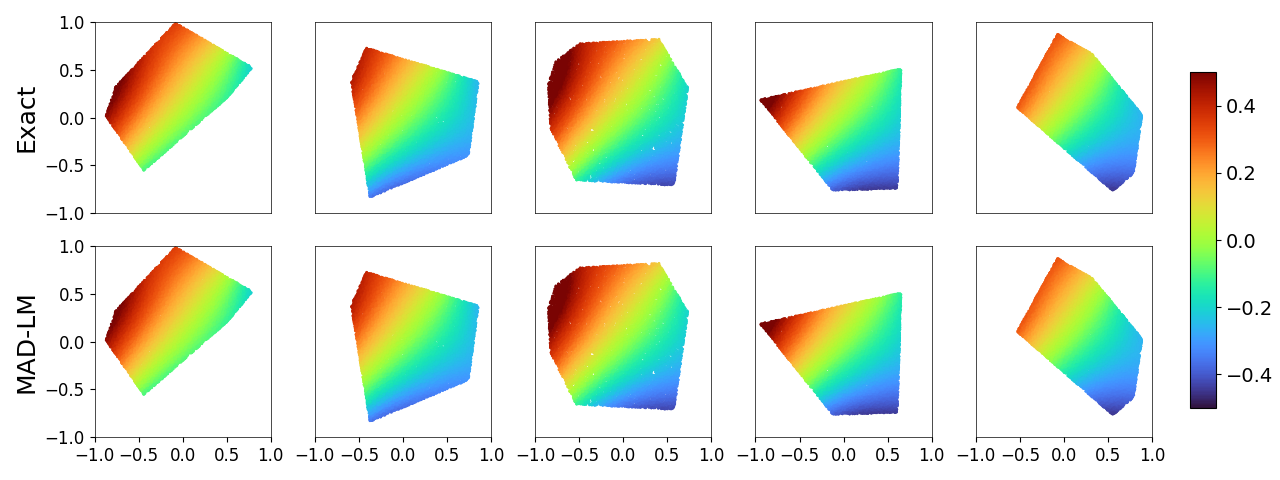}}
\caption{\textbf{Laplace's equation:} Analytical solutions and model predictions of \textit{MAD-LM} when the computational domain $\Omega$ is a polygon with different shapes.}
\label{fig:laplace_polygon_result}
\end{center}
\end{figure}

Fig.\ref{fig:laplace_polygon_a} compares the convergence curves of mean $L_2\ error$ corresponding to different methods, 
and a zoomed-in version of the first 1000 iterations is given in Fig.\ref{fig:laplace_polygon_b}. 
Compared to other methods, \textit{MAD-L} and \textit{MAD-LM} can achieve faster adaptation, i.e. very low $L_2\ error$ in less than $100$ iterations.
Fig.\ref{fig:laplace_polygon_result} compares the predictions of \textit{MAD-LM} with the analytical solutions under 5 randomly selected samples in $S_2$.

\bmhead{Extrapolation}
\begin{figure}
	\centering
	\includegraphics[width=0.7\columnwidth]{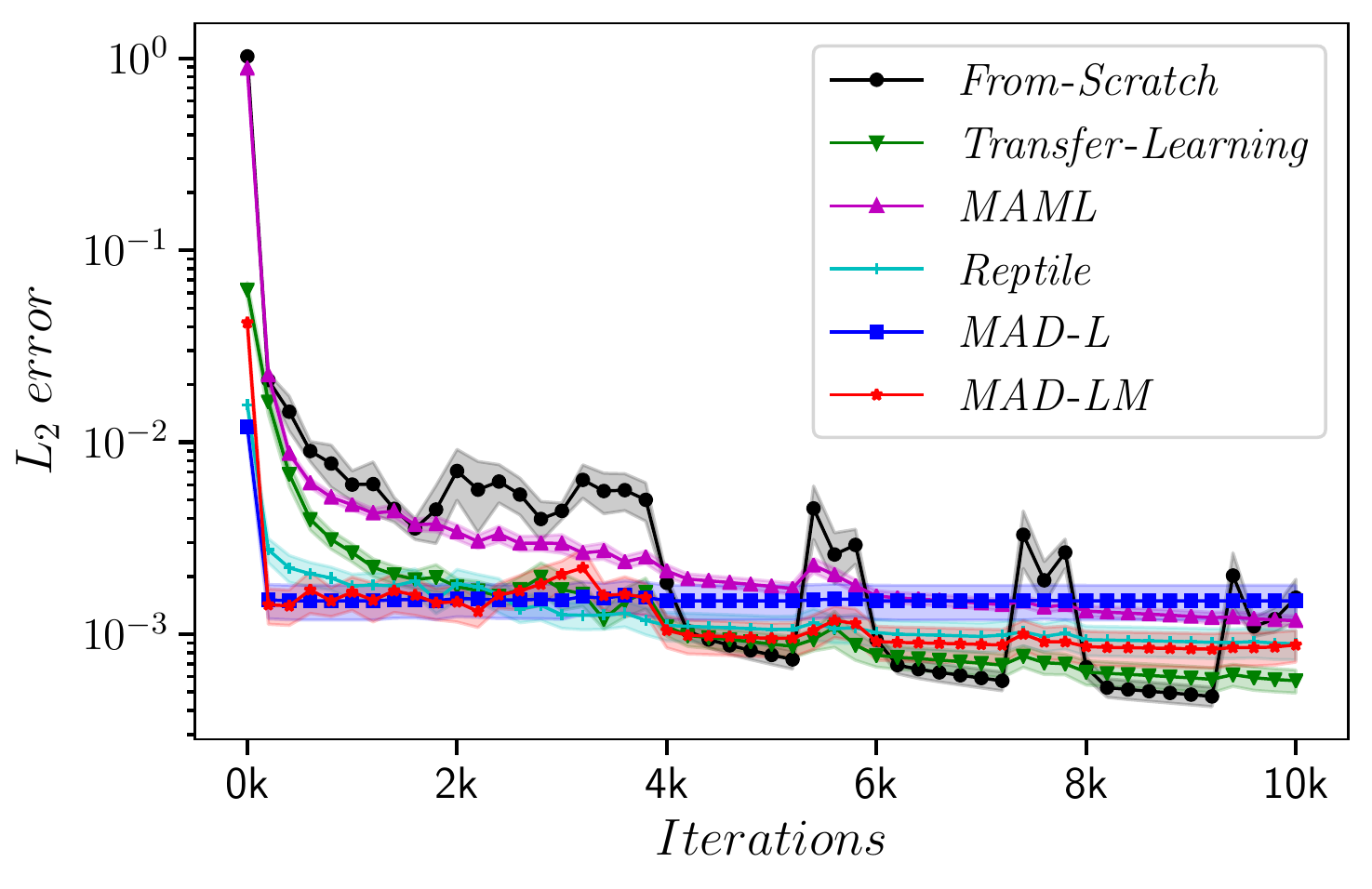}
	\caption{\textbf{Laplace's equation:} The convergence of mean $L_2\ error$ with respect to the number of training iterations for extrapolation experiments.}
	\label{fig:laplace_ellipse_a}
\end{figure}
\begin{figure}
	\centering
	\includegraphics[width=0.7\columnwidth]{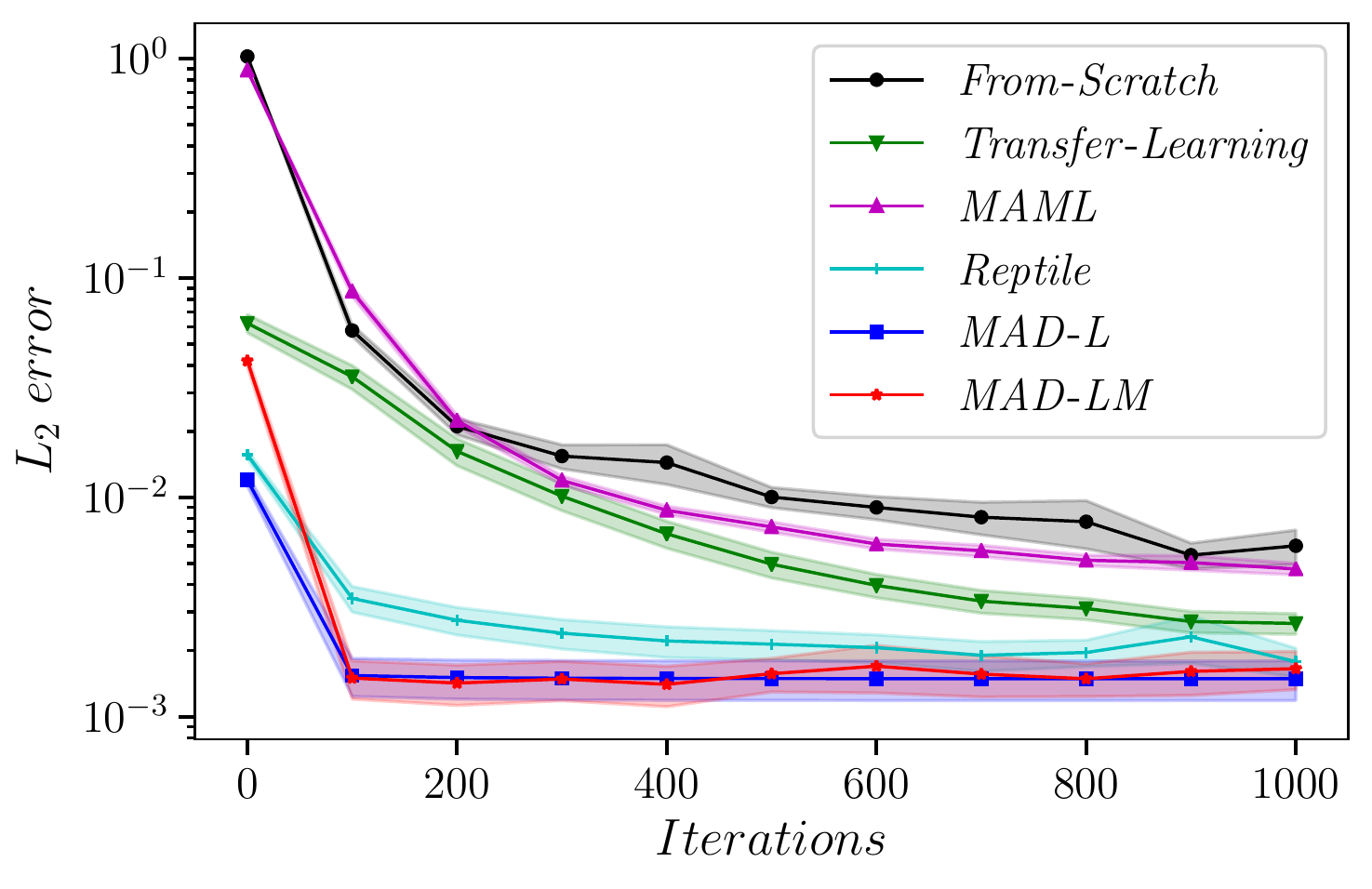}
	\caption{\textbf{Laplace's equation:} The convergence of mean $L_2\ error$ with respect to the number of training iterations for extrapolation experiments, zoomed-in to the first 1000 iterations.}
	\label{fig:laplace_ellipse_b}
\end{figure}

\begin{figure}
\begin{center}
\centerline{\includegraphics[width=\columnwidth]{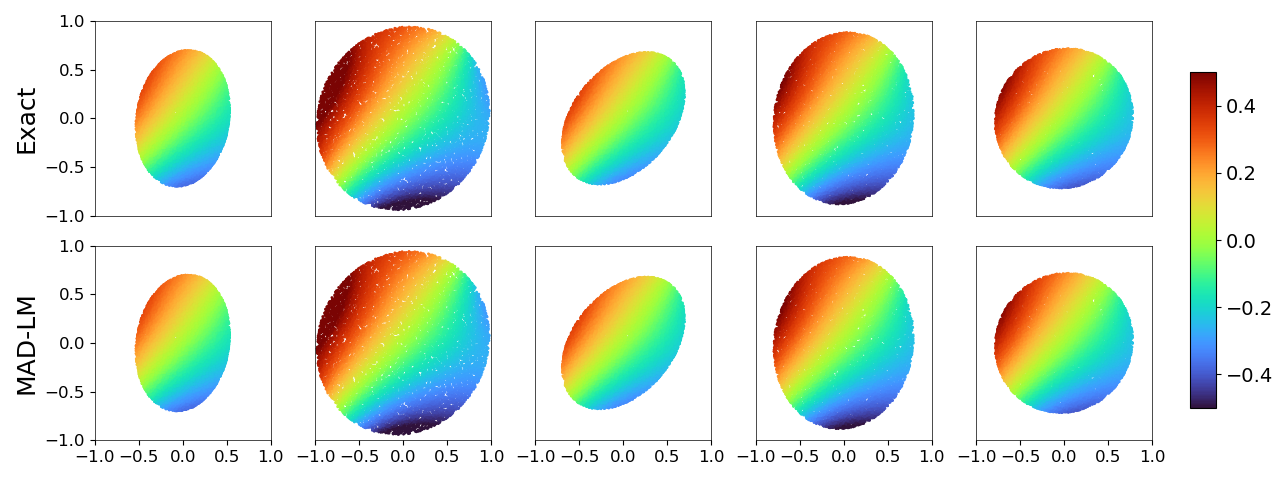}}
\caption{\textbf{Laplace's equation:} Analytical solutions and model predictions of \textit{MAD-LM} for extrapolation experiments.}
\label{fig:laplace_ellipse_result}
\end{center}
\end{figure}

We also conducted an extrapolation experiment for Laplace's Equation.
Specifically, in the pre-training stage, the shape of the computational domain $\Omega$ in $S_1$ is a convex polygon as before. 
However, in the fine-tuning stage, the shape of the computational domain $\Omega$ in $S_2$ is an ellipse arbitrarily taken from the interior of the unit circle.
In this experiment, $|S_1|=100$ and $|S_2|=20$.
Fig.\ref{fig:laplace_ellipse_b} shows that even in the case of extrapolation, \textit{MAD-LM} can achieve faster adaptation compared to other methods.
A comparison of the predictions of \textit{MAD-LM} with the analytical solutions under 5 randomly selected samples in $S_2$ is given in Fig.\ref{fig:laplace_ellipse_result}, which demonstrates the high accuracy of the solutions obtained by \textit{MAD-LM}.

\subsection{Helmholtz's Equation}\label{sec:helmholtz}
Consider the 2-D heterogeneous Helmholtz's equation with a perfectly matched layer (PML):
\begin{equation}\label{def:laplace}
\begin{aligned}
	\frac{1}{\epsilon_x}\frac{\partial }{\partial x}\left(\frac{1}{\epsilon_x}\frac{\partial u}{\partial x}\right)
	+\frac{1}{\epsilon_y}\frac{\partial }{\partial y}\left(\frac{1}{\epsilon_y}\frac{\partial u}{\partial y}\right)
	+\left(\frac{\omega}{c}\right)^2u
	=\rho,
	\quad (x,y) \in [0,4]^2,
\end{aligned}
\end{equation}
where $\omega=4\pi$ is the angular frequency of the source, $c(x,y)$ is the speed of sound, $\rho(x,y)$ is the source distribution, and $u(x,y)$ is the \emph{complex} acoustic wavefield to be solved. 
The source term $\rho(x,y)$ is set to be Gaussian in the form
\begin{equation}
	\rho(x,y)=\exp\left(-\frac{(x-x_0)^2+(y-y_0)^2}{2\tau^2}\right)
,\end{equation}
where $(x_0,y_0)=(0.78,2)$ is the location of the center of the source, and $\tau=0.625$ is the standard deviation. 
The PML for wave absorption has a thickness of $L=0.5$, with coefficients
\begin{equation}
	\epsilon_x(x,y)=1+\frac{\mathrm{i}\sigma_{\max}}{\omega L^2}\bigl(L-\min(x,4-x,L)\bigr)^2
\end{equation}
and $\epsilon_y$ defined in the same manner as $\epsilon_x$, where $\mathrm{i}$ is the imaginary unit, and we set $\sigma_{\max}=24$. 

The variable parameter of the PDE is taken to be $\eta=c(x,y)$, the distribution of the sound speed inside the computational domain. 
Following~\cite{Stanziola2021HelmholtzES} which shows interest in transcranial ultrasound therapy, we generate $c(x,y)$ that resembles a human skull. 
Specifically, a uniform grid of size $256\times 256$ is used to represent the sound speed distribution in the computational domain, with a spatial step-size $h=\frac{4}{256}$. 
A closed curve parameterized by $s\in[0,2\pi]$ is generated by summing up several circular harmonics of random amplitude and phase, in the form
\begin{equation}
\begin{aligned}
	x_\text{curve}(s)&=2.11+\sum_{k=1}^{4}a_k^x\sin(ks+\phi_k^x),
	\\y_\text{curve}(s)&=2+\sum_{k=1}^{4}a_k^y\cos(ks+\phi_k^y),
\end{aligned}
\end{equation}
where $a_1^x,a_1^y\sim U(40h,60h)$, $a_k^x,a_k^y\sim U(-A_kh,A_kh)$ for $k=2,3,4$ with $(A_2,A_3,A_4)=(5,2,1)$, and $\phi_k^x,\phi_k^y\sim\mathcal{N}\left(0,(\frac{\pi}{16})^2\right)$ for $k=1,2,3,4$. 
Expanding this curve by a random thickness $d_\text{ring}\sim U(4h,16h)$ gives rise to a ring with a hollow structure. 
We choose the sound speed of the ring according to $c_\text{ring}\sim U(1.5,2)$, while the background speed of sound is $c_\text{bg}=1$. 
The final sound speed distribution $c(x,y)$ is then obtained by applying a Gaussian filter with standard deviation $4h$ to avoid the jump discontinuity of sound speed on the boundaries of the ring. 
See Fig.\ref{fig:helmholtz_sos_eg} for several examples. 
\begin{figure}
	\centering
	\includegraphics[width=0.7\linewidth]{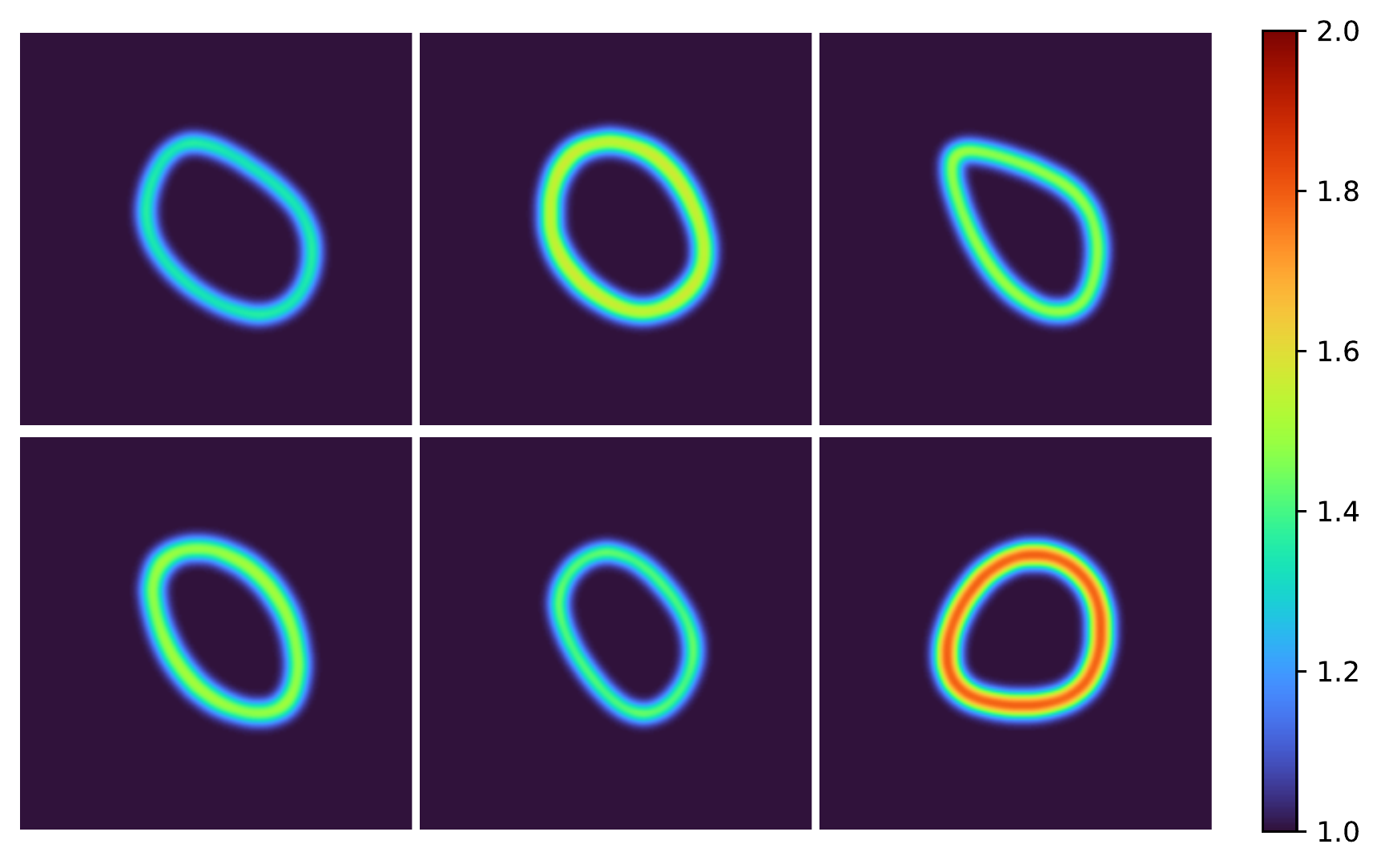}
	\caption{\textbf{Helmholtz's equation:} Examples of the sound speed distributions in Helmholtz's equation. The first five cases are taken from $S_1$, and the last one is taken from $S_2$. }
	\label{fig:helmholtz_sos_eg}
\end{figure}

A total of $20$ such sound speed distributions are randomly generated to form $S_1$, and $5$ to form $S_2$. 
We obtain the reference solutions based on the code released by~\cite{Stanziola2021HelmholtzES}.
The corresponding time-domain equation with a continuous wave sinusoidal source term is solved using the open-source k-Wave acoustics toolbox~\cite{Treeby2010kWaveMT}. 
After running to the steady state, the complex wavefield extracted by temporal Fourier transform is taken as the reference solution. 
We normalize both the reference solutions and the predicted solutions to amplitude 1 and phase 0 at the center of the source location, i.e., $u(x_0,y_0)=1+0\mathrm{i}$, and exclude the PML region when calculating the $L_2\ error$ as it is not of physical interest. 
The MAD method is pre-trained for $100$k iterations, and all methods taken into comparison involve $50$k iterations in the fine-tuning stage. 

\begin{figure}
	\centering
	\includegraphics[width=0.7\linewidth]{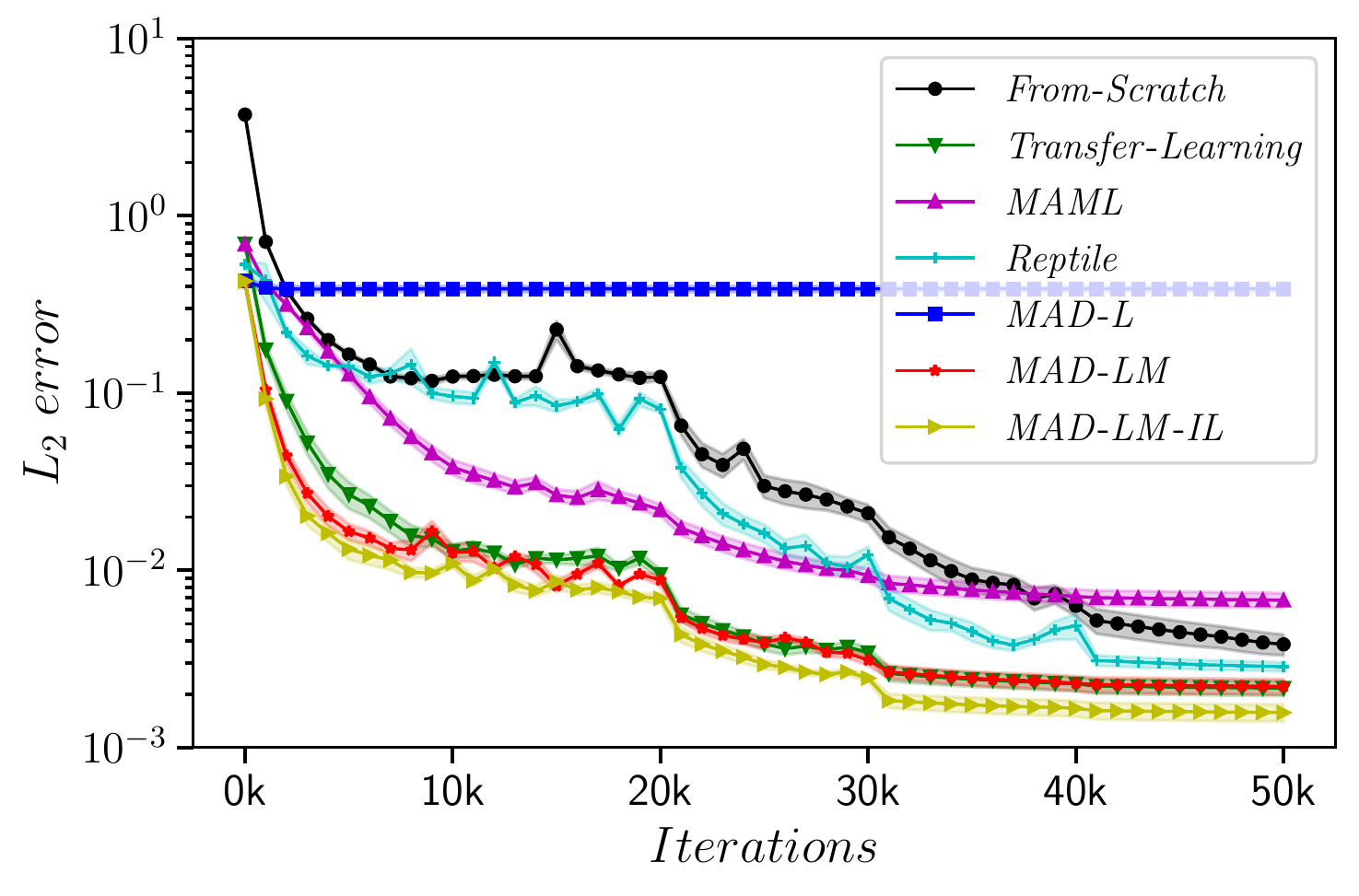}
	\caption{\textbf{Helmholtz's equation:} The convergence of mean $L_2\ error$ with respect to the number of training iterations.
		\textit{MAD-LM-IL} differs from \textit{MAD-LM} only in that it utilizes the network architecture shown in Fig.\ref{fig:helmholtz_NNarch2}. 
	}
	\label{fig:helmholtz_compare_a}
\end{figure}
\begin{figure}
	\centering
	\includegraphics[width=0.9\linewidth]{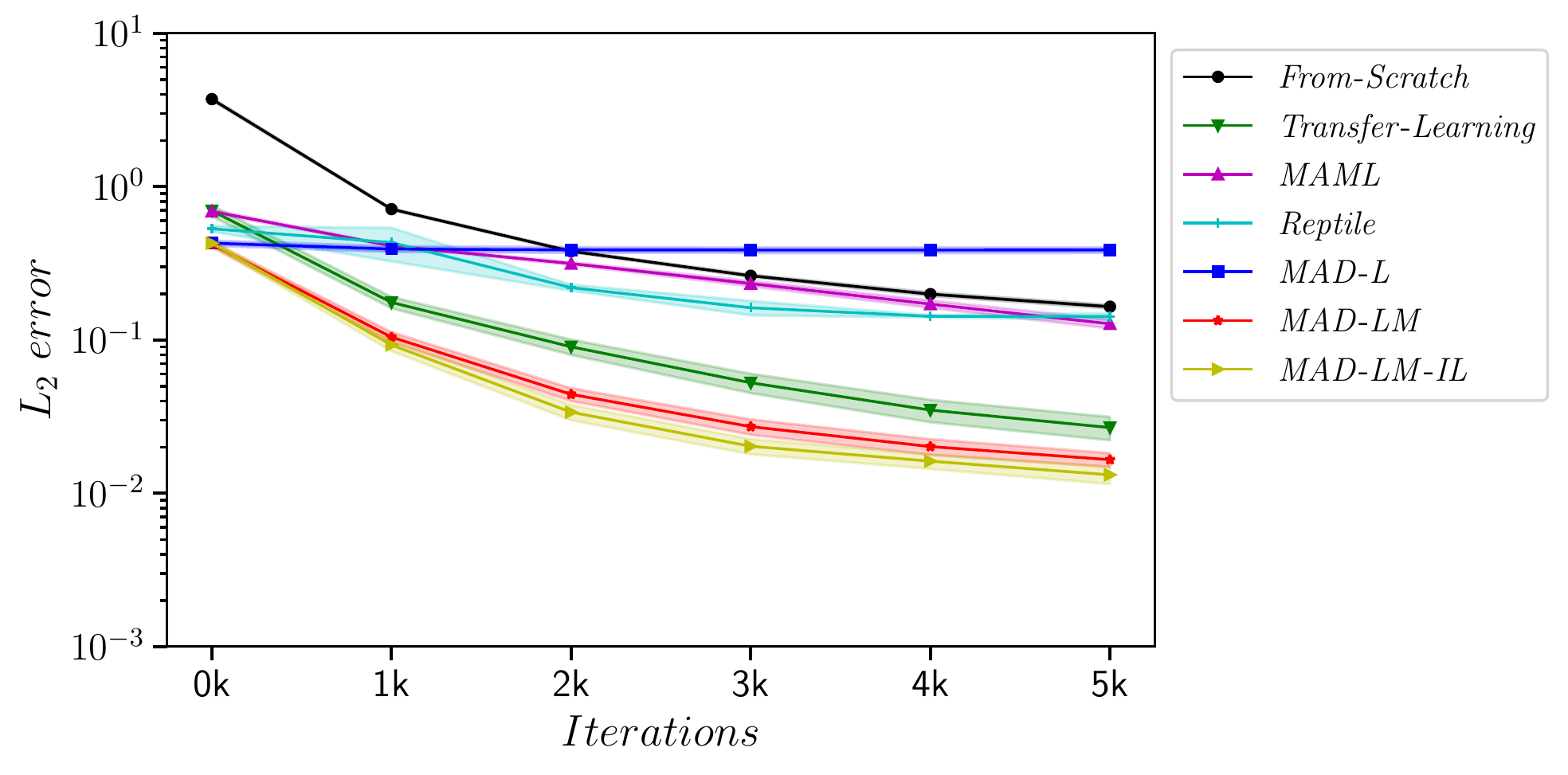}
	\caption{\textbf{Helmholtz's equation:} The convergence of mean $L_2\ error$ with respect to the number of training iterations, zoomed-in to the first 5000 iterations.}
	\label{fig:helmholtz_compare_b}
\end{figure}
\begin{figure}
	\centering
	\includegraphics[width=0.8\linewidth]{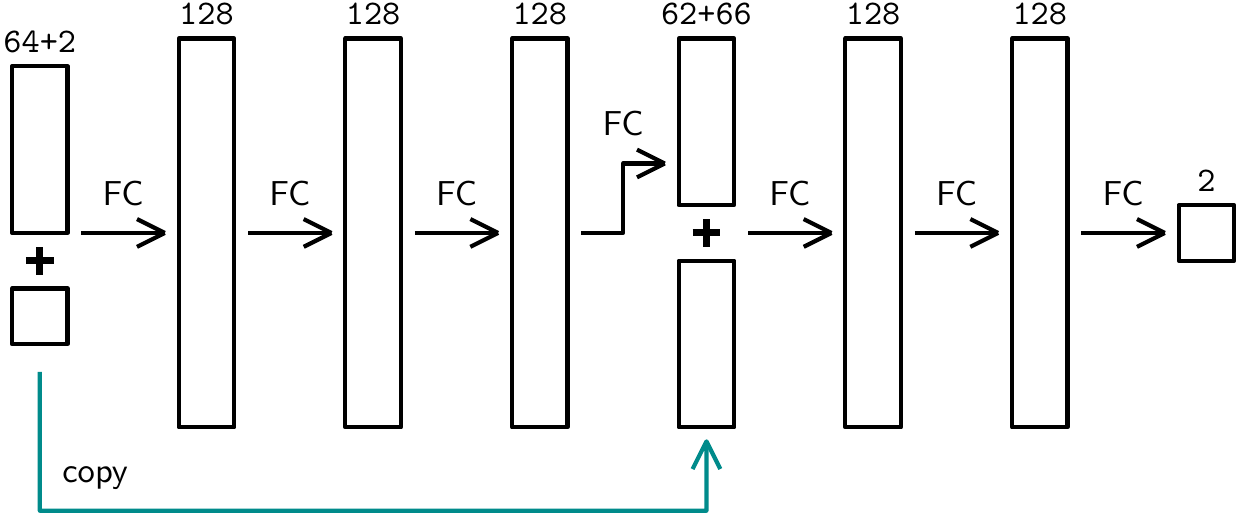}
	\caption{\textbf{Helmholtz's equation:} The alternative network architecture for the MAD method, where boxes represent vectors, arrows represent operations, and the plus signs represent vector concatenation. 
		The dimensions of the vectors are marked on the top of the corresponding boxes. 
		Compared with the original simple network with 7 fully-connected (FC) layers, 
		a copy of the latent vector is now inserted to the fourth hidden layer, 
		and the number of outputs of the fourth fully-connected layer shrinks from $128$ to $62$. 
	}%
	\label{fig:helmholtz_NNarch2}
\end{figure}
\begin{figure}
	\centering
	\includegraphics[width=0.8\linewidth]{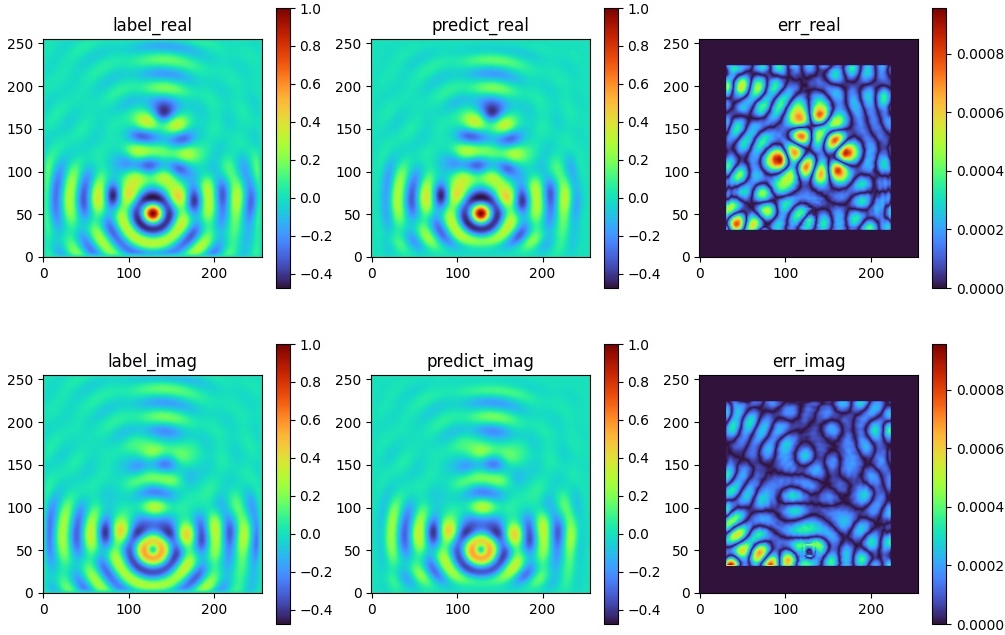}
	\caption{\textbf{Helmholtz's equation:} Solutions predicted by \textit{MAD-LM} (with network architecture shown in Fig.\ref{fig:helmholtz_NNarch2}) compared with the reference solutions, including the real part (top row) and the imaginary part (bottom row). 
		\textbf{Left:} The reference solutions generated by k-Wave.
		\textbf{Middle:} The predicted solutions given by the learned model.
		\textbf{Right:} The absolute error between model predictions and the reference solutions, with the PML region exluded.
		The corresponding sound speed distribution $c(x,y)$ is shown in the bottom-right of Fig.\ref{fig:helmholtz_sos_eg}. 
	}
	\label{fig:helmholtz_MAD_lm_il}
\end{figure}

Fig.\ref{fig:helmholtz_compare_a} shows the mean $L_2\ error$ of all methods as the number of training iterations increases in the fine-tuning stage, 
and Fig.\ref{fig:helmholtz_compare_b} zooms-in to the first $5$k iterations. 
Similar to what is observed in Burgers' equation, \textit{MAD-L} fails to provide an accurate solution, possibly due to the $c$-gap introduced in Sec.\ref{sec:MAD_LM} as before. 
However, \textit{MAD-LM} provides a solution with competitive accuracy compared to the other methods, and achieves the highest convergence speed in the early stage of the fine-tuning process. 
Although \textit{MAML} and \textit{Reptile} show superiority over \textit{From-Scratch} by utilizing the knowledge obtained in the pre-training (or meta-training) stage, 
they both appear to be less effective than the simpler \textit{Transfer-Learning} method. 
We suspect that this is caused by the global oscillatory pattern of the solutions to Helmholtz's equation. 
In order to obtain the solution quickly when given a new PDE parameter, 
such an oscillatory pattern should be fully captured during the pre-training stage. 
Attempting to find a network initialization that is equally good for a set of tasks, both \textit{MAML} and \textit{Reptile} probably fit an average of the solutions in certain sense, in which the global oscillatory pattern is canceled out by each other. 
On the contrary, MAD distinguishes different tasks $\eta_i$ by assigning to them different latent vectors $\pmb{z}_i$ in the pre-training stage, and thereby accumulates knowledge from different PDE parameters successfully without being restricted to an averaged wavefield. 

\bmhead{Effect of a different network architecture}
A simple neural network architecture is used for MAD throughout all the experiments above, where the latent vector $z$ is introduced solely as the extended part of the input layer. 
Inspired by~\cite{park2019deepsdf}, here we consider an alternative architecture as shown in Fig.\ref{fig:helmholtz_NNarch2}, which inserts the latent vector again to one of the middle layers. 
Note that the total number of network parameters decreases compared with the original network, since the output of a middle layer shrinks. 
For distinguishing purposes, results with this alternative network architecture are 
postfixed by ``\textit{IL}'' (standing for ``inserting latent'') in Fig.\ref{fig:helmholtz_compare_a} and~\ref{fig:helmholtz_compare_b}. 
\textit{MAD-L} performs as before with the new architecture, so the corresponding results are not included in the figures. 
However, \textit{MAD-LM} is improved in both precision and convergence speed, and we show the predicted wavefields in Fig.\ref{fig:helmholtz_MAD_lm_il} in comparison with the reference solutions. 
Further improvements of the MAD method could be potentially introduced by utilizing more advanced network architectures~\cite{Mehta2021ModulationPA,Chan2021PiGAN,dupont2022DataFY}. 

\subsection{Summary of Experimental Results}
Achieving fast adaptation is the major focus of this paper, and solutions within a reasonable precision need to be found. 
Indeed, in many control and inverse problems, a higher precision in solving the forward problem (such as parametric PDEs) does not always lead to better results. 
For example, a solution with about $5\%$ relative error is already enough for Maxwell's equations in certain engineering scenarios. 
We are therefore interested in reducing the cost of solving the PDE with a new set of parameters by using only a relatively small number of iterations, in order to obtain an accurate enough solution in practice. 
The advantages of MAD (especially \textit{MAD-LM}) is directly validated in the numerical experiments, as it achieves very fast convergence in the early stage of the training process. 
Some other applications may focus more on the final precision, and is not as sensitive to the training cost. 
In this alternate criterion, the superiority of the MAD method becomes less obvious in our test cases except Maxwell's equations, but its performance is still comparable to other methods. 

\section{Conclusions}\label{sec:conclusion}
\verB{
The proposed method may seem similar to PI-DeepONet~\cite{wang2021learning} at first glance, and we would like to make a brief comment here. 
Both PI-DeepONet and MAD utilize the physics-informed loss Eq.\eqref{eq:MCPIloss} during the offline stage, and incorporate the knowledge obtained to accelerate the online solving process. 
However, the designing logic between them are quite different.  
PI-DeepONet aims to directly learn the solution mapping $G:\mathcal{A}\to\mathcal{U}$, and evaluate $u^\eta=G(\eta)$ via one step of direct network inference. 
MAD seeks to approximate the solution set $\mathcal{K}=G(\mathcal{A})$ (i.e., the image set of the mapping $G:\mathcal{A}\to\mathcal{U}$), and then recover $u^\eta$ by solving a simpler optimization problem in the online stage. 
Compared with finding the exact mapping $G$, approximating the image set $G(\mathcal{A})$ could be a more flexible problem, and hence might be solved up to higher accuracy. 
This might be the intuition behind the higher accuracy of MAD for Burgers' equation, as is shown in Table~\ref{tb:ad_vs_pi_deeponet}. 
Furthermore, PI-DeepONet takes the PDE parameter $\eta$ as network input, bringing inconvenience when $\eta$ is heterogeneous. 
MAD, however, incorporates the information of $\eta$ via the definition of the loss function Eq.\eqref{eq:MCPIloss}, which provides more flexibility in network implementation. 


In sum, this paper proposes a novel reduced order modeling method MAD for solving parametric PDEs. 
}
Based on the idea of meta-learning, the reduced order model obtained is mesh-free in its nature, and the nonlinear trial manifold is constructed in an unsupervised way without requiring any precomputed solution snapshots. 
The decoder width is introduced to quantify the best possible performance of such a trial manifold, serving as a theoretical tool to analyze its effectiveness. 
Fast adaptation for a new set of PDE parameter is attained by searching on this trial manifold, and optionally fine-tuning the shape of the trial manifold simultaneously. 
The advantages of the MAD method is verified by extensive numerical experiments.

\section*{Statements and Declarations}
\bmhead{Competing interests} This work was supported by National Key R\&D Program of China under Grant No. 2021ZD0110400.
Xiang Huang performed this work during an internship at Huawei Technologies Co. Ltd.
The authors have no relevant non-financial interests to disclose.


\begin{appendices}
\section{Nomenclature}
\begin{table}[h]
	\caption{Summary of the commonly used notations in this work. }
	\begin{tabular}{@{}ll@{}}
		\toprule
		Notation & Description\\
		\midrule
		$\Omega\subset\R^d$ & computational domain of the PDEs\\
		$\widetilde{\pmb{x}}\in\Omega$ & the spatiotemporal variable\\
		$\mathcal{U}=\mathcal{U}(\Omega;\R^{d_u})$ & Banach space, with elements of the form $u:\Omega\to\R^{d_u}$\\
		$u:\Omega\to\R^{d_u}$ & solution to the (parametric) PDEs\\
		$\mathcal{A}$ & the space of the PDE parameters\\
		$\eta$ & variable parameter of the PDEs\\
		$\mathcal{L}^{\gamma_1}$ & partial differential operator of the PDEs, with parameter $\gamma_1$\\
		$\mathcal{B}^{\gamma_2}$ & boundary operator of the PDEs, with parameter $\gamma_2$\\
		$G:\mathcal{A}\to\mathcal{U}$ & the solution operator of the parametric PDEs\\
		$Z=\R^n$ & the latent space\\
		$Z_B$ & the closed unit ball of $Z$\\
		$\pmb{z}\in Z$ & latent vector\\
		$d_{n,l}^\text{Deco}(\cdot)$ & the decoder width Eq.\eqref{eq:decWidth}\\
		$D:Z\to\mathcal{U}$ & the decoder mapping\\
		$u_\theta(\cdot,\pmb{z})$ & neural network ansatz of the PDE solution, with latent vector $\pmb{z}$\\
		$L^\eta[\cdot]$ & the physics-informed loss Eq.\eqref{eq:PIloss}\\
		$\hat L^\eta[\cdot]$ & Monte Carlo estimation of $L^\eta[\cdot]$ Eq.\eqref{eq:MCPIloss}\\
		$\theta^*$ & the optimal model weight after pre-training Eq.\eqref{eq:MADtr}\\
		$\pmb{z}_i^*$ & the optimal latent vector for $\eta_i$ after pre-training Eq.\eqref{eq:MADtr}\\
		$\sigma$ & regularization coefficient\\
		\botrule
	\end{tabular}
\end{table}

\end{appendices}

\bibliography{references}


\begin{thebibliography}{53}
\ifx \bisbn   \undefined \def \bisbn  #1{ISBN #1}\fi
\ifx \binits  \undefined \def \binits#1{#1}\fi
\ifx \bauthor  \undefined \def \bauthor#1{#1}\fi
\ifx \batitle  \undefined \def \batitle#1{#1}\fi
\ifx \bjtitle  \undefined \def \bjtitle#1{#1}\fi
\ifx \bvolume  \undefined \def \bvolume#1{\textbf{#1}}\fi
\ifx \byear  \undefined \def \byear#1{#1}\fi
\ifx \bissue  \undefined \def \bissue#1{#1}\fi
\ifx \bfpage  \undefined \def \bfpage#1{#1}\fi
\ifx \blpage  \undefined \def \blpage #1{#1}\fi
\ifx \burl  \undefined \def \burl#1{\textsf{#1}}\fi
\ifx \doiurl  \undefined \def \doiurl#1{\url{https://doi.org/#1}}\fi
\ifx \betal  \undefined \def \betal{\textit{et al.}}\fi
\ifx \binstitute  \undefined \def \binstitute#1{#1}\fi
\ifx \binstitutionaled  \undefined \def \binstitutionaled#1{#1}\fi
\ifx \bctitle  \undefined \def \bctitle#1{#1}\fi
\ifx \beditor  \undefined \def \beditor#1{#1}\fi
\ifx \bpublisher  \undefined \def \bpublisher#1{#1}\fi
\ifx \bbtitle  \undefined \def \bbtitle#1{#1}\fi
\ifx \bedition  \undefined \def \bedition#1{#1}\fi
\ifx \bseriesno  \undefined \def \bseriesno#1{#1}\fi
\ifx \blocation  \undefined \def \blocation#1{#1}\fi
\ifx \bsertitle  \undefined \def \bsertitle#1{#1}\fi
\ifx \bsnm \undefined \def \bsnm#1{#1}\fi
\ifx \bsuffix \undefined \def \bsuffix#1{#1}\fi
\ifx \bparticle \undefined \def \bparticle#1{#1}\fi
\ifx \barticle \undefined \def \barticle#1{#1}\fi
\bibcommenthead
\ifx \bconfdate \undefined \def \bconfdate #1{#1}\fi
\ifx \botherref \undefined \def \botherref #1{#1}\fi
\ifx \url \undefined \def \url#1{\textsf{#1}}\fi
\ifx \bchapter \undefined \def \bchapter#1{#1}\fi
\ifx \bbook \undefined \def \bbook#1{#1}\fi
\ifx \bcomment \undefined \def \bcomment#1{#1}\fi
\ifx \oauthor \undefined \def \oauthor#1{#1}\fi
\ifx \citeauthoryear \undefined \def \citeauthoryear#1{#1}\fi
\ifx \endbibitem  \undefined \def \endbibitem {}\fi
\ifx \bconflocation  \undefined \def \bconflocation#1{#1}\fi
\ifx \arxivurl  \undefined \def \arxivurl#1{\textsf{#1}}\fi
\csname PreBibitemsHook\endcsname

\bibitem{cohen2015approximation}
\begin{barticle}
\bauthor{\bsnm{Cohen}, \binits{A.}},
\bauthor{\bsnm{DeVore}, \binits{R.}}:
\batitle{Approximation of high-dimensional parametric pdes}.
\bjtitle{Acta Numerica}
\bvolume{24},
\bfpage{1}--\blpage{159}
(\byear{2015})
\end{barticle}
\endbibitem

\bibitem{khoo2021solving}
\begin{barticle}
\bauthor{\bsnm{Khoo}, \binits{Y.}},
\bauthor{\bsnm{Lu}, \binits{J.}},
\bauthor{\bsnm{Ying}, \binits{L.}}:
\batitle{Solving parametric pde problems with artificial neural networks}.
\bjtitle{European Journal of Applied Mathematics}
\bvolume{32}(\bissue{3}),
\bfpage{421}--\blpage{435}
(\byear{2021})
\end{barticle}
\endbibitem

\bibitem{zienkiewicz1977finite}
\begin{bbook}
\bauthor{\bsnm{Zienkiewicz}, \binits{O.C.}},
\bauthor{\bsnm{Taylor}, \binits{R.L.}},
\bauthor{\bsnm{Nithiarasu}, \binits{P.}},
\bauthor{\bsnm{Zhu}, \binits{J.}}:
\bbtitle{The Finite Element Method}
vol. \bseriesno{3}.
\bpublisher{McGraw-hill},
\blocation{London}
(\byear{1977})
\end{bbook}
\endbibitem

\bibitem{liszka1980finite}
\begin{barticle}
\bauthor{\bsnm{Liszka}, \binits{T.}},
\bauthor{\bsnm{Orkisz}, \binits{J.}}:
\batitle{The finite difference method at arbitrary irregular grids and its
  application in applied mechanics}.
\bjtitle{Computers \& Structures}
\bvolume{11}(\bissue{1-2}),
\bfpage{83}--\blpage{95}
(\byear{1980})
\end{barticle}
\endbibitem

\bibitem{Quarteroni2015ReducedBM}
\begin{bbook}
\bauthor{\bsnm{Quarteroni}, \binits{A.}},
\bauthor{\bsnm{Manzoni}, \binits{A.}},
\bauthor{\bsnm{Negri}, \binits{F.}}:
\bbtitle{Reduced Basis Methods for Partial Differential Equations: An
  Introduction},
pp. \bfpage{1}--\blpage{263}.
\bpublisher{Springer},
\blocation{Switzerland}
(\byear{2015})
\end{bbook}
\endbibitem

\bibitem{Benner2017ModelRA}
\begin{bbook}
\bauthor{\bsnm{Benner}, \binits{P.}},
\bauthor{\bsnm{Ohlberger}, \binits{M.}},
\bauthor{\bsnm{Cohen}, \binits{A.}},
\bauthor{\bsnm{Willcox}, \binits{K.}}:
\bbtitle{Model Reduction and Approximation}.
\bpublisher{Society for Industrial and Applied Mathematics},
\blocation{Philadelphia, PA}
(\byear{2017})
\end{bbook}
\endbibitem

\bibitem{Greif2019DecayKN}
\begin{barticle}
\bauthor{\bsnm{Greif}, \binits{C.}},
\bauthor{\bsnm{Urban}, \binits{K.}}:
\batitle{Decay of the kolmogorov n-width for wave problems}.
\bjtitle{Applied Mathematics Letters}
\bvolume{96},
\bfpage{216}--\blpage{222}
(\byear{2019}).
\doiurl{10.1016/j.aml.2019.05.013}
\end{barticle}
\endbibitem

\bibitem{Lee2020ModelRD}
\begin{barticle}
\bauthor{\bsnm{Lee}, \binits{K.}},
\bauthor{\bsnm{Carlberg}, \binits{K.T.}}:
\batitle{Model reduction of dynamical systems on nonlinear manifolds using deep
  convolutional autoencoders}.
\bjtitle{Journal of Computational Physics}
\bvolume{404},
\bfpage{108973}
(\byear{2020}).
\doiurl{10.1016/j.jcp.2019.108973}
\end{barticle}
\endbibitem

\bibitem{Fresca2021AComprehensiveDL}
\begin{barticle}
\bauthor{\bsnm{Fresca}, \binits{S.}},
\bauthor{\bsnm{Dede'}, \binits{L.}},
\bauthor{\bsnm{Manzoni}, \binits{A.}}:
\batitle{A comprehensive deep learning-based approach to reduced order modeling
  of nonlinear time-dependent parametrized pdes}.
\bjtitle{Journal of Scientific Computing}
\bvolume{87}(\bissue{2}),
\bfpage{61}
(\byear{2021}).
\doiurl{10.1007/s10915-021-01462-7}
\end{barticle}
\endbibitem

\bibitem{Fresca2022DeepLB}
\begin{barticle}
\bauthor{\bsnm{Fresca}, \binits{S.}},
\bauthor{\bsnm{Gobat}, \binits{G.}},
\bauthor{\bsnm{Fedeli}, \binits{P.}},
\bauthor{\bsnm{Frangi}, \binits{A.}},
\bauthor{\bsnm{Manzoni}, \binits{A.}}:
\batitle{Deep learning-based reduced order models for the real-time simulation
  of the nonlinear dynamics of microstructures}.
\bjtitle{International Journal for Numerical Methods in Engineering}
\bvolume{123}(\bissue{20}),
\bfpage{4749}--\blpage{4777}
(\byear{2022}).
\doiurl{10.1002/nme.7054}
\end{barticle}
\endbibitem

\bibitem{DeVore1989OptimalNA}
\begin{barticle}
\bauthor{\bsnm{DeVore}, \binits{R.A.}},
\bauthor{\bsnm{Howard}, \binits{R.}},
\bauthor{\bsnm{Micchelli}, \binits{C.A.}}:
\batitle{Optimal nonlinear approximation}.
\bjtitle{manuscripta mathematica}
\bvolume{63},
\bfpage{469}--\blpage{478}
(\byear{1989})
\end{barticle}
\endbibitem

\bibitem{Cohen2022OptimalSN}
\begin{barticle}
\bauthor{\bsnm{Cohen}, \binits{A.}},
\bauthor{\bsnm{DeVore}, \binits{R.A.}},
\bauthor{\bsnm{Petrova}, \binits{G.}},
\bauthor{\bsnm{Wojtaszczyk}, \binits{P.}}:
\batitle{Optimal stable nonlinear approximation}.
\bjtitle{Foundations of Computational Mathematics}
\bvolume{22},
\bfpage{607}--\blpage{648}
(\byear{2022})
\end{barticle}
\endbibitem

\bibitem{raissi2018deep}
\begin{barticle}
\bauthor{\bsnm{Raissi}, \binits{M.}}:
\batitle{Deep hidden physics models: Deep learning of nonlinear partial
  differential equations}.
\bjtitle{The Journal of Machine Learning Research}
\bvolume{19}(\bissue{1}),
\bfpage{932}--\blpage{955}
(\byear{2018})
\end{barticle}
\endbibitem

\bibitem{ChiyuMaxJiang2020MeshfreeFlowNetAP}
\begin{botherref}
\oauthor{\bsnm{Jiang}, \binits{C.M.}},
\oauthor{\bsnm{Esmaeilzadeh}, \binits{S.}},
\oauthor{\bsnm{Azizzadenesheli}, \binits{K.}},
\oauthor{\bsnm{Kashinath}, \binits{K.}},
\oauthor{\bsnm{Mustafa}, \binits{M.}},
\oauthor{\bsnm{Tchelepi}, \binits{H.A.}},
\oauthor{\bsnm{Marcus}, \binits{P.}},
\oauthor{\bsnm{Prabhat}},
\oauthor{\bsnm{Anandkumar}, \binits{A.}}:
Meshfreeflownet: A physics-constrained deep continuous space-time
  super-resolution framework.
international conference for high performance computing, networking, storage,
  and analysis
(2020)
\end{botherref}
\endbibitem

\bibitem{DmitriiKochkov2021MachineLA}
\begin{botherref}
\oauthor{\bsnm{Kochkov}, \binits{D.}},
\oauthor{\bsnm{Smith}, \binits{J.}},
\oauthor{\bsnm{Alieva}, \binits{A.}},
\oauthor{\bsnm{Wang}, \binits{Q.}},
\oauthor{\bsnm{Brenner}, \binits{M.}},
\oauthor{\bsnm{Hoyer}, \binits{S.}}:
Machine learning accelerated computational fluid dynamics.
arXiv: Fluid Dynamics
(2021)
\end{botherref}
\endbibitem

\bibitem{sirignano2018dgm}
\begin{barticle}
\bauthor{\bsnm{Sirignano}, \binits{J.}},
\bauthor{\bsnm{Spiliopoulos}, \binits{K.}}:
\batitle{Dgm: A deep learning algorithm for solving partial differential
  equations}.
\bjtitle{Journal of computational physics}
\bvolume{375},
\bfpage{1339}--\blpage{1364}
(\byear{2018})
\end{barticle}
\endbibitem

\bibitem{weinan2018deep}
\begin{botherref}
\oauthor{\bsnm{E}, \binits{W.}},
\oauthor{\bsnm{Yu}, \binits{B.}}:
The deep ritz method: A deep learning-based numerical algorithm for solving
  variational problems.
Communications in Mathematics and Statistics
(2018)
\end{botherref}
\endbibitem

\bibitem{zang2020weak}
\begin{barticle}
\bauthor{\bsnm{Zang}, \binits{Y.}},
\bauthor{\bsnm{Bao}, \binits{G.}},
\bauthor{\bsnm{Ye}, \binits{X.}},
\bauthor{\bsnm{Zhou}, \binits{H.}}:
\batitle{Weak adversarial networks for high-dimensional partial differential
  equations}.
\bjtitle{Journal of Computational Physics}
\bvolume{411},
\bfpage{109409}
(\byear{2020})
\end{barticle}
\endbibitem

\bibitem{long2018pde}
\begin{bchapter}
\bauthor{\bsnm{Long}, \binits{Z.}},
\bauthor{\bsnm{Lu}, \binits{Y.}},
\bauthor{\bsnm{Ma}, \binits{X.}},
\bauthor{\bsnm{Dong}, \binits{B.}}:
\bctitle{Pde-net: Learning pdes from data}.
In: \bbtitle{International Conference on Machine Learning},
pp. \bfpage{3208}--\blpage{3216}
(\byear{2018}).
\bcomment{PMLR}
\end{bchapter}
\endbibitem

\bibitem{long2019pde}
\begin{barticle}
\bauthor{\bsnm{Long}, \binits{Z.}},
\bauthor{\bsnm{Lu}, \binits{Y.}},
\bauthor{\bsnm{Dong}, \binits{B.}}:
\batitle{Pde-net 2.0: Learning pdes from data with a numeric-symbolic hybrid
  deep network}.
\bjtitle{Journal of Computational Physics}
\bvolume{399},
\bfpage{108925}
(\byear{2019})
\end{barticle}
\endbibitem

\bibitem{lu2019deeponet}
\begin{botherref}
\oauthor{\bsnm{Lu}, \binits{L.}},
\oauthor{\bsnm{Jin}, \binits{P.}},
\oauthor{\bsnm{Karniadakis}, \binits{G.E.}}:
Deeponet: Learning nonlinear operators for identifying differential equations
  based on the universal approximation theorem of operators.
arXiv preprint arXiv:1910.03193
(2019)
\end{botherref}
\endbibitem

\bibitem{bhattacharya2020model}
\begin{botherref}
\oauthor{\bsnm{Bhattacharya}, \binits{K.}},
\oauthor{\bsnm{Hosseini}, \binits{B.}},
\oauthor{\bsnm{Kovachki}, \binits{N.B.}},
\oauthor{\bsnm{Stuart}, \binits{A.M.}}:
Model reduction and neural networks for parametric pdes.
arXiv preprint arXiv:2005.03180
(2020)
\end{botherref}
\endbibitem

\bibitem{li2020fourier}
\begin{botherref}
\oauthor{\bsnm{Li}, \binits{Z.}},
\oauthor{\bsnm{Kovachki}, \binits{N.}},
\oauthor{\bsnm{Azizzadenesheli}, \binits{K.}},
\oauthor{\bsnm{Liu}, \binits{B.}},
\oauthor{\bsnm{Bhattacharya}, \binits{K.}},
\oauthor{\bsnm{Stuart}, \binits{A.}},
\oauthor{\bsnm{Anandkumar}, \binits{A.}}:
Fourier neural operator for parametric partial differential equations.
arXiv preprint arXiv:2010.08895
(2020)
\end{botherref}
\endbibitem

\bibitem{cai2021physics}
\begin{barticle}
\bauthor{\bsnm{Cai}, \binits{S.}},
\bauthor{\bsnm{Wang}, \binits{Z.}},
\bauthor{\bsnm{Wang}, \binits{S.}},
\bauthor{\bsnm{Perdikaris}, \binits{P.}},
\bauthor{\bsnm{Karniadakis}, \binits{G.E.}}:
\batitle{Physics-informed neural networks for heat transfer problems}.
\bjtitle{Journal of Heat Transfer}
\bvolume{143}(\bissue{6}),
\bfpage{060801}
(\byear{2021})
\end{barticle}
\endbibitem

\bibitem{wang2021learning}
\begin{botherref}
\oauthor{\bsnm{Wang}, \binits{S.}},
\oauthor{\bsnm{Wang}, \binits{H.}},
\oauthor{\bsnm{Perdikaris}, \binits{P.}}:
Learning the solution operator of parametric partial differential equations
  with physics-informed deeponets.
arXiv preprint arXiv:2103.10974
(2021)
\end{botherref}
\endbibitem

\bibitem{finn2017model}
\begin{bchapter}
\bauthor{\bsnm{Finn}, \binits{C.}},
\bauthor{\bsnm{Abbeel}, \binits{P.}},
\bauthor{\bsnm{Levine}, \binits{S.}}:
\bctitle{Model-agnostic meta-learning for fast adaptation of deep networks}.
In: \bbtitle{International Conference on Machine Learning},
pp. \bfpage{1126}--\blpage{1135}
(\byear{2017}).
\bcomment{PMLR}
\end{bchapter}
\endbibitem

\bibitem{antoniou2019train}
\begin{bchapter}
\bauthor{\bsnm{Antoniou}, \binits{A.}},
\bauthor{\bsnm{Edwards}, \binits{H.}},
\bauthor{\bsnm{Storkey}, \binits{A.}}:
\bctitle{How to train your maml}.
In: \bbtitle{Seventh International Conference on Learning Representations}
(\byear{2019})
\end{bchapter}
\endbibitem

\bibitem{nichol2018reptile}
\begin{barticle}
\bauthor{\bsnm{Nichol}, \binits{A.}},
\bauthor{\bsnm{Schulman}, \binits{J.}}:
\batitle{Reptile: a scalable metalearning algorithm}.
\bjtitle{arXiv preprint arXiv:1803.02999}
\bvolume{2}(\bissue{3}),
\bfpage{4}
(\byear{2018})
\end{barticle}
\endbibitem

\bibitem{yoon2018bayesian}
\begin{bchapter}
\bauthor{\bsnm{Yoon}, \binits{J.}},
\bauthor{\bsnm{Kim}, \binits{T.}},
\bauthor{\bsnm{Dia}, \binits{O.}},
\bauthor{\bsnm{Kim}, \binits{S.}},
\bauthor{\bsnm{Bengio}, \binits{Y.}},
\bauthor{\bsnm{Ahn}, \binits{S.}}:
\bctitle{Bayesian model-agnostic meta-learning}.
In: \bbtitle{Proceedings of the 32nd International Conference on Neural
  Information Processing Systems},
pp. \bfpage{7343}--\blpage{7353}
(\byear{2018})
\end{bchapter}
\endbibitem

\bibitem{chen2022meta}
\begin{botherref}
\oauthor{\bsnm{Chen}, \binits{Y.}},
\oauthor{\bsnm{Dong}, \binits{B.}},
\oauthor{\bsnm{Xu}, \binits{J.}}:
Meta-mgnet: Meta multigrid networks for solving parameterized partial
  differential equations.
Journal of Computational Physics,
110996
(2022)
\end{botherref}
\endbibitem

\bibitem{liu2021novel}
\begin{botherref}
\oauthor{\bsnm{Liu}, \binits{X.}},
\oauthor{\bsnm{Zhang}, \binits{X.}},
\oauthor{\bsnm{Peng}, \binits{W.}},
\oauthor{\bsnm{Zhou}, \binits{W.}},
\oauthor{\bsnm{Yao}, \binits{W.}}:
A novel meta-learning initialization method for physics-informed neural
  networks.
arXiv preprint arXiv:2107.10991
(2021)
\end{botherref}
\endbibitem

\bibitem{park2019deepsdf}
\begin{bchapter}
\bauthor{\bsnm{Park}, \binits{J.J.}},
\bauthor{\bsnm{Florence}, \binits{P.}},
\bauthor{\bsnm{Straub}, \binits{J.}},
\bauthor{\bsnm{Newcombe}, \binits{R.}},
\bauthor{\bsnm{Lovegrove}, \binits{S.}}:
\bctitle{Deepsdf: Learning continuous signed distance functions for shape
  representation}.
In: \bbtitle{Proceedings of the IEEE/CVF Conference on Computer Vision and
  Pattern Recognition},
pp. \bfpage{165}--\blpage{174}
(\byear{2019})
\end{bchapter}
\endbibitem

\bibitem{huang2021metaautodecoder}
\begin{bchapter}
\bauthor{\bsnm{Huang}, \binits{X.}},
\bauthor{\bsnm{Ye}, \binits{Z.}},
\bauthor{\bsnm{Liu}, \binits{H.}},
\bauthor{\bsnm{Shi}, \binits{B.}},
\bauthor{\bsnm{Wang}, \binits{Z.}},
\bauthor{\bsnm{Yang}, \binits{K.}},
\bauthor{\bsnm{Li}, \binits{Y.}},
\bauthor{\bsnm{Weng}, \binits{B.}},
\bauthor{\bsnm{Wang}, \binits{M.}},
\bauthor{\bsnm{Chu}, \binits{H.}},
\bauthor{\bsnm{Zhou}, \binits{J.}},
\bauthor{\bsnm{Yu}, \binits{F.}},
\bauthor{\bsnm{Hua}, \binits{B.}},
\bauthor{\bsnm{Chen}, \binits{L.}},
\bauthor{\bsnm{Dong}, \binits{B.}}:
\bctitle{Meta-auto-decoder for solving parametric partial differential
  equations}.
In: \bbtitle{Advances in Neural Information Processing Systems}
(\byear{2022})
\end{bchapter}
\endbibitem

\bibitem{Cohen2010AnalyticRP}
\begin{botherref}
\oauthor{\bsnm{Cohen}, \binits{A.}},
\oauthor{\bsnm{DeVore}, \binits{R.}},
\oauthor{\bsnm{Schwab}, \binits{C.}}:
Analytic regularity and polynomial approximation of parametric and stochastic
  elliptic pde's.
Analysis and Applications
\textbf{9}
(2010).
\doiurl{10.1142/S0219530511001728}
\end{botherref}
\endbibitem

\bibitem{Tran2017AnalysisQO}
\begin{barticle}
\bauthor{\bsnm{Tran}, \binits{H.}},
\bauthor{\bsnm{Webster}, \binits{C.G.}},
\bauthor{\bsnm{Zhang}, \binits{G.}}:
\batitle{Analysis of quasi-optimal polynomial approximations for parameterized
  pdes with deterministic and stochastic coefficients}.
\bjtitle{Numerische Mathematik}
\bvolume{137}(\bissue{2}),
\bfpage{451}--\blpage{493}
(\byear{2017}).
\doiurl{10.1007/s00211-017-0878-6}
\end{barticle}
\endbibitem

\bibitem{Franco2021DeepLA}
\begin{botherref}
\oauthor{\bsnm{Franco}, \binits{N.R.}},
\oauthor{\bsnm{Manzoni}, \binits{A.}},
\oauthor{\bsnm{Zunino}, \binits{P.}}:
A Deep Learning approach to Reduced Order Modelling of Parameter Dependent
  Partial Differential Equations
(2021)
\end{botherref}
\endbibitem

\bibitem{baydin2018automatic}
\begin{botherref}
\oauthor{\bsnm{Baydin}, \binits{A.G.}},
\oauthor{\bsnm{Pearlmutter}, \binits{B.A.}},
\oauthor{\bsnm{Radul}, \binits{A.A.}},
\oauthor{\bsnm{Siskind}, \binits{J.M.}}:
Automatic differentiation in machine learning: a survey.
Journal of machine learning research
\textbf{18}
(2018)
\end{botherref}
\endbibitem

\bibitem{huang2021solving}
\begin{botherref}
\oauthor{\bsnm{Huang}, \binits{X.}},
\oauthor{\bsnm{Liu}, \binits{H.}},
\oauthor{\bsnm{Shi}, \binits{B.}},
\oauthor{\bsnm{Wang}, \binits{Z.}},
\oauthor{\bsnm{Yang}, \binits{K.}},
\oauthor{\bsnm{Li}, \binits{Y.}},
\oauthor{\bsnm{Weng}, \binits{B.}},
\oauthor{\bsnm{Wang}, \binits{M.}},
\oauthor{\bsnm{Chu}, \binits{H.}},
\oauthor{\bsnm{Zhou}, \binits{J.}}, et al.:
Solving partial differential equations with point source based on
  physics-informed neural networks.
arXiv preprint arXiv:2111.01394
(2021)
\end{botherref}
\endbibitem

\bibitem{sitzmann2020implicit}
\begin{botherref}
\oauthor{\bsnm{Sitzmann}, \binits{V.}},
\oauthor{\bsnm{Martel}, \binits{J.}},
\oauthor{\bsnm{Bergman}, \binits{A.}},
\oauthor{\bsnm{Lindell}, \binits{D.}},
\oauthor{\bsnm{Wetzstein}, \binits{G.}}:
Implicit neural representations with periodic activation functions.
Advances in Neural Information Processing Systems
\textbf{33}
(2020)
\end{botherref}
\endbibitem

\bibitem{kingma2014adam}
\begin{botherref}
\oauthor{\bsnm{Kingma}, \binits{D.P.}},
\oauthor{\bsnm{Ba}, \binits{J.}}:
Adam: A method for stochastic optimization.
arXiv preprint arXiv:1412.6980
(2014)
\end{botherref}
\endbibitem

\bibitem{liu2019advances}
\begin{barticle}
\bauthor{\bsnm{Liu}, \binits{Y.}},
\bauthor{\bsnm{Li}, \binits{J.}},
\bauthor{\bsnm{Sun}, \binits{S.}},
\bauthor{\bsnm{Yu}, \binits{B.}}:
\batitle{Advances in gaussian random field generation: a review}.
\bjtitle{Computational Geosciences}
\bvolume{23}(\bissue{5}),
\bfpage{1011}--\blpage{1047}
(\byear{2019})
\end{barticle}
\endbibitem

\bibitem{lu2021physics}
\begin{botherref}
\oauthor{\bsnm{Lu}, \binits{L.}},
\oauthor{\bsnm{Pestourie}, \binits{R.}},
\oauthor{\bsnm{Yao}, \binits{W.}},
\oauthor{\bsnm{Wang}, \binits{Z.}},
\oauthor{\bsnm{Verdugo}, \binits{F.}},
\oauthor{\bsnm{Johnson}, \binits{S.G.}}:
Physics-informed neural networks with hard constraints for inverse design.
arXiv preprint arXiv:2102.04626
(2021)
\end{botherref}
\endbibitem

\bibitem{lu2022machineLE}
\begin{bchapter}
\bauthor{\bsnm{Lu}, \binits{Y.}},
\bauthor{\bsnm{Chen}, \binits{H.}},
\bauthor{\bsnm{Lu}, \binits{J.}},
\bauthor{\bsnm{Ying}, \binits{L.}},
\bauthor{\bsnm{Blanchet}, \binits{J.}}:
\bctitle{Machine learning for elliptic {PDE}s: Fast rate generalization bound,
  neural scaling law and minimax optimality}.
In: \bbtitle{International Conference on Learning Representations}
(\byear{2022})
\end{bchapter}
\endbibitem

\bibitem{wang2022isLP}
\begin{bchapter}
\bauthor{\bsnm{Wang}, \binits{C.}},
\bauthor{\bsnm{Li}, \binits{S.}},
\bauthor{\bsnm{He}, \binits{D.}},
\bauthor{\bsnm{Wang}, \binits{L.}}:
\bctitle{Is \$l{\textasciicircum}2\$ physics informed loss always suitable for
  training physics informed neural network?}
In: \bbtitle{Advances in Neural Information Processing Systems}
(\byear{2022})
\end{bchapter}
\endbibitem

\bibitem{Psaros2022MetaLP}
\begin{barticle}
\bauthor{\bsnm{Psaros}, \binits{A.F.}},
\bauthor{\bsnm{Kawaguchi}, \binits{K.}},
\bauthor{\bsnm{Karniadakis}, \binits{G.E.}}:
\batitle{Meta-learning pinn loss functions}.
\bjtitle{Journal of Computational Physics}
\bvolume{458},
\bfpage{111121}
(\byear{2022}).
\doiurl{10.1016/j.jcp.2022.111121}
\end{barticle}
\endbibitem

\bibitem{gedney2011introduction}
\begin{barticle}
\bauthor{\bsnm{Gedney}, \binits{S.D.}}:
\batitle{Introduction to the finite-difference time-domain (fdtd) method for
  electromagnetics}.
\bjtitle{Synthesis Lectures on Computational Electromagnetics}
\bvolume{6}(\bissue{1}),
\bfpage{1}--\blpage{250}
(\byear{2011})
\end{barticle}
\endbibitem

\bibitem{schneider2010understanding}
\begin{botherref}
\oauthor{\bsnm{Schneider}, \binits{J.B.}}:
Understanding the finite-difference time-domain method.
School of electrical engineering and computer science Washington State
  University
\textbf{28}
(2010)
\end{botherref}
\endbibitem

\bibitem{LuLu2021ACA}
\begin{botherref}
\oauthor{\bsnm{Lu}, \binits{L.}},
\oauthor{\bsnm{Meng}, \binits{X.}},
\oauthor{\bsnm{Cai}, \binits{S.}},
\oauthor{\bsnm{Mao}, \binits{Z.}},
\oauthor{\bsnm{Goswami}, \binits{S.}},
\oauthor{\bsnm{Zhang}, \binits{Z.}},
\oauthor{\bsnm{Karniadakis}, \binits{G.E.}}:
A comprehensive and fair comparison of two neural operators (with practical
  extensions) based on fair data.
arXiv: Computational Physics
(2021)
\end{botherref}
\endbibitem

\bibitem{Stanziola2021HelmholtzES}
\begin{barticle}
\bauthor{\bsnm{Stanziola}, \binits{A.}},
\bauthor{\bsnm{Arridge}, \binits{S.R.}},
\bauthor{\bsnm{Cox}, \binits{B.T.}},
\bauthor{\bsnm{Treeby}, \binits{B.E.}}:
\batitle{A helmholtz equation solver using unsupervised learning: Application
  to transcranial ultrasound}.
\bjtitle{Journal of Computational Physics}
\bvolume{441},
\bfpage{110430}
(\byear{2021}).
\doiurl{10.1016/j.jcp.2021.110430}
\end{barticle}
\endbibitem

\bibitem{Treeby2010kWaveMT}
\begin{barticle}
\bauthor{\bsnm{Treeby}, \binits{B.E.}},
\bauthor{\bsnm{Cox}, \binits{B.T.}}:
\batitle{k-wave: Matlab toolbox for the simulation and reconstruction of
  photoacoustic wave fields.}
\bjtitle{Journal of biomedical optics}
\bvolume{15 2},
\bfpage{021314}
(\byear{2010})
\end{barticle}
\endbibitem

\bibitem{Mehta2021ModulationPA}
\begin{barticle}
\bauthor{\bsnm{Mehta}, \binits{I.}},
\bauthor{\bsnm{Gharbi}, \binits{M.}},
\bauthor{\bsnm{Barnes}, \binits{C.}},
\bauthor{\bsnm{Shechtman}, \binits{E.}},
\bauthor{\bsnm{Ramamoorthi}, \binits{R.}},
\bauthor{\bsnm{Chandraker}, \binits{M.}}:
\batitle{Modulated periodic activations for generalizable local functional
  representations}.
\bjtitle{2021 IEEE/CVF International Conference on Computer Vision (ICCV)}
(\byear{2021}).
\doiurl{10.1109/iccv48922.2021.01395}
\end{barticle}
\endbibitem

\bibitem{Chan2021PiGAN}
\begin{barticle}
\bauthor{\bsnm{Chan}, \binits{E.R.}},
\bauthor{\bsnm{Monteiro}, \binits{M.}},
\bauthor{\bsnm{Kellnhofer}, \binits{P.}},
\bauthor{\bsnm{Wu}, \binits{J.}},
\bauthor{\bsnm{Wetzstein}, \binits{G.}}:
\batitle{pi-gan: Periodic implicit generative adversarial networks for 3d-aware
  image synthesis}.
\bjtitle{2021 IEEE/CVF Conference on Computer Vision and Pattern Recognition
  (CVPR)}
(\byear{2021}).
\doiurl{10.1109/cvpr46437.2021.00574}
\end{barticle}
\endbibitem

\bibitem{dupont2022DataFY}
\begin{bchapter}
\bauthor{\bsnm{Dupont}, \binits{E.}},
\bauthor{\bsnm{Kim}, \binits{H.}},
\bauthor{\bsnm{Eslami}, \binits{S.M.A.}},
\bauthor{\bsnm{Rezende}, \binits{D.J.}},
\bauthor{\bsnm{Rosenbaum}, \binits{D.}}:
\bctitle{From data to functa: Your data point is a function and you can treat
  it like one}.
In: \bbtitle{Proceedings of the 39th International Conference on Machine
  Learning}
(\byear{2022})
\end{bchapter}
\endbibitem

\end{thebibliography}


\end{document}